\documentclass[11pt]{article}
\usepackage{graphicx} % Required for inserting images
\usepackage{bbm}
\usepackage{bm}
\usepackage[margin=1in]{geometry}
\usepackage{color}
\usepackage{amsfonts}
\usepackage{algorithm}
\usepackage{algorithmic}
\usepackage{latexsym, amssymb, amsmath, amscd, amsthm, amsxtra}
\usepackage{mathtools}
\usepackage{enumerate}
\usepackage[all]{xy}
\usepackage{mathrsfs}
\usepackage{fancyhdr}
\usepackage{listings}
\usepackage{hyperref}
\usepackage{enumitem}
\usepackage{multirow}
\usepackage{tikz}

\def\ER{Erd\H{o}s-R\'enyi\ }

\def\op{\operatorname{op}}
\newcommand{\osc}{\operatorname{osc}}
\newcommand{\dist}{\operatorname{dist}}

\newcommand{\Bin}{\operatorname{Bin}}

\newcommand{\atanh}{\operatorname{arctanh}}
\newcommand{\sech}{\operatorname{sech}}
\newcommand{\e}{\mathrm e}
\newcommand{\dd}{\,\mathrm d}

\def\Pr{\mathrm{P}}

\newcommand{\Ent}{\mathrm{Ent}}
\newcommand{\Var}{\mathrm{Var}}

\def\E{\mathbb{E}}

\def\R{\mathbb{R}}

\newcommand{\Pb}{\mathbb P}

\def\Eb{\mathbb E}

\newcommand{\cB}{\mathcal B}

\newcommand{\cL}{\mathcal L}

\newcommand{\T}{\mathbf T}

\newtheorem{thm}{Theorem}[section]

\newtheorem{proposition}[thm]{Proposition}

\newtheorem{lemma}[thm]{Lemma}
\newtheorem{cor}[thm]{Corollary}
\newtheorem{defn}[thm]{Definition}

\newtheorem{claim}[thm]{Claim}

\newtheorem{remark}[thm]{Remark}

\numberwithin{equation}{section}



\hypersetup{colorlinks=true,linkcolor=blue}

\title{Improved rapid phase ordering for Ising dynamics on random regular graphs}

\author{
Zhangsong Li\thanks{School of Mathematical Sciences, Peking University. Email: \textit{ramblerlzs@pku.edu.cn}} 
\and 
Allan Sly\thanks{Department of Mathematics, Princeton University. Email: \textit{allansly@princeton.edu}}
}

\date{\today}
\begin{document}
\maketitle

\begin{abstract}
We study Glauber dynamics for the Ising model on random $d$-regular graphs with $n$ vertices at low temperature, in the regime $\beta \geq C/d$ where $C,d$ are sufficiently large constants. In this regime, the full mixing time is exponential in $n$, because of the bottleneck between the two dominant phases: configurations with a majority of plus spins and configurations with a majority of minus spins. We show that, for sufficiently large $d$, if the dynamics is initialized with an excess of at least $\epsilon_{C,d} n$ plus spins over minus spins, then it quasi-equilibrates in time $O(\log n)$ to the Gibbs measure conditioned on positive magnetization. In addition, the required bias $\epsilon_{C,d}$ can be chosen so that $\epsilon_{C,d}\to 0$ as $d,C\to\infty$.

Previous bound of \cite{GSS25+} in the Ising and large-$d$ setting required $\beta \gtrsim (\log d)/d$. A key difficulty is that, even from the all-plus initialization, one must control the propagation of negative information in spacetime in a low-temperature model with strong local correlations. We address this by developing a new bootstrap approach to proving subcriticality of disagreement clusters under the grand coupling.
\end{abstract}

\tableofcontents

\section{Introduction}{\label{sec:intro}}

The out-of-equilibrium dynamics of spin systems, including the Ising and Potts models, have been studied extensively across statistical physics, probability, and theoretical computer science. A defining feature of these systems at low temperatures is the emergence of multiple ordered phases. Transitions between these phases require the dynamics to pass through configurations of very small stationary probability, creating a bottleneck and leading to exponentially slow worst-case mixing. This raises a fundamental question: is the bottleneck between phases the only obstruction to rapid relaxation? More concretely, if the dynamics is initialized with a bias toward one of the ground states, does it rapidly approach the Gibbs distribution restricted to the corresponding phase?

This question is closely related to the classical theory of phase-ordering kinetics. A longstanding prediction in the mathematical physics literature is that, when the Ising dynamics is started from a random configuration, its magnetization first moves away from zero and selects one of the two phases, after which equilibration within the selected phase occurs comparatively rapidly; see the influential survey \cite{Bray94}. The same question has an algorithmic interpretation. Markov Chain Monte Carlo (MCMC) algorithms for low-temperature sampling are often initialized either from a product measure or from a high-temperature distribution, as in simulated annealing \cite{KGV83}. Understanding whether such natural initializations rapidly produce an equilibrium sample within one phase is therefore important even when transitions between phases (and hence worst-case mixing) remain prohibitively slow. The Ising and Potts models on nontrivial geometries provide canonical settings in which to place these predictions on a rigorous footing.

A major mathematical difficulty is that most general tools for bounding mixing times are designed to control convergence from the worst possible initialization. Classical methods based on functional inequalities, spectral gaps, and bottleneck estimates \cite{LP17}, as well as more recent approaches based on spectral independence \cite{ALG20} and localization schemes \cite{CE22}, typically rely on suitable correlation-decay estimates and aim to establish global mixing. In a phase-coexistence regime, however, global mixing can be exponentially slow even though relaxation within each individual phase may be rapid. Proving fast convergence from a structured or biased class of initial distributions therefore requires methods that distinguish between inter-phase tunneling and intra-phase equilibration. Indeed, despite more than three decades of work on the mixing of Ising Glauber dynamics, relatively few positive-temperature results make this distinction. Results of this form can be summarized as follows. First, on the complete graph, the dynamics is essentially one-dimensional, being characterized by the birth-and-death chain of the magnetization; it was shown in \cite{DLP09, LJP10} that at all low temperatures, the mixing time of continuous-time Glauber dynamics restricted to positive magnetization is $O(\log n)$. On the infinite tree, \cite{CM06} showed that from i.i.d.\ $\mathsf{Rad}(1-\epsilon)$ initializations, at all sufficiently low temperatures the dynamics converges locally to the plus Gibbs measure. For the integer lattice, it was shown a few years ago in \cite{GS22} that at low temperatures, the quasi-equilibration to the plus phase is fast from the all-plus initialization. This specific initialization gives access to monotonicity tools, since it stochastically dominates the stationary measure, whereas even a very biased product initialization does not. Recently, it was shown in \cite{GS26+} that the Ising Glauber dynamics on $\mathbb Z^2$ initialized from i.i.d.\ spins on each vertex that are $+1$ with sufficiently large probability converges to the plus Gibbs measure quickly throughout the low-temperature regime.

Sparse, locally tree-like graphs such as random regular graphs and \ER graphs form another fundamental setting for graphical models in statistical physics, combinatorics, and theoretical computer science \cite{MM09}. Models on these graph ensembles have long been predicted to display thermodynamic behavior analogous to that of classical lattice systems, particularly in the nature of their phase transitions and the structure of their Gibbs measures \cite{ACTA73, CLR79, Tholess86}. At the same time, the scarcity of short cycles and the recursive structure of local neighborhoods often make locally tree-like models more analytically tractable. For low-temperature Ising and Potts dynamics on random $d$-regular graphs, it was shown in \cite{GS22} that when the inverse-temperature $\beta$ is a sufficiently large constant, the quasi-equilibration to the plus phase is fast from the all-plus initialization. This result was significantly improved in a recent work \cite{GSS25+}, where the same result was established for sufficiently biased initializations provided that the inverse temperature satisfies $\beta \gtrsim \frac{\log d}{d}$. A complementary result of \cite[Theorem~1.7]{GPY25+} shows that, for $\beta<1/(8\sqrt{d-1})$ in our normalization, Ising Glauber dynamics on random $d$-regular graphs converges to the full Gibbs measure in subexponential time from a uniformly random initialization. Thus, this result permits an unbiased initialization, but gives a weaker time bound than the rapid within-phase relaxation considered here.

We finally mention some works that have studied mixing of low-temperature dynamics when forced to stay in narrow subsets of the state space. One such chain is the Kawasaki dynamics where the Ising dynamics is forced to have a fixed magnetization, and updates are made by swapping adjacent $+1$ and $-1$ spins. While sampling from this distribution is hard on general maximum-degree $d$ graphs at low temperatures \cite{KPPY24}, the recent paper \cite{BBD25} showed that for every magnetization, it mixes quickly as long as $\beta \lesssim 1/\sqrt{d}$ (well into the low-temperature regime on random regular graphs). In another direction, \cite{CGG+21} showed that Ising and Potts dynamics are fast to mix at sufficiently low temperatures on random $d$-regular graphs if constrained to remain close to the monochromatic configuration not just in the sense of total spin count, but in the much stronger sense of never allowing connected components of the non-dominant state larger than $c\log n$. These constrained chains can be used to sample from the Ising and Potts distributions quickly, but do not imply anything for the unconstrained Glauber dynamics because with such strong constraints, the standard Glauber chain leaves the subset of the state space more quickly than the restricted chain mixes.

In this paper, we study the phase ordering problem for low (but positive) temperature Ising Glauber dynamics initialized from biased initializations on random $d$-regular graphs. Our main result is that such dynamics quasi-equilibrate in optimal $O(\log n)$ time to the corresponding metastable distribution, i.e., the Ising distribution conditioned on positive magnetization, provided that the degree $d$ is large and the inverse temperature $\beta \gtrsim \frac{1}{d}$. This improvement closes a logarithmic gap and reaches the natural degree-dependent scale at which ordered behavior emerges on locally tree-like graphs. Thus, rapid phase ordering is not confined to the deeply ordered, very-low-temperature regime. More conceptually, our result establishes a sharp separation of time scales: although transitions between the two phases may be exponentially slow, phase selection and equilibration within the selected phase occur in polylogarithmic time from biased initializations. This provides rigorous support for the physical picture that inter-phase tunneling is the principal obstruction to fast relaxation and strengthens the theoretical foundation for low-temperature MCMC algorithms started from natural, non-extremal distributions.

\subsection{Main results}{\label{subsec:main-results}}

We begin by introducing the Ising Glauber dynamics. The Ising model on a graph $G=(V,E)$ with $|V|=n$, at inverse temperature $\beta>0$, is the following probability distribution over assignments $\Omega=\{ -1,+1 \}^n$:
\begin{align}{\label{eq-def-Ising-measure}}
    \pi_{G,\beta}(\sigma) \propto \exp\left( \beta \sum_{u \sim w} \sigma_u \sigma_w \right) \,.
\end{align}
We will simply denote the measure by $\pi$ if $G$ and $\beta$ are clear from the context. The (continuous-time) Glauber dynamics $(X_t)_{t \geq 0}$ for the Ising model on $G$ is the Markov chain that is initialized from $X_0 \in \Omega$, and: 
\begin{itemize}
    \item Assigns each vertex $v \in V$ a rate-$1$ Poisson clock;
    \item If the clock at vertex $v$ rings at time $t$, then it generates $X_t$ from $X_{t^-}$ by resampling 
    \begin{align}{\label{eq-def-Glauber-dynamics}}
        X_t(v) \sim \pi_{G,\beta}\Big( \sigma_v \in \cdot \mid (\sigma_w)_{w \neq v} = (X_{t^-}(w))_{w \neq v} \Big) \,,
    \end{align}
    and leaving $X_t(w)=X_{t^-}(w)$ for all $w\neq v$.
\end{itemize}
We denote the law of this Markov chain by $\Pb_{X_0}(X_t \in \cdot)$.

We consider the Glauber dynamics of the Ising model on graphs $G$ drawn from the uniform distribution on $d$-regular graphs on $n$ vertices, which we denote by $G \sim \mathcal G_d(n)$. The Ising model on $G \sim \mathcal G_d(n)$ undergoes a phase transition at $\beta_c(d):=\tanh^{-1}(\frac{1}{d-1})$. For high temperatures $\beta<\beta_c(d)$, the worst-case mixing time (the time to be at total-variation distance at most $\frac{1}{4}$ to the stationary distribution from a worst-case initialization) is $O(\log n)$, which is optimal \cite{MS13}; whereas for $\beta>\beta_c(d)$, the worst-case mixing time is $\exp(\Theta(n))$, because of a bottleneck between configurations with a plus majority, 
\begin{align}{\label{eq-def-plus-phase}}
    \Omega^+=\{ \sigma:m(\sigma) \geq 0 \} \quad \mbox{where} \quad m(\sigma):= \frac{1}{n} \sum_{v \in V} \sigma(v)  \,,
\end{align}
and those with a minus majority $\Omega^-:=\{\sigma:m(\sigma)\le0\}$ \cite{DM10}. This in particular means that if we initialize from the plus phase, by which we mean the Gibbs distribution conditioned on being in $\Omega^+$, i.e., $\pi_+=\pi(\cdot \mid \Omega^+)$, the time for the Markov chain to hit $\{\sigma:m(\sigma)<0\}$ is exponentially long. Exactly at the critical point $\beta_c(d)$, polynomial mixing time was recently shown by combining \cite{BD24} with \cite{Weitz06} (see \cite[Example~6.19]{BBD24}).
\begin{thm}{\label{MAIN-THM}}
    There exist sufficiently large constants $C_0>0$ such that for all $d \ge 7$ and $C>C_0$, there exist $\epsilon=\epsilon(C,d) \in (0,1)$ and $L=L({C,d})>0$ such that the following holds with probability $1-o(1)$ over $G \sim \mathcal G_d(n)$. For any $x_0 \in \{ \pm 1 \}^V$ such that $m(x_0) \ge \epsilon$, suppose $(X_t)_{t \geq 0}$ is the continuous-time Ising Glauber dynamics on $G$ with inverse temperature $\beta=\frac{C}{d}$ and initialized from $X_0=x_0$. Then
    \begin{equation*}
        \mathsf{TV}\Big( \Pb_{x_0}(X_t \in \cdot), \pi_G(\cdot\mid\Omega^+) \Big) \leq n^{-10} \quad \mbox{for all} \quad L\log n \leq t \leq \exp(n/L) \,.
    \end{equation*}
    In addition, we have $\epsilon(C,d) \to 0$ as $C,d \to \infty$. By symmetry, we have the same bound on the distance to $\pi_G(\cdot\mid\Omega^-)$ if $X_0=x_0$ satisfies $m(x_0) \le -\epsilon$.
\end{thm}
\begin{remark}
    In concurrent work, \cite{PS26+} establish rapid phase ordering for the Ising model on random $d$-regular graphs when $d$ is sufficiently large and $c_0/d \le\beta\le 1/(16\sqrt d)$, for a sufficiently large absolute constant $c_0$. They show that Glauber dynamics initialized from any configuration with $m(X_0)\ge2\beta$ quasi-equilibrates to $\pi_+$ in $O(n\log n)$ single-site updates, corresponding to $O(\log n)$ time in our normalization. Their result therefore also reaches the inverse-temperature scale $1/d$ and gives a sharper explicit initialization condition in the common parameter range. Their proof establishes modified log-Sobolev inequalities for Ising measures conditioned on sufficiently large magnetization, using a sparse spectral width condition, stochastic localization, and magnetization drift. Our proof develops a direct bootstrap analysis of disagreement clusters in spacetime under the grand coupling, controlling their propagation through a typical plus-phase environment.
\end{remark}

\subsection{Heuristics and proof overview}{\label{subsec:proof-overview}}

We first provide some heuristics behind Theorem~\ref{MAIN-THM}. Note that in \cite{GSS25+}, the same result has been proven for all $d \geq 7$ and $C \geq C_0' \log d$ for a sufficiently large absolute constant $C_0'$. Thus, from now on we will assume that
\begin{equation}{\label{eq-assump-C-d-epsilon}}
    C_0<C<C_0\log d, \quad d_0<d, \quad C_0,100\log C<C\epsilon \quad \mbox{for sufficiently large } C_0,d_0 \,.
\end{equation}
The starting point of our approach is to directly bound the coupling time between the chain $X_t=X_t^{x_0}$ initialized from $x_0$ where $m(x_0) \ge \epsilon$, and the chain $X_t^{\pi_+}$ started from the plus measure $\pi_+=\pi_G(\cdot\mid\Omega^+)$. 
\begin{defn}{\label{def-grand-coupling}}
    Independently for each vertex $v \in V(G)$ we assign an infinite sequence of times $(T^v_1,T^v_2,\ldots)$ given by the rings of a rate-$1$ Poisson clock, and an infinite sequence $(U^v_1,U^v_2,\ldots)$ of $\mathsf{Unif}[0,1]$ random variables. For any $x_0 \in \{ -1,+1 \}^n$, the dynamics $(X^{x_0}_t)_{t \geq 0}$ is generated as follows (where $\pi$ denotes the stationary distribution):
    \begin{enumerate}
        \item[(1)] Set $X_0^{x_0}=x_0$. Let $t_1<t_2<\ldots$ be the times in $\cup_k \cup_v \{ T^v_k \}$ in increasing order (these times are distinct almost surely). Also let $t_0=0$.
        \item[(2)] For $j \geq 1$, let $(v,k)$ be the unique pair such that $t_j=T^v_k$. Set $X^{x_0}_t = X^{x_0}_{t_{j-1}}$ for all $t\in [t_{j-1},t_j)$; then set $X^{x_0}_{t_j}(u)=X^{x_0}_{t_{j-1}}(u)$ for all $u \neq v$ and 
        \begin{align*}
            X^{x_0}_{t_j}(v) = 
            \begin{cases}
                +1, & \mbox{if } U_k^v \leq \pi\left( \sigma_v=+1 \mid \sigma_{u}=X^{x_0}_{t_{j-1}}(u), u \neq v \right) \,; \\
                -1, & \mbox{otherwise} \,.
            \end{cases}
        \end{align*}
        By using the same times $(T^v_1,T^v_2,\ldots)$ and uniform random variables $(U^v_1,U^v_2,\ldots)$, and taking $\pi=\pi_{G,\beta}$, the above construction gives a grand coupling of $(X^{x_0}_t)_{t \geq 0}$, the Markov chains initialized from all possible configurations for all $x_0 \in \{ -1,+1 \}^{V(G)}$. It is well known that in the case of the Ising model this coupling is monotone, in the sense that if $x_0 \geq x_0'$, then $X^{x_0}_t \geq X^{x_0'}_t$ for all $t \geq 0$.
    \end{enumerate}
\end{defn}
We will also consider the Glauber dynamics $X_t^{\pi_+}$ restricted to $\Omega^+$. This chain, denoted $\widehat{X}_t^{\pi_+}$, is defined exactly as $X_t^{\pi_+}$ except that if the update in \eqref{eq-def-Glauber-dynamics} would cause $\sigma\not\in\Omega^+$ we leave $\sigma_v$ unchanged. It is easy to check that $\widehat{X}_t^{\pi_+}$ is a Markov chain reversible with respect to $\pi_+$. We first need a result from \cite{GS22}, which states that with probability $1-o(1)$, we have $X_t^{\pi_+}=\widehat{X}_t^{\pi_+}$ for all $t\le e^{n/L}$.
\begin{lemma}[Lemma~3.5 in \cite{GS22}]{\label{lem-equal-to-restricted-dynamics}}
    Define
    \begin{equation}{\label{eq-def-E-0}}
        \mathcal E_{\ref{lem-equal-to-restricted-dynamics}} = \Big\{ \widehat{X}^{\pi_+}_t=X_t^{\pi_+} \mbox{ for all } 0 \leq t \leq \exp\big( n/L \big) \Big\} \,.
    \end{equation}
    We have $\Pb(\mathcal E_{\ref{lem-equal-to-restricted-dynamics}})=1-\exp(-\Omega(n))$ for a sufficiently large constant $L=L(C,d)$.
\end{lemma}
The main focus of our proof is the disagreement set between $(X_t)_{t \geq 0}$ and $(X^{\pi_+}_t)_{t \geq 0}$, defined as
\begin{equation}{\label{eq-def-disagreement-cluster}}
    \mathcal D_t := \Big\{ v \in V: X_t(v) \neq X^{\pi_+}_t(v) \Big\} \,.
\end{equation}
We will also consider the ``bad'' vertices, defined to be the vertices whose neighborhood field is at most $0.5d$,
\begin{equation}{\label{eq-def-bad-cluster}}
    \mathcal B_t := \Big\{ v \in V: \sum_{u \sim v} \widehat{X}^{\pi_+}_t(u) \leq 0.5d \Big\} \,.
\end{equation}
The main approach in our argument is to prove that under the grand coupling of $X_t,X^{\pi_+}_t$, the disagreement set $\mathcal D_t$ will shrink to empty quickly. We first give a heuristic explanation for the expected negative drift of \(|\mathcal D_t|\). Write
\begin{align*}
    H_t(v):=\sum_{u\sim v}X_t(u),\quad H_t^+(v):=\sum_{u\sim v}X_t^{\pi_+}(u), \quad g(h):=\frac{1}{1+e^{-2\beta h}} \,.
\end{align*}
If the clock at \(v\) rings, then, under the common uniform random variable in the grand coupling, the conditional probability that the two updated spins disagree is $a_t(v):=|g(H_t(v))-g(H_t^+(v))|$. Consequently, the instantaneous conditional drift of the number of disagreements is
\begin{align*}
    \lim_{s\downarrow0}\frac{1}{s} \mathbb E\big[ |\mathcal D_{t+s}|-|\mathcal D_t| \mid \mathcal F_t \big] = -|\mathcal D_t|+\sum_{v\in V} a_t(v) \,.
\end{align*}
Indeed, each disagreeing vertex is refreshed at rate one, while \(a_t(v)\) is the probability that a disagreement is present at \(v\) after its refresh. Thus the question is whether the total rate at which disagreements remain or are created is substantially smaller than \(|\mathcal D_t|\). To see why this should hold once \(\mathcal D_t\) is sparse, let \(k_t(v):=\deg_{\mathcal D_t}(v)\). The two local fields differ only through the disagreeing neighbors, and hence $|H_t(v)-H_t^+(v)|\leq 2k_t(v)$. On the event in Lemma~\ref{lem-equal-to-restricted-dynamics}, if \(v\notin B_t\) and \(k_t(v)\leq 0.01d\), then $H_t^+(v)>0.5d$ and $\min\{H_t(v),H_t^+(v)\}>0.48d$. Since
\[
    g'(h)=\frac{2\beta e^{-2\beta h}}{(1+e^{-2\beta h})^2} \leq 2\beta e^{-2\beta h}, \quad h\geq0,
\]
the mean-value theorem suggests
\[
    a_t(v) \leq 2k_t(v)\sup_{h\geq0.48d}g'(h) \leq \frac{4C}{d}e^{-0.96C}k_t(v) \,.
\]
If this estimate held at every vertex, then
\[
    \sum_{v\in V}a_t(v) \leq \frac{4C}{d}e^{-0.96C} \sum_{v\in V}\deg_{\mathcal D_t}(v)= 4Ce^{-0.96C}|\mathcal D_t| \ll |\mathcal D_t|,
\]
which would give a strict negative drift for \(|\mathcal D_t|\). Equivalently, one may think of each disagreement as dying at rate close to one and producing new disagreements at total rate \(O(Ce^{-0.96C})<1\).

This calculation also indicates the main difficulties. The exponential suppression may fail at a bad vertex, where the plus-phase local field is not sufficiently positive, and it may fail when a vertex has an atypically large number of disagreeing neighbors. Intuitively, we will use the fact that the chain is near the plus phase to make the first type of exceptional vertex exponentially sparse, and we will use the expansion and local tree structure of a typical random regular graph to prevent a small disagreement set from being locally dense around many vertices. 

We now explain our proof strategy. We first show that a random regular graph $G \sim \mathcal G_d(n)$ belongs to a family $\mathcal G$ with suitable expansion properties and local tree structure with probability $1-o(1)$; this is formalized in Section~\ref{sec:structure-random-regular-graph}. Fix such a ``typical'' graph $G\in\mathcal G$. A preliminary burn-in argument uses the initial bias and expansion to bring the disagreement set into a sparse regime. More precisely, for a sufficiently large constant $t_0=t_0(C,d)$,
\begin{equation}\label{eq-initial-sparsity}
    \Pb\left(
        \sup_{0\leq t\leq n}|\mathcal B_t|\leq e^{-0.004Cd}n,\quad
        \sup_{t_0\leq t\leq n}|\mathcal D_t|\leq 4e^{-0.002Cd}n
    \right)\geq 1-n^{-12} \,.
\end{equation}
The precise statement and proof of this preliminary estimate are given in Lemma~\ref{lem-bound-initial-steps} in Section~\ref{sec:initial-analysis}.

The main challenge, and the focus of our proof, is to show that these sparse disagreements disappear in $O(\log n)$ time. We analyze their evolution in spacetime: a disagreement generates a branching-like causal cluster, but away from the rare bad regions each possible transmission across an edge pays the additional $e^{-\Omega(C)}$ factor above. Thus the effective offspring mean is subcritical, and the rare bad regions should not be numerous or connected enough to compensate for this decay. Making this picture rigorous requires controlling the interaction between the evolving disagreements and the bad regions. Our main argument is a bootstrap analysis that yields the following quantitative contraction property.
\begin{defn}{\label{def-admissible-configuration}}
    For all $A \subset V$ and $\sigma \in \{ -1,+1 \}^n$, define 
    \begin{equation}{\label{eq-def-sigma-A}}
        \sigma_A(u) = 
        \begin{cases}
            \sigma(u), & u \not\in A \,; \\
            -\sigma(u), & u \in A \,.
        \end{cases}
    \end{equation}
    Let $(X^{\sigma}_t)_{t \ge 0}$ and $(X^{\sigma_A}_t)_{t\ge 0}$ be the Ising Glauber dynamics under the grand coupling starting from $\sigma$ and $\sigma_A$, respectively. We say a configuration $\sigma \in \{ -1,+1 \}^{V}$ is admissible if for all $A \subset V$ with $|A|=pn$ with $0<p\leq e^{-0.02d}$, we have
    \begin{align}{\label{eq-def-shrink-time}}
        \mathbb E\Big[ \#\big\{ v \in V: X_t^{\sigma}(v) \neq X^{\sigma_A}_t(v) \big\} \Big] \leq 4e^{-t/(2C^3)}pn, \quad t = \frac{ \log(1/p) }{ 10^3C\log(d)^{10} } \,.
    \end{align}
    Define $\mathscr A$ to be the set of admissible configurations.
\end{defn}
\begin{lemma}{\label{lem-configuration-are-admissible}}
    Fix $G \in \mathcal G$. Define 
    \begin{align}{\label{eq-def-E-diamond-diamond}}
        \mathcal E_{\ref{lem-configuration-are-admissible}}= \Big\{ \widehat{X}_t^{\pi_+} \in \mathscr A \mbox{ for all } 0 \leq t \leq n \Big\} \,.
    \end{align}
    We have $\Pb(\mathcal E_{\ref{lem-configuration-are-admissible}}) \ge 1-n^{-12}$ for all $G \in \mathcal G$.
\end{lemma}
The proof of Lemma~\ref{lem-configuration-are-admissible} contains the main bootstrap argument and is given in Section~\ref{sec:bootstrap-disagree-cluster}. We now deduce Theorem~\ref{MAIN-THM} from the preliminary estimate \eqref{eq-initial-sparsity} and Lemmas~\ref{lem-equal-to-restricted-dynamics} and \ref{lem-configuration-are-admissible}.
\begin{proof}[Proof of Theorem~\ref{MAIN-THM}]
    Fix a graph $G \in\mathcal G$ and condition throughout on this graph.  We write $\mathbb P_G$ and $\mathbb E_G$ for probability and expectation over the initial configurations and the graphical construction. Recall \eqref{eq-def-disagreement-cluster} and let $Z_t=|\mathcal D_t|$. Let $\mathcal E:=\mathcal E_{\ref{lem-equal-to-restricted-dynamics}} \cap \mathcal E_{\ref{lem-bound-initial-steps}} \cap \mathcal E_{\ref{lem-configuration-are-admissible}}$, where $\mathcal E_{\ref{lem-bound-initial-steps}}$ denotes the event in \eqref{eq-initial-sparsity}. By \eqref{eq-initial-sparsity} and Lemmas~\ref{lem-equal-to-restricted-dynamics} and \ref{lem-configuration-are-admissible}, for all sufficiently large $n$, 
    \[
        \mathbb P_G(\mathcal E^c) \leq \exp(-\Omega(n))+2n^{-12} \leq 3n^{-12} \,.
    \]
    On $\mathcal E$ we have
    \begin{align*}
        & X_t^{\pi_+} = \widehat X_t^{\pi_+} \text{ for every } 0\leq t\leq \exp(n/L),  \quad \widehat X_t^{\pi_+} \in \mathscr A \text{ for every } 0\leq t\leq n \,; \\
        & Z_t\leq p_\star n \text{ for every }t_0\leq t\leq n, \quad p_\star:=4e^{-0.002Cd} \overset{\eqref{eq-assump-C-d-epsilon}}{\le} e^{-0.02d} \,.
    \end{align*}
    Let $(\mathcal F_t)_{t\geq0}$ be the completed filtration generated by the initial states and the marked graphical construction. We now define a sequence of stopping times. Set $\tau_0:=t_0$. Having defined $\tau_i$, write $Z_i:=Z_{\tau_i}$ and define
    \[
        q_i:=\min\left\{ p_\star,\frac{\max\{Z_i,1\}}{n} \right\}, \quad \tau_{i+1}:=\tau_i+s(q_i), \quad s(p) := \frac{ \log(1/p) }{ 10^3C(\log d)^{10} } \,.
    \]
    Note that $\frac1n\leq q_i\leq p_\star$. The variable $q_i$ is $\mathcal F_{\tau_i}$-measurable and takes values in a finite set, so $\tau_i$ are indeed stopping times. Moreover, by \eqref{eq-assump-C-d-epsilon}, we have the deterministic increment bounds
    \begin{equation}{\label{eq:increment-bound}}
        4C^3 \log 4 \le s(p_\star) \leq \tau_{i+1}-\tau_i \leq \frac{ \log n }{ 10^3C(\log d)^{10} } \,.
    \end{equation}
    For each $i\geq0$, introduce the adapted event
    \[
        H_i:=\Big\{ \tau_i\leq n,\ Z_i\leq p_\star n,\ X^{\pi_+}_{\tau_i} = \widehat X^{\pi_+}_{\tau_i},\ \widehat X^{\pi_+}_{\tau_i} \in \mathscr A \Big\} \,.
    \]
    In addition, put $F_i:=\cap_{j=0}^i H_j$ and let $F_{-1}$ be the whole space. On the event $H_i$, suppose first that $Z_i>0$, and set $A_i:=\mathcal D_{\tau_i}$, $\sigma_i:=X^{\pi_+}_{\tau_i}$. The configurations $X_{\tau_i}$ and $X^{\pi_+}_{\tau_i}$ disagree precisely on $A_i$, so $X_{\tau_i}=(\sigma_i)_{A_i}$. Furthermore,
    \[
        \frac{|A_i|}{n}=\frac{Z_i}{n}\leq p_\star\leq e^{-0.02d}, \quad  \sigma_i=\widehat X^{\pi_+}_{\tau_i} \in \mathscr A \,.
    \]
    By the strong Markov property at $\tau_i$ and Definition~\ref{def-admissible-configuration}, 
    \[
        \mathbb E_G\left[ Z_{i+1} \mid \mathcal F_{\tau_i} \right] \leq 4e^{-s(Z_i/n)/(2C^3)}Z_i \leq e^{-(\tau_{i+1}-\tau_i)/(4C^3)} Z_i \text{ on }H_i \,.
    \]
    If $Z_i=0$, the two ordinary chains have already coupled and remain equal forever under the grand coupling, so the same inequality holds trivially. Thus, in all cases,
    \[
        \mathbf 1_{H_i} \cdot \mathbb E_G\left[ Z_{i+1} \mid \mathcal F_{\tau_i} \right] \leq \mathbf 1_{H_i} \cdot e^{-(\tau_{i+1}-\tau_i)/(4C^3)} Z_i \,.
    \]
    For $i \ge 0$, define $M_i:= \mathbf 1_{F_{i-1}} \cdot e^{(\tau_i-t_0)/(4C^3)}Z_i$. Since $F_i=F_{i-1}\cap H_i$ is $\mathcal F_{\tau_i}$-measurable,
    \begin{align}
        \mathbb E_G\left[ M_{i+1} \mid \mathcal F_{\tau_i} \right] &=e^{(\tau_{i+1}-t_0)/(4C^3)} \mathbf 1_{F_i} \mathbb E_G\left[ Z_{i+1} \mid \mathcal F_{\tau_i} \right] \nonumber \\
        &\leq e^{(\tau_i-t_0)/(4C^3)} Z_i \mathbf 1_{F_i} \leq e^{(\tau_i-t_0)/(4C^3)} Z_i \mathbf 1_{F_{i-1}} =M_i \,.  \nonumber
    \end{align}
    Thus, $(M_i)_{i\geq0}$ is a nonnegative supermartingale with respect to the discrete filtration $(\mathcal F_{\tau_i})_{i\geq0}$. In particular, if $\nu$ is any bounded stopping index for this discrete filtration, then optional stopping gives
    \begin{equation}\label{eq:M-martingale}
        \mathbb E_G\big[ M_\nu \big] \leq \mathbb E_G\big[ M_0 \big] = \mathbb E_G\big[ Z_{t_0} \big] \leq n \,.
    \end{equation}
    Let $L=L(C,d)$ to be sufficiently large. We now argue that two chains $X_t$ and $X^{\pi_+}_t$ are coupled with high probability after time $L\log n$. To this end, let
    \begin{equation}
        \nu:=\inf\left\{ i\geq0: \tau_i \geq L\log n - \frac{ \log n }{ 10^3C(\log d)^{10} } \right\} \,.
    \end{equation}
    Because every increment is at least $s(p_\star)>0$, the stopping index $\nu$ is bounded by a deterministic multiple of $\log n$. Moreover, by the upper increment bound \eqref{eq:increment-bound},
    \[
        L\log n - \frac{ \log n }{ 10^3C(\log d)^{10} } \leq \tau_\nu \leq L\log n \,.
    \]
    On the event $\mathcal E$, we have $\tau_j\leq\tau_\nu\leq L\log n<n$ and $H_j$ holds for $j \le \nu$. Hence $\mathcal E\subseteq F_{\nu-1}$. Also, once $\mathcal D_t$ becomes empty it remains empty. Since $\tau_\nu\leq L\log n$, we have $\mathcal D_{L\log n} \neq \emptyset$ implies that $Z_\nu\geq1$. It follows that, on the event $\{ \mathcal D_{L\log n} \neq\emptyset \} \cap\mathcal E$,
    \[
        M_\nu = \mathbf 1_{F_{\nu-1}} \cdot e^{(\tau_\nu-t_0)/(4C^3)} Z_\nu  \geq e^{(L-1)\log n/(4C^3)} \,.
    \]
    Therefore, by Markov's inequality and optional stopping,
    \begin{align*}
        \mathbb P_G\big( \mathcal D_{L\log n} \neq \emptyset, \mathcal E \big) \leq e^{-(L-1)\log n/(4C^3)}\mathbb E_G[M_\nu] \leq n e^{-(L-1)\log n/(4C^3)} \,.
    \end{align*}
    Thus, by choosing $L=L(C,d)$ sufficiently large,
    \[
        \mathbb P_G\big( \mathcal D_{L\log n} \neq \emptyset, \mathcal E \big) \leq n^{-13}
    \]
    for all sufficiently large $n$. Finally, fix any deterministic time satisfying $L\log n\leq t\leq\exp(n/L)$. On $\mathcal E\cap\{ \mathcal D_{L\log n}=\emptyset \}$, the ordinary chains $(X_t)$ and $(X_t^{\pi_+})$ remain coupled for all times after $L\log n$, while Lemma~\ref{lem-equal-to-restricted-dynamics} gives $X_s^{\pi_+}=\widehat X^{\pi_+}_s$ for every $s\leq\exp(n/L)$. Hence
    \[
        X_t=X_t^{\pi_+}=\widehat X_t^{\pi_+} \quad \text{on} \quad \mathcal E\cap \{\mathcal D_{L\log n}=\emptyset \} \,.
    \]
    Since the restricted chain is started from its reversible stationary distribution, $\widehat X_t^{\pi_+} \sim \pi_+$ for every $t\geq0$.  The coupling inequality therefore yields
    \begin{align*}
        \mathsf{TV}\left(\mathcal L_G(X_t), \pi_G(\,\cdot\mid\Omega^+) \right) &\leq \mathbb P_G(X_t\neq\widehat X^{\pi_+}_t) \\
        &\leq \mathbb P_G(\mathcal E^c) +\mathbb P_G( \mathcal D_{L\log n} \neq\emptyset,\mathcal E) \leq 3n^{-12}+n^{-13} \leq n^{-10} \,.
    \end{align*}
    The same coupling event works simultaneously for every deterministic $t$ in the displayed interval, so the bound is uniform over $L\log n\leq t\leq\exp(n/L)$.
\end{proof}

\subsection{Notation}{\label{subsec:notation}}

We record in this subsection some notation conventions. 
For any graph $H$, denote the excess of $H$ by $\tau(H)=|E(H)|-|V(H)|$. Given $u \in V(H)$, denote by $\mathcal{N}_H(u)$ the set of neighbors of $u$ in $H$, and denote by $\mathsf{deg}_H(u)$ the degree of $u$ in $H$. For two vertices $u,v\in V(H)$, we define $\mathsf{dist}_{H}(u,v)$ to be the graph distance. For a graph $H=(V,E)$ and a subset $A \subset V$, define $H_A =(A,E_A)$ to be the vertex-induced subgraph of $H$ in $A$, where $E_A = \{ (u,v) \in E: u,v \in A \}$. Also, define $H_{\setminus A} = (V \setminus A,E_{\setminus A})$ to be the subgraph of $H$ obtained by deleting all edges with endpoints intersecting $A$. Define $\mathcal B(v,H,r)$ to be the subgraph of $H$ induced by $\{ u \in V(H):\mathsf{dist}_H(u,v) \leq r \}$, and define $\mathcal B(A,H,r)=\cup_{v \in A} \mathcal B(v,H,r)$ for $A \subset V$. Denote by $\mathsf L(H)$ the set of leaves of $H$, and by $\mathsf I(H)$ the set of isolated vertices of $H$.  

Throughout, a path means a self-avoiding path. For a path $P$, denote by $\mathsf{EndP}(P)$ its endpoints. An oriented walk $P=(x_0,\ldots,x_m)$ has length $|P|=m$ and is non-backtracking if $x_{i-1}\ne x_{i+1}$ for $1\le i<m$. For a walk, $V(P)$ and $E(P)$ are the vertices and distinct edges of its trace, and $\mathsf{EndP}(P)=\{x_0,x_m\}$. In graph unions, we identify a walk with its trace, with repeated edges counted only once. Walks are unrestricted unless explicitly required to be non-backtracking.

Define $\rho(T)$ to be the root of a tree $T$. For a rooted tree $(T,o)$ and $u\in V(T)$, we define $\mathsf{dep}_{T}(u)$ to be the depth of $u$ in $T$, and define $\mathsf{dep}(T)$ to be the depth of $T$. For $u,v \in V(T)$, denote by $u \rightarrow v$ (or equivalently, $v \leftarrow u$) if $v$ is a child of $u$, and denote by $u \hookrightarrow v$ (or equivalently, $v \hookleftarrow u$) if $v$ is a descendant of $u$. For an unlabeled rooted tree $(T,o)$ and $u \in V(T)$, denote by $\mathsf{Des}_{T}(u)$ the descendant tree of $u$ in $T$, that is, the rooted tree with root $u$ and with vertices given by the descendants of $u$ in the unlabeled rooted tree $T$. 

For any two positive sequences $\{a_n\}$ and $\{b_n\}$, we write interchangeably $a_n=O(b_n)$, $b_n=\Omega(a_n)$, $a_n\lesssim b_n$ and $b_n\gtrsim a_n$ if there exists a positive absolute constant $c$ such that $a_n/b_n\leq c$ holds for all $n$. We write $a_n=o(b_n)$, $b_n=\omega(a_n)$, $a_n\ll b_n$, and $b_n\gg a_n$ if $a_n/b_n\to 0$ as $n\to\infty$.

\section{Preliminary results}{\label{sec:prelim}}

In this section, we will introduce preliminary lemmas that serve as useful ingredients for proving our main theorems.

\subsection{Combinatorial estimates}{\label{subsec:comb-enumeration}}

We collect in this subsection some estimates for the enumeration of trees and forests in bounded degree graphs.
\begin{lemma}{\label{lem-enu-tree-on-regular-graph}}
    For any graph $G$ with degree bounded by $d$ and $v \in V(G)$, we have
    \begin{align*}
        \#\Big\{ T \subset G: T \mbox{ is a tree}, v \in V(T), |V(T)|=k \Big\} \leq (4d)^{k-1} \,.
    \end{align*}
\end{lemma}
\begin{proof}
    Root the tree $T$ at $v$. Every rooted tree with $k$ vertices can be viewed as a rooted plane tree after ordering the children. The number of rooted plane trees with $k$ vertices is the Catalan number $\mathsf{Cat}(k-1)\leq 4^{k-1}$. For a fixed rooted plane tree shape, the number of ways to embed it into $G$ with root mapped to $v$ is at most $d^{k-1}$, because each non-root vertex chooses one neighbor of its parent, and each parent has at most $d$ neighbors. Hence we have the desired bound.
\end{proof}

\begin{lemma}{\label{lem-enu-forest-on-regular-graph}}
    For any graph $G$ with degree bounded by $d$ and $A \subset V(G)$, we have
    \begin{align*}
        \#\Big\{ F \subset G: F=T_1 \sqcup \ldots \sqcup T_r, T_i \mbox{ are trees}, V(T_i) \cap A \neq \emptyset, |V(F)|=k \Big\} \leq \binom{|A|}{r} \binom{2k-r}{k-r} d^{k-r} \,.
    \end{align*}
\end{lemma}
\begin{proof}
    Root each tree $T_i$ at $v_i \in A$. Every rooted tree with $k$ vertices can be viewed as a rooted plane tree after ordering the children. Note that the number of choices of $(v_1,\ldots,v_r)$ is bounded by $\binom{|A|}{r}$. In addition, the number of plane forests with exactly $r$ rooted ordered trees and total $k-r$ edges is given by
    \begin{align*}
        \frac{r}{2k-r} \binom{2k-r}{k-r} \,.
    \end{align*}
    For a fixed rooted plane forest shape, the number of ways to embed it into $G$ with roots mapped to $v_1,\ldots,v_r$ is at most $d^{k-r}$, because each non-root vertex chooses one neighbor of its parent, and each parent has at most $d$ neighbors. Hence we have the desired bound.
\end{proof}

\begin{lemma}{\label{lem:-enu-sequence-longrightarrow-forest}}
    For any tree $T$ with degree bounded by $d$ and $A,A' \subset V(T)$, denote by $A \Rightarrow A'$ if every connected component of $A'$ intersects $A$. For any $S \subset V(T)$ with $|S| \le 2$, we have
    \begin{align*}
        \#\Big\{ A_0,\ldots,A_N \subset V(T): |A_i|=m_i>0, S \Rightarrow A_0 \Rightarrow \ldots \Rightarrow A_N \Big\} \le 3\times (128d)^{m_0+\ldots+m_N-1} \,.
    \end{align*}
\end{lemma}
\begin{proof}
    When $|S|=1$, suppose that $S=\{ v \}$. By Lemma~\ref{lem-enu-tree-on-regular-graph}, the number of connected choices for \(A_0\) of size \(m_0\) containing \(v\) is at most $(4d)^{m_0-1}$. Given \(A_j\) with \(|A_j|=m_j\), Lemma~\ref{lem-enu-forest-on-regular-graph}, summed over the possible number of components, gives the crude bound
    \begin{align}{\label{eq:enu-A-j+1-given-A-j}}
        \#\big\{ A_{j+1}: A_j\Rightarrow A_{j+1},\ |A_{j+1}|=m_{j+1} \big\} \le 2^{m_j} (8d)^{m_{j+1}} \,.
    \end{align}
    The estimate then follows from a simple multiplication principle. When $|S|=2$, suppose that $S=\{ u,v \}$. By Lemma~\ref{lem-enu-forest-on-regular-graph}, the number of choices for \(A_0\) of size \(m_0\) such that $\{ u,v \} \Rightarrow A_0$ is bounded by
    \begin{align*}
        2(4d)^{m_0-1} + \mathbf1_{\{m_0\ge2\}}(8d)^{m_0-1} \le 3 \times (8d)^{m_0-1} \,.  
    \end{align*}
    Combining with \eqref{eq:enu-A-j+1-given-A-j} and the multiplication principle again leads to the desired estimate.
\end{proof}

\subsection{Random-cluster representation}{\label{subsec:random-cluster}}

We first formally define the random-cluster representation of the Ising model, and the Edwards-Sokal coupling of the random-cluster model to the Ising model. For more details, and a discussion of the random-cluster model for non-integer $q$, we refer the reader to \cite{Grimmett04}.

For $p \in [0,1]$ and $q>0$, the random-cluster model with parameters $p,q$ on a (finite) graph $G=(V,E)$ is the distribution over edge-subsets $\omega \subset E$, naturally identified with $\omega\in\{ 0,1 \}^E$ via $\omega_e=1$ if and only if $e \in \omega$, given by
\begin{align}
    \pi^{\mathsf{RC}}_G(\omega) = \frac{ 1 }{ \mathcal Z^{\mathsf{RC}}_{G,p,q} } p^{|\omega|} (1-p)^{|E|-|\omega|} q^{\mathsf{Comp}(\omega)} \,,  \label{eq-def-Random-Cluster}
\end{align}
where $\mathsf{Comp}(\omega)$ is the number of connected components (clusters) in the subgraph $(V,\omega)$. We will write $\pi^{\mathsf{RC}}_G$ as $\pi^{\mathsf{RC}}$ if the graph $G$ is clear from the context. When $\omega_e=1$, we say that $e$ is open, and when $\omega_e=0$, we say it is closed. If $x,y\in V$ are in the same connected component of the subgraph $(V,\omega)$, we write $x \overset{\omega}{\longleftrightarrow} y$.

For $q=2$, there is a canonical way to couple the random-cluster model to the Ising model so that the random-cluster model encodes the correlation structure of the Ising model. (An analogous coupling holds for all integer $q$, relating the random-cluster model to the Potts model.) The $q=2$ Edwards-Sokal coupling $\pi_G^{\mathsf{ES}}$ is the probability distribution over spin-edge pairs $(\omega,\sigma) \in \{ 0,1 \}^{E} \times \{ -1,1 \}^{V}$ given by:
\begin{enumerate}
    \item[(1)] Sampling a random-cluster configuration $\omega \sim \pi^{\mathsf{RC}}_G$ with parameters $p=1-e^{-2\beta}$ and $q=2$;
    \item[(2)] Independently assigning each connected component $\mathcal C$ of $\omega$ a random variable $\eta_{\mathcal C}$ which equals $+1$ or $-1$, each with probability $\frac{1}{2}$, and setting $\sigma_v=\eta_{\mathcal C}$ for every $v \in \mathcal C$.
\end{enumerate}
The marginal of this coupling on $\sigma$ then gives a sample from $\pi_G$ at inverse temperature $\beta$. We now provide some basic properties of the random-cluster model on $G \sim \mathcal G_d(n)$.
\begin{lemma}\label{lem-domination-Ber-FK-perco}
    Let $G=(V,E)$ be any finite graph. Put
    \begin{align*}
        \rho=\frac{p}{p+2(1-p)} \geq \frac{p}{2} \,.
    \end{align*}
    Then $\pi_{G}^{\mathsf{RC}}$ stochastically dominates independent bond percolation on $G$ with probability $\rho$. Equivalently, there is a coupling of configurations $\xi \in \{ 0,1 \}^{E}$ and $\omega\in \{ 0,1 \}^E$ such that
    \begin{align*}
        \xi\sim\mathsf{Ber}_G(\rho), \quad \omega\sim\pi_{G}^{\mathsf{RC}}, \quad \xi\le \omega \quad\text{ almost surely} \,.
    \end{align*}
\end{lemma}
\begin{proof}
    Fix an edge $e=(x,y)$ and condition on the states of all edges other than $e$. If $x$ and $y$ are already connected without using $e$, then opening $e$ does not change the number of clusters, and hence
    \begin{align*}
        \pi_{G}^{\mathsf{RC}}(\omega_e=1\mid \omega_{E\setminus\{e\}})=p \,.
    \end{align*}
    If $x$ and $y$ are not connected without using $e$, then opening $e$ decreases the number of clusters by one. Thus the conditional probability that $e$ is open is
    \begin{align*}
        \frac{2^{k-1}p}{2^{k-1}p+2^k(1-p)} =\frac{p}{p+2(1-p)}=\rho \,.
    \end{align*}
    Therefore every single-edge conditional opening probability is at least $\rho$. The usual finite-volume domination criterion implies stochastic domination of product Bernoulli-$\rho$ percolation.
\end{proof}

\begin{lemma}\label{lem-FK-giant-size}
    Given $G$, let $\omega \sim \pi_G^{\mathsf{RC}}$ with $p=1-e^{-2\beta}$ and $q=2$. Write $\mathcal C_1(\omega)$ for the largest open component of $\omega$, breaking ties by a deterministic rule depending only on $\omega$. If \eqref{eq-assump-C-d-epsilon} holds, then with probability $1-o(1)$ over $G \sim \mathcal G_d(n)$, we have
    \begin{align*}
        \pi_G^{\mathsf{RC}}\Big( |\mathcal C_1(\omega)|\le (1-\eta)n \Big) \le e^{-0.01C\eta n} \quad \mbox{for all} \quad \frac{100}{C} \leq \eta \leq 0.01 \,.
    \end{align*}
\end{lemma}
\begin{proof}
    We first work under the annealed setting where $\Pb$ is the joint law of $G \sim \mathcal G_d(n)$ and $\omega \sim \pi^{\mathsf{RC}}_G$. Using Lemma~\ref{lem-domination-Ber-FK-perco}, it suffices to show that if $\xi \sim \mathsf{Ber}_G(\rho)$, then for all sufficiently large $n$ with $nd$ even we have
    \begin{align*}
        \Pb\Big( |\mathcal C_1(\xi)|\le (1-\eta)n \Big) \le e^{-0.015C\eta n} \quad \mbox{for all} \quad \frac{100}{C} \leq \eta \leq 0.01 \,.
    \end{align*}
    We will work under the configuration model. Fix a deterministic set $S\subseteq[n]$ with $s=|S|\le n/2$, and let $a=ds$ be the number of half-edges belonging to $S$. Let $E_S$ be the event that, after percolation, there is no retained edge between $S$ and $S^c$. Expose the random pairing of half-edges sequentially. During the first $m=\lfloor a/4 \rfloor$ such steps there is always an unmatched half-edge in $S$ as each step removes at most two half-edges of $S$. Also, before step $j<m$, at most $j$ half-edges outside $S$ have been removed. Since $s\le n/2$, initially there are at least $dn/2$ outside half-edges, so the number of unmatched outside half-edges is at least
    \begin{align*}
        \frac{dn}{2}-m \ge \frac{dn}{2}-\frac{ds}{4} \ge \frac{3dn}{8}.
    \end{align*}
    The total number of unmatched half-edges is at most $dn$. Hence, at each of these steps, conditionally on the past, the probability that the revealed partner lies outside $S$ and that the resulting edge is retained is at least $\frac{3}{8}\rho$. If $E_S$ occurs, then none of these $m$ trials produces a retained crossing edge. Consequently,
    \begin{align*}
        \Pb(E_S) \le \left(1-\frac{3\rho}{8}\right)^m \le \exp\left(-\frac{3\rho m}{8}\right) \le \exp\left(-\frac{\rho ds}{20}\right) \leq \exp\left( -\frac{Cs}{40} \right) \,,
    \end{align*}
    where the last inequality follows from $\rho\geq \frac{p}{2}$ and $p=1-e^{-\frac{2C}{d}} \geq \frac{C}{d}$ from \eqref{eq-assump-C-d-epsilon}. Let $F$ be the event that there exists a set $S\subseteq[n]$ with $\eta n\le |S|\le n/2$ and no retained edge between $S$ and $S^c$. By the union bound and the estimate above,
    \begin{align*}
        \Pb(F) &\le \sum_{s=\lceil\eta n\rceil}^{\lfloor n/2\rfloor} \binom ns \exp\left(-\frac{Cs}{40}\right) \le \sum_{s=\lceil\eta n\rceil}^{\lfloor n/2\rfloor} \exp\left(s\log\frac{en}{s}-\frac{Cs}{40}\right) \\
        &\le \sum_{s=\lceil\eta n\rceil}^{\infty} \exp\left(-0.02Cs\right) \le e^{-0.019C\eta n} \,.
    \end{align*}
    It remains to observe that $F^c$ implies the existence of an open component of size at least $(1-\eta)n$. Indeed, if every open component had size at most $n/2$, then either some component itself had size in $[\eta n,n/2]$, or else all components had size less than $\eta n$ and one could form a union of components whose size lies in $[\eta n,2\eta n]\subseteq[\eta n,n/2]$. In both cases this union has no retained edge to its complement, contradicting $F^c$. Hence the largest open component has size greater than $n/2$. If its complement had size at least $\eta n$, that complement would itself be a union of open components of size in $[\eta n,n/2]$ with no retained edge to its complement, again contradicting $F^c$. Therefore $|\mathcal C_1(\xi)|>(1-\eta)n$ on $F^c$, proving the estimate in the annealed setting. The original lemma then follows from a Markov inequality.
\end{proof}

\begin{lemma}{\label{lem-TV-dist-giant-is-plus}}
    Consider the Edwards-Sokal coupling $(\omega,\sigma) \in \{ 0,1 \}^{E} \times \{ -1,1 \}^{V}$ such that $\sigma \sim \pi_{G,\beta}$ and $\omega \sim \pi_{G,p,q}^{\mathsf{RC}}$ where $p=1-e^{-2\beta}$ and $q=2$. Write $\mathcal C_1(\omega)$ for the largest open component of $\omega$, breaking ties by a deterministic rule depending only on $\omega$. Define
    \begin{align*}
        \pi_{G,+}^{\mathsf{giant}}:=\mathcal L(\sigma\mid A) \quad \mbox{where} \quad A=\big\{ \sigma_v=+ \mbox{ for all }v \in \mathcal C_1(\omega) \big\} \,.
    \end{align*}
    Suppose that $\beta=\frac{C}{d}$ where \eqref{eq-assump-C-d-epsilon} holds. Then, with probability at least $1-o(1)$ over $G \sim \mathcal G_d(n)$,
    \begin{align*}
        \mathsf{TV}\Big( \pi_{G,+}, \pi_{G,+}^{\mathsf{giant}} \Big) \le 4e^{-10^{-4}Cn} \,.
    \end{align*}
\end{lemma}
\begin{proof}
    We begin with a simple probabilistic observation. Let $X$ be a random variable on a probability space, and let $A,B,D$ be events such that $A\triangle B\subseteq D$ and $\Pb(A),\Pb(B)\ge 1/4$. We claim that
    \begin{align}
        \mathsf{TV}\Big( \mathcal L(X\mid A),\mathcal L(X\mid B) \Big) \le 4\Pb(D) \,.  \label{eq-TV-dist-conditional-measure}
    \end{align}
    Indeed, by the data-processing inequality, we have 
    \begin{align*}
        \mathsf{TV}\Big( \mathcal L(X\mid A),\mathcal L(X\mid B) \Big) &\leq \mathsf{TV}\Big( \Pb(\cdot \mid A),\Pb(\cdot \mid B) \Big) \\
        &=\frac{1}{2} \mathbb E_{\Pb}\Bigg[ \Big| \frac{\mathbf 1_A}{\Pb(A)}- \frac{\mathbf 1_B}{\Pb(B)} \Big| \Bigg] \leq \frac{\Pb(A\triangle B)}{\Pb(A)\wedge\Pb(B)} \le 4\Pb(D) \,.
   \end{align*}
   Now, note that 
   \begin{align*}
       &\pi_{G,+}^{\mathsf{giant}}:=\mathcal L(\sigma\mid A) \quad \mbox{where} \quad A=\Big\{ \sigma_v=+1 \mbox{ for all }v \in \mathcal C_1(\omega) \Big\} \,; \\
       &\pi_{G,+} := \mathcal L(\sigma\mid B) \quad \mbox{where} \quad B=\Big\{ \sum_{v \in V} \sigma_v\geq 0 \Big\} \,.
   \end{align*}
   Let $D=\{ |\mathcal C_1(\omega)|\le 0.99n \}$. It is clear that $A \triangle B \subseteq D$. In addition, by Lemma~\ref{lem-FK-giant-size}, we have $\Pb(D) \leq \exp(-10^{-4}Cn)$. Also, it is clear that $\Pb(A) =\frac{1}{2}$. Thus, $\Pb(B) \geq \Pb(A)-\Pb(D) \geq \frac{1}{4}$. We then have the desired lemma from \eqref{eq-TV-dist-conditional-measure}.
\end{proof}

\subsection{Log-Sobolev estimates}{\label{subsec:log-Sobolev-estimate}}

For a probability measure $\mu$ on $\{-1,+1\}^{V}$ and a nonnegative function $F$, write 
\[
  \Ent_\mu(F)=\mu[F\log F]-\mu[F]\log\mu[F] \,.
\]
For $v\in V$, let $\Var_v(f)$ and $\Ent_v(F)$ denote conditional variance and conditional entropy with respect to the spin at $v$, conditioned on all other spins. Define the log-Sobolev constant $c_{\mathsf{LS}}=c_{\mathsf{LS}}(\mu)$ by
\begin{equation}{\label{eq:cls-def}}
    c_{\mathsf{LS}}(\mu):=\inf_{f\not\equiv \mathrm{const}} \left\{ \frac{\mathcal E_\mu(f)}{\Ent_\mu(f^2)} \right\}, \quad \mathcal E_\mu(f):=\sum_{v\in V}\mu\left[\Var_v(f)\right] \,.
\end{equation}
In addition, define
\begin{equation}{\label{eq:tilde-cls-def}}
    \widetilde{c}_{\mathsf{LS}}(\mu):=\inf_{f\not\equiv \mathrm{const}} \left\{ \frac{\widetilde{\mathcal E}_\mu(f)}{\Ent_\mu(f^2)} \right\}, \quad \widetilde{\mathcal E}_\mu(f):=\sum_{v\in V}\mu\left[\Ent_v(f^2)\right] \,.
\end{equation}
We will collect in this subsection some facts about $c_{\mathsf{LS}}$ and $\widetilde{c}_{\mathsf{LS}}$. We will first need the following result for one-site log-Sobolev estimates in \cite{DSC96, GZ02}.
\begin{lemma}{\label{lem:two-point-LSI}}
    For $s\in\{ \pm 1 \}$, let $\nu_h(s)=\frac{e^{hs}}{2\cosh h}$. Then, for every $g:\{-1,+1\}\to\R$,
    \begin{equation}{\label{eq:two-point-exact}}
        \Ent_{\nu_h}(g^2) \le \ell(h)\,\Var_{\nu_h}(g), \quad \ell(h):=
        \begin{cases}
            \frac{2|h|}{\tanh |h|},  &h\ne0 \,; \\
            2,  &h=0 \,.
        \end{cases}
    \end{equation}
    Moreover, $\ell(h)\le 2(1+|h|)$.
\end{lemma}
\begin{proof}
    Let $p=\nu_h(+1)$ and $q=\nu_h(-1)$. Using \cite[Theorem~A.2]{DSC96}, we have
    \begin{align*}
        \Ent_{(p,q)}(g^2) \le \frac{\log p-\log q}{p-q}\,\Var_{(p,q)}(g) \,,
    \end{align*}
    with the continuous value $2$ at $p=q=1/2$. Since $\log(p/q)=2h$ and $p-q=\tanh h$, this yields \eqref{eq:two-point-exact}. To see that $\ell(h) \le 2(1+|h|)$, it suffices to note that $x/\tanh x\le1+x$.
\end{proof}

Approximate tensorization of entropy provides a route from one-site log-Sobolev estimates to a log-Sobolev inequality for the full measure; see \cite{Marton15} for results of this type on discrete product spaces. Here we use the following Ising-specific approximate tensorization theorem of \cite[Theorem~4.1]{Lee24}.
\begin{lemma}\label{thm:lee}
    Let
    \begin{align*}
        \nu_{\mathcal J,h}(x) \propto \exp\!\left(\frac12\langle x,\mathcal Jx\rangle+\langle h,x\rangle\right), \quad x\in\{-1,+1\}^n \,,
    \end{align*}
    where $\mathcal J$ is symmetric and positive semidefinite. If $\|\mathcal J\|_{\op}<1$, then for every nonnegative $F$,
    \begin{equation*}
        \Ent_{\nu_{\mathcal J,h}}(F) \le \frac{1}{1-\|\mathcal J\|_{\op}} \sum_{i=1}^n \nu_{\mathcal J,h}\!\left[\Ent_i(F)\right] \,.
    \end{equation*}
\end{lemma}

We also use two standard comparison principles. They are stated in the forms needed below.
\begin{lemma}{\label{lem-bounded-perturbation}}
    Let $\mu$ and $\nu$ be probability measures on the same finite spin space. Suppose $e^{-A}\le \frac{\mathsf d\nu}{\mathsf d\mu}\le e^A$ pointwise. Then $\widetilde{c}_{\mathsf{LS}}(\nu)\ge e^{-4A} \widetilde{c}_{\mathsf{LS}}(\mu)$. Consequently, if $\nu\propto e^V\mu$ and $\mathsf{osc}(V):=\max V-\min V\le A$, then $\widetilde{c}_{\mathsf{LS}}(\nu)\ge e^{-4A} \widetilde{c}_{\mathsf{LS}}(\mu)$.
\end{lemma}
\begin{proof}
    Let $h=\mathsf d\nu/\mathsf d\mu$. For $a>0$ put $\Phi_a(t)= t\log(t/a)-t+a\ge0$. The variational identity
    \begin{align*}
        \Ent_\mu(f)=\inf_{a>0}\mathbb E_\mu[\Phi_a(f)]
    \end{align*}
    gives
    \begin{equation}{\label{eq-comparison-Ent(f)}}
        \Ent_\nu(f)\le e^A\Ent_\mu(f) \,. 
    \end{equation}
    Fix a site $x$. For any exterior configuration, the conditional density of $\nu_x(\cdot\mid\eta)$ relative to $\mu_x(\cdot\mid\eta)$ is $h$ divided by its conditional $\mu_x$-mean. Since both $h$ and this conditional mean lie in $[e^{-A},e^A]$, this conditional density is at least $e^{-2A}$. Hence
    \begin{align*}
        \Ent_x^\nu(f)(\eta)\ge e^{-2A}\Ent_x^\mu(f)(\eta) \,.
    \end{align*}
    The exterior marginal density of $\nu$ relative to the exterior marginal of $\mu$ is again at least $e^{-A}$. Therefore
    \begin{equation}{\label{eq-comparison-Entx(f)}}
        \mathbb E_\nu[\Ent_x^\nu(f)] \ge e^{-3A} \mathbb E_\mu[\Ent_x^\mu(f)] \,.
    \end{equation}
    Combining \eqref{eq-comparison-Ent(f)}, \eqref{eq-comparison-Entx(f)}, and the log-Sobolev inequality \eqref{eq:tilde-cls-def} for $\mu$ yields
    \begin{align*}
        \Ent_\nu(f) \le \frac{e^{4A}}{\widetilde{c}_{\mathsf{LS}}(\mu)} \sum_x \mathbb E_\nu[\Ent_x^\nu(f)]
    \end{align*}
    which proves the first claim. If $\nu\propto e^V\mu$ and $\mathsf{osc}(V)\le A$, then the normalized density has ratio at most $e^{A}$, and hence lies between $e^{-A}$ and $e^{A}$, so the second claim follows by directly applying the first claim.
\end{proof}

\begin{lemma}\label{lem-marginal}
    Let $\mu$ be a probability measure on $\{\pm1\}^{\Lambda}$ and let $S\subset\Lambda$. Let $\mu_S$ be the marginal of $\mu$ on $S$. Then
    \begin{align*}
        \widetilde{c}_{\mathsf{LS}}(\mu_S)\ge \widetilde{c}_{\mathsf{LS}}(\mu) \,.
    \end{align*}
\end{lemma}
\begin{proof}
    Let $f$ be a nonnegative function depending only on the spins in $S$. Then $\Ent_{\mu_S}(f)=\Ent_\mu(f)$. For $x\notin S$, the single-site entropy $\Ent_x^\mu(f)$ is zero. For $x\in S$, the conditional law of $\sigma_x$ under $\mu_S$, given $\sigma_{S\setminus\{x\}}$, is a mixture of the finer conditional laws under $\mu$, where the spins in $\Lambda\setminus S$ are also conditioned on. Since $t\mapsto t\log t$ is convex, we have 
    \begin{align*}
        \mathbb E_{\mu_S}[\Ent_x^{\mu_S}(f)] \ge \mathbb E_\mu[\Ent_x^\mu(f)] \,.
    \end{align*}
    Applying the log-Sobolev inequality \eqref{eq:tilde-cls-def} for $\mu$ to functions depending only on $S$ gives the result.
\end{proof}

\begin{remark}\label{rmk-tv-ratio}
    We shall repeatedly use the following elementary consequence of a bounded density ratio. If two probability measures have densities whose pointwise ratio has supremum divided by infimum at most $R$, then their total variation distance is at most $(R-1)/(R+1)$ \cite{AZ97}. In particular, if the log-density oscillation is at most $4\beta$, then the total variation distance is at most
    \[
        \frac{e^{4\beta}-1}{e^{4\beta}+1}=\tanh(2\beta) \,.
    \]
\end{remark}

\subsection{Log-Sobolev estimates for the Ising model on trees}{\label{subsec:Log-Sobolev-on-tree}}

We first introduce some notation for trees. 
\begin{defn}{\label{def:specific-tree}}
    \underline{The rooted tree $\mathbb T_{k,\ell}$.} Let $\mathbb T_{0,\ell}$ denote the rooted finite $d$-regular tree of depth $\ell$: the root $o$ has $d$ children, every non-root non-leaf vertex has $d-1$ children, and the leaves are exactly the vertices at graph distance $\ell$ from $o$. For $0\le k\le d$, define $\mathbb T_{k,\ell}$ as follows. Start from $\mathbb T_{0,\ell}$ and delete $k$ of the $d$ depth-one descendant subtrees of the root. Thus the root of $\mathbb T_{k,\ell}$ has exactly $d-k$ genuine children, and below those children the tree is still $d$-regular away from the root. Let $\mathbb T_{k,\ell}^{\circ}$ be obtained from $\mathbb T_{k,\ell}$ by attaching $k$ new leaves $a_1,\dots,a_k$ to the root. The boundary of $\mathbb T^\circ_{k,\ell}$ consists of the ordinary depth-$\ell$ leaves of $\mathbb T_{k,\ell}$ together with the attached leaves $a_1,\dots,a_k$. Define $\phi_{k,\ell}^{+}$ to be the Ising measure on the non-boundary vertices of $\mathbb T^\circ_{k,\ell}$ with boundary condition
    \begin{align*}
        \tau_x=+1 \text{ on ordinary depth-$\ell$ leaves}, \quad \tau_{a_j}=-1 \text{ for } 1\le j\le k \,.
    \end{align*}
    Thus the root has $d-k$ genuine child subtrees and $k$ fixed minus leaves adjacent to it. For $k=0$, $\phi_{0,\ell}^{+}$ is the usual all-plus measure on $\mathbb T_{0,\ell}$.

    \underline{The spine tree $\mathbb S_{\gamma,\ell}$.} Let $\mathbb P_\gamma$ with $\gamma\le 2\ell$ be a path with vertices $u_1,u_2,\dots,u_\gamma$. For each $1\le i\le \gamma$, set
    \begin{align*}
        m_i:= \operatorname{dist}_{\mathbb P_\gamma}(u_i,\{u_1,u_\gamma\})= \min\{i-1,\gamma-i\}, \quad h_i := \ell-m_i \,.
    \end{align*}
    The tree $\mathbb S_{\gamma,\ell}$ is obtained from this path by attaching to each spine vertex $u_i$ exactly $d-2$ pairwise disjoint off-spine descendant trees of depth $h_i$: every off-spine vertex at distance strictly less than $h_i$ from $u_i$ has $d-1$ children farther from the spine, and the vertices at distance $h_i$ from $u_i$ are leaves. The boundary of $\mathbb S_{\gamma,\ell}$ is the set of all these off-spine leaves. Let $\mathbb S_{\gamma,\ell}^{\circ}$ be the tree obtained from $\mathbb S_{\gamma,\ell}$ by adding one extra leaf adjacent to $u_1$ and one extra leaf adjacent to $u_\gamma$. These two extra leaves are also treated as boundary vertices and are fixed to $-1$. We denote the corresponding plus-boundary (where the ordinary depth-$\ell$ leaves are $+1$ and the extra two leaves are $-1$) measure by $\mu_{\gamma,\ell}^{+}$. For $\gamma=1$, both extra leaves are attached to the single spine vertex.
\end{defn}

The primary goal is to show that $\phi_{k,\ell}^+$ have positive log-Sobolev constant. We first need a quantitative version of \cite[Theorem~1.2]{MSW04} regarding the log-Sobolev constant for $\phi_{0,\ell}^+$, as incorporated in the following two lemmas.
\begin{lemma}{\label{lem:compare-c-LS-widetilde-c-LS}}
    We have $c_{\mathsf{LS}}(\phi_{k,\ell}^+) \ge (4C)^{-1} \widetilde{c}_{\mathsf{LS}}(\phi_{k,\ell}^+)$.
\end{lemma}
\begin{proof}
    Fix $\nu=\phi_{k,\ell}^+$ on $\mathbb T=\mathbb T_{k,\ell}$, and denote by $\widetilde{c}=\widetilde{c}_{\mathsf{LS}}(\nu)$. Using \eqref{eq:tilde-cls-def}, we have 
    \[
        \sum_{v \in V} \nu[\Ent_v(f^2)] \ge \widetilde{c} \cdot \Ent_\nu(f^2) \,.
    \]
    Condition on all spins except $\sigma(v)$, the conditional law of $\sigma(v)$ is the two-point measure in Lemma~\ref{lem:two-point-LSI} $\nu_h(s)$ with $h=h_v=\beta\sum_{u\sim v}\sigma_u$, where fixed plus-boundary spins are included in the sum. Since every vertex of $\mathbb T$ has degree bounded by $d$, we have $|h_v|\le \beta d=C$. By Lemma~\ref{lem:two-point-LSI} and \eqref{eq-assump-C-d-epsilon},
    \[
        \Ent_v(f^2) \le 2(1+|h_v|)\Var_v(f) \le 2(1+C)\Var_v(f) \le 4C\Var_v(f) \,.
    \]
    Thus, we get that
    \begin{align*}
        \mathcal E_\nu(f) \ge (4C)^{-1} \sum_{v \in V} \nu[\Ent_v(f^2)] \ge (4C)^{-1} \widetilde{c} \cdot \Ent_\nu(f^2) \,,
    \end{align*}
    leading to the desired result.
\end{proof}

\begin{lemma}{\label{lem-MSW}}
    Suppose the inverse-temperature $\beta=\frac{C}{d}$ where $C,d$ satisfy \eqref{eq-assump-C-d-epsilon}. The log-Sobolev constant satisfies $\widetilde{c}_{\mathsf{LS}}(\phi^+_{0,\ell}) \geq \frac{1}{2}$ and $c_{\mathsf{LS}}(\phi^+_{0,\ell}) \geq (8C)^{-1}$.
\end{lemma}
\begin{proof}
    Let $A_{\ell}$ be the adjacency matrix of the free vertices in $V(\mathbb T_{0,\ell})$. We first argue that $\|A_{\ell}\|_{\op}\le 2\sqrt{d-1}$. In fact, orient every edge away from the root and put $\Lambda=(d-1)^{-1/2}$. For an oriented edge $x\to y$ and any $z\in \R^{V(\mathbb T_{0,\ell})}$, we have $2|z_xz_y|\le \Lambda z_x^2+\Lambda^{-1}z_y^2$. A non-root vertex receives coefficient $\Lambda^{-1}$ from the edge to its parent and coefficient at most $(d-1)\Lambda$ from its children, for a total of at most $2\sqrt{d-1}$. The root receives coefficient at most $d\Lambda=\frac{d}{\sqrt{d-1}}\le 2\sqrt{d-1}$. Therefore
    \[
        |\langle z,A_{\ell}z\rangle| \le 2\sqrt{d-1}\,\|z\|_2^2 \,,
    \]
    and thus $\|A_{\ell}\|_{\op}\le 2\sqrt{d-1}$. Define $\mathcal J:=\beta(A_\ell+2\sqrt{d-1}\, \mathbb I) \succeq 0$. Adding the diagonal term does not change the Ising measure because $\beta\sqrt{d-1} \langle\sigma,\sigma\rangle =\beta\sqrt{d-1} |V(\mathbb T_{0,\ell})|$ is constant. Thus $\phi_{0,\ell}^+$ is exactly an Ising measure $\nu_{\mathcal J,h}$ for a suitable external field $h$ encoding the fixed boundary spins. Moreover,
    \begin{equation}{\label{eq:Jnorm}}
        \|\mathcal J\|_{\op} \le \beta\bigl(\|A_\ell\|_{\op}+2\sqrt{d-1}\, \bigr) \le 4\beta\sqrt{d-1} =\frac{4C\sqrt{d-1}}{d} \le \frac{4C}{\sqrt d} \overset{\eqref{eq-assump-C-d-epsilon}}{\le} \frac{1}{2} \,.
    \end{equation}
    Applying Lemma~\ref{thm:lee} with $F=f^2$ yields
    \begin{align}
        \Ent_{\phi_{0,\ell}^+}(f^2) \le 2\sum_{v\in V} \phi_{0,\ell}^+\left[\Ent_v(f^2)\right] \,. \label{eq:large-d-lsi-1}
    \end{align}
    Thus $\widetilde{c}_{\mathsf{LS}}(\phi_{0,\ell}^+) \ge \frac{1}{2}$. By Lemma~\ref{lem:compare-c-LS-widetilde-c-LS}, we then have $c_{\mathsf{LS}}(\phi_{0,\ell}^+) \ge (8C)^{-1}$.
\end{proof}

Now we can prove that $\phi_{k,\ell}^+$ has positive log-Sobolev constant, as incorporated in the following lemma.
\begin{lemma}{\label{lem-log-Sobolev-Tkl}}
    In the regime \eqref{eq-assump-C-d-epsilon}, for every $\ell\ge1$ and every $1\le k\le10^4$, we have $\widetilde{c}_{\mathsf{LS}}(\phi_{k,\ell}^{+})\ge 0.49$ and $c_{\mathsf{LS}}(\phi_{k,\ell}^{+})\ge (9C)^{-1}$. 
\end{lemma}
\begin{proof}
    By Lemma~\ref{lem:compare-c-LS-widetilde-c-LS}, it suffices to show that $\widetilde{c}_{\mathsf{LS}}(\phi_{k,\ell}^{+})\ge 0.49$. Fix $k\le10^4$. Start with the full tree $\mathbb T_{0,\ell}$ and choose $k$ depth-one child branches $R_1,\dots,R_k$ of the root. Let
    \begin{align*}
        S = V(T_{0,\ell}) \setminus \cup_{j=1}^k R_j \,,
    \end{align*}
    where the root is retained. Let $\mu=\phi_{0,\ell}^{+}$ be the all-plus measure on $\mathbb T_{0,\ell}$, and let $\nu=\mu_S$ be its marginal on the retained vertices. By Lemmas~\ref{lem-MSW} and \ref{lem-marginal},
    \begin{equation}{\label{eq-log-sobolev-Tkl-1}}
        \widetilde{c}_{\mathsf{LS}}(\nu)\ge \widetilde{c}_{\mathsf{LS}}(\mu) \ge \tfrac{1}{2} \,.  
    \end{equation}
    We now compare $\nu$ with $\phi_{k,\ell}^{+}$. For a deleted branch $R_j$, let $r_j$ be the child of the root in that branch, and define the effective branch partition function (below we write $\Lambda(R_j)$ to be the collection of $\eta\in\{\pm1\}^{R_j}$ where the ordinary leaves of $R_j$ are fixed to $+1$)
    \begin{align*}
        W_j(s) = \sum_{ \eta\in\Lambda(R_j) } \exp\Bigg\{ \beta s\eta_{r_j} + \beta\sum_{xy\in E(R_j)} \eta_x \eta_y \Bigg\}, \quad s\in\{\pm1\} \,.
    \end{align*}
    If $\ell=1$, the same formula is interpreted as $W_j(s)=e^{\beta s}$, since the deleted child is itself a fixed plus boundary leaf. Under $\nu$, each deleted branch contributes the effective weight $W_j(\sigma_o)$ to the root spin. Under $\phi_{k,\ell}^{+}$, that branch is replaced by a fixed minus leaf and contributes $e^{-\beta\sigma_o}$. Hence on the common spin space,
    \begin{equation}{\label{eq-log-sobolev-Tkl-2}}
        \frac{\mathsf d\phi_{k,\ell}^{+}}{\mathsf d\nu}(\sigma) \propto \prod_{j=1}^k G_j(\sigma_o), \quad G_j(s):=\frac{e^{-\beta s}}{W_j(s)} \,.
    \end{equation}
    For every branch configuration, changing the root spin from $-1$ to $+1$ changes only the root-branch edge contribution, and the corresponding multiplicative factor lies between $e^{-2\beta}$ and $e^{2\beta}$. Therefore
    \begin{align*}
        e^{-2\beta} \le \frac{W_j(+)}{W_j(-)} \le e^{2\beta} \,.
    \end{align*}
    Thus,
    \begin{align*}
        \frac{G_j(+)}{G_j(-)} = e^{-2\beta}\frac{W_j(-)}{W_j(+)} \in [e^{-4\beta},1] \,.
    \end{align*}
    Thus the log-density in \eqref{eq-log-sobolev-Tkl-2} has oscillation at most $4\beta k$. Applying Lemma \ref{lem-bounded-perturbation} and \eqref{eq-log-sobolev-Tkl-1},
    \begin{align*}
        \widetilde{c}_{\mathsf{LS}}(\phi_{k,\ell}^{+}) \ge e^{-16\beta k} \widetilde{c}_{\mathsf{LS}}(\nu) \ge e^{-16\beta k} \cdot \frac{1}{2} \,.
    \end{align*}
    Since $k\le10^4$ and $\beta=C/d$ with \eqref{eq-assump-C-d-epsilon}, this is bounded below by $0.49$.
\end{proof}

\begin{cor}\label{cor-any-signs}
    Fix $r\le10^4$ and signs $\xi_1,\dots,\xi_r\in\{\pm1\}$. Consider the analogue of $\phi_{r,\ell}^{+}$ in which the $r$ attached root leaves are fixed to $\xi_1,\dots,\xi_r$, while the ordinary depth-$\ell$ leaves are fixed to $+1$. Then its log-Sobolev constant $\widetilde{c}_{\mathsf{LS}}$ is bounded below by $\frac{4}{9}+0.01$, uniformly over all sign choices and all $r\le10^4$.
\end{cor}
\begin{proof}
    Changing an attached root leaf from $-1$ to $+1$ changes the Hamiltonian by $2\beta\sigma_o$. If at most $r$ attached leaves are changed, the perturbation has oscillation at most $4\beta r$. Lemma~\ref{lem-bounded-perturbation}, applied together with Lemma \ref{lem-log-Sobolev-Tkl}, gives the claim.
\end{proof}

We end this subsection by listing some corollaries of a positive log-Sobolev constant bound. 
\begin{lemma}{\label{lem-almost-indepdent-in-time-ls}}
    Suppose that the Ising measure \(\mu\) satisfies $c_{\mathsf{LS}}(\mu) \ge \rho>0$. Let \((X_t)_{t\ge 0}\) be the continuous-time Glauber dynamics reversible with respect to \(\mu\), started from \(X_0\sim\mu\). Fix \(\Delta=10^{-2}\).  Let \(A_1,\ldots,A_T\) be events such that \(A_i\) is measurable with respect to the path segment $\{X_t: i\Delta\le t\le (i+1)\Delta\}$. Then there is a constant \(c=c(\rho)>0\), depending only on \(\rho\), such that
    \[
        \Pb_\mu\left(\bigcap_{i=1}^T A_i\right) \le \left( \prod_{i=1}^T \Pb_\mu(A_i) \right)^c \,.
    \]
    In addition, one may take explicitly
    \[
        c=\frac14\tanh(\rho\Delta) =\frac14\tanh(10^{-2}\rho) \ge 10^{-3} \rho \mbox{ for } \rho<0.01 \,.
    \]
\end{lemma}

\begin{proof}
    By stationarity, shifting all events by time \(-\Delta\) does not change either side of the desired inequality.  Hence it is enough to prove the same bound under the notational convention that \(A_i\) is measurable with respect to $\{X_t:(i-1)\Delta\le t\le i\Delta\}$. For each \(i\), let \(\mathcal A_i\) be the corresponding event for a path of length \(\Delta\), so that $A_i= \{ (X_{(i-1)\Delta+s})_{0\le s\le \Delta}\in\mathcal A_i \}$. Define the restricted, or killed, path-segment operator
    \[
        Q_i f(x) := \E_x\left[ \mathbf 1_{\mathcal A_i}\, f(X_\Delta) \right], \quad x\in\Omega,
    \]
    for bounded measurable \(f\ge 0\). By the Markov property, if we let $F_{T+1}=\mathbf 1$ and $F_i=Q_iF_{i+1}$ for $1 \le i \le T$, then
    \[
        \Pb_\mu\left(\cap_{i=1}^T A_i\right) = \langle \mathbf 1,F_1 \rangle = \int F_1 \, \mathsf d\mu \,.
    \]
    Also, we have $\int Q_i\mathbf 1\,\mathsf d\mu = \Pb_\mu(A_i)$ and $0\le Q_i\mathbf 1\le 1$. We next prove the main one-step estimate. Since \(c_{\mathsf{LS}}(\mu)\ge\rho\), Gross hypercontractivity for the continuous-time semigroup (see, e.g., \cite[Theorem~1]{Gross75} and \cite[Theorem~3.5]{DSC96}) gives, for every \(q>1\),
    \[
        \|P_s h\|_{r(s)}\le \|h\|_q, \qquad r(s)=1+(q-1)e^{2\rho s} \,.
    \]
    We shall use this only with \(q=2\) and \(s=\Delta\). Put $b:=1+e^{2\rho\Delta}$. Then $\|P_\Delta h\|_{b}\le \|h\|_2$. Choose \(a>1\) by $\frac1a+\frac1b=\frac12$ (this is possible because \(b>2\)). For \(f\ge 0\), Cauchy-Schwarz under the law \(\Pb_x\) of the path segment gives
    \begin{align*}
        Q_i f(x) =\E_x\left[\mathbf 1_{\mathcal A_i} f(X_\Delta)\right] \le
        \left(\E_x[\mathbf 1_{\mathcal A_i}]\right)^{1/2}
        \left(\E_x[f(X_\Delta)^2]\right)^{1/2} =
        \bigl(Q_i\mathbf 1(x)\bigr)^{1/2}
        \bigl(P_\Delta f^2(x)\bigr)^{1/2} \,.
    \end{align*}
    Taking \(L^4(\mu)\)-norms and applying H\"older's inequality with \(1/a+1/b=1/2\), we obtain
    \begin{align*}
        \|{Q_i f}\|_4 \le \left(\int (Q_i\mathbf 1)^2 (P_\Delta f^2)^2\,\mathsf d\mu\right)^{1/4} \le
        \|{Q_i\mathbf 1}\|_a^{1/2}
        \|{P_\Delta f^2}\|_b^{1/2} \,.
    \end{align*}
    Because \(0\le Q_i\mathbf 1\le 1\),
    \[
        \|{Q_i\mathbf 1}\|_a
        =
        \left(\int (Q_i\mathbf 1)^a\,\mathsf d\mu\right)^{1/a}
        \le
        \left(\int Q_i\mathbf 1\,\mathsf d\mu\right)^{1/a}
        =
        \Pb_\mu(A_i)^{1/a} \,.
    \]
    The hypercontractive estimate gives $\|{P_\Delta f^2}\|_b \le \|{f^2}\|_2 = \|{f}\|_4^2$. Therefore $\|{Q_i f}\|_4 \le \Pb_\mu(A_i)^{\frac{1}{2a}}\|{f}\|_4$. Set $c:=\frac{1}{2a}$. Since \(1/a=1/2-1/b\), this constant equals
    \[
        c= \frac14-\frac{1}{2(1+e^{2\rho\Delta})}=
        \frac{e^{2\rho\Delta}-1}{4(1+e^{2\rho\Delta})} =
        \frac14\tanh(\rho\Delta)>0 \,.
    \]
    Thus the one-step estimate is $\|{Q_i f}\|_4 \le \Pb_\mu(A_i)^c\|{f}\|_4$. Applying this estimate recursively to \(F_i=Q_iF_{i+1}\) yields $\|{F_i}\|_4 \le \Pb_\mu(A_i)^c\|{F_{i+1}}\|_4$. Iterating from \(i=T\) down to \(i=1\), and using \(\|{F_{T+1}}\|_4=1\), we get
    \begin{align*}
        \Pb_\mu\left(\bigcap_{i=1}^T A_i\right) =\int F_1\,\mathsf d\mu 
        \le \|{F_1}\|_4 \le \prod_{i=1}^T \Pb_\mu(A_i)^c = \left(\prod_{i=1}^T \Pb_\mu(A_i)\right)^c .
    \end{align*}
    This proves the lemma.
\end{proof}

\begin{lemma}{\label{lem-BAD-lift-ls}}
    Suppose $c_{\mathsf{LS}}(\mu)\ge \rho_0>0$. Let $(X_s)_{s\ge0}$ be the Glauber dynamics started from stationary measure $X_0\sim\mu$. Let $E$ be an event measurable with respect to $\sigma(X_s:s\ge t)$, and set $p := \Pb_\mu(E)$. For $0<p<1$, define
    \[
        \mathcal I := \left\{ x\in\Omega: \Pb_\mu(E\mid X_0=x) \ge p^{1/2} \right\} \,.
    \]
    Then
    \begin{equation}\label{eq:explicit-bound}
        \mu(\mathcal I) \le p^{\frac16+\frac1{12}e^{2\rho_0 t}} \,.
    \end{equation}
\end{lemma}
\begin{proof}
    Define the future-hitting function $h(y) := \Pb_\mu(E\mid X_t=y)$. Clearly we have $0 \le h \le 1$. Because the chain is started from stationarity, $X_t\sim\mu$, and therefore
    \begin{equation}\label{eq:h-mean}
        \int h\,\mathsf d\mu = \E_\mu\bigl[\Pb_\mu(E\mid X_t)\bigr] = \Pb_\mu(E) = p \,.
    \end{equation}
    By the Markov property, conditioning first on $X_t$ gives, for every initial state $x$,
    \begin{equation*}
        \Pb_\mu(E\mid X_0=x) = \E_x[h(X_t)] = \mathsf P_t h(x) \,.
    \end{equation*}
    Thus $\mathcal I = \{x\in\Omega: \mathsf P_t h(x)\ge p^{1/2}\}$. Fix any $r\in(1,2)$ and set $q := 1+(r-1)e^{2\rho_0 t}$. Applying Markov's inequality to the nonnegative random variable $(\mathsf P_t h)^q$ under $\mu$ gives
    \begin{align}
        \mu(\mathcal I) = \mu\bigl(\mathsf P_t h\ge p^{1/2}\bigr) \le p^{-q/2}\int (\mathsf P_t h)^q\,\mathsf d\mu = p^{-q/2} \big\| {\mathsf P_t h} \big\|_{q}^{q}. \label{eq:markov}
    \end{align}
    By hypercontractivity, $0\le h\le1$, and \eqref{eq:h-mean},
    \begin{equation}\label{eq:hc-applied}
        \big\| \mathsf P_t h \big\|_{q}^{q} \le \big\| {h} \big\|_{r}^{q} = \left(\int h^r\,\mathsf d\mu\right)^{q/r} \le \left(\int h\,\mathsf d\mu\right)^{q/r} = p^{q/r} \,.
    \end{equation}
    Combined with \eqref{eq:markov}, we obtain
    \begin{equation}\label{eq:general-r}
        \mu(\mathcal I) \le p^{-q/2}p^{q/r} = p^{q(1/r-1/2)} \,.
    \end{equation}
    Now take $r=3/2$. Then $q = 1+\frac12 e^{2\rho_0 t}$ and $\frac1r-\frac12 = \frac16$. Hence \eqref{eq:general-r} becomes
    \begin{equation*}
        \mu(\mathcal I) \le p^{\left(1+\frac12 e^{2\rho_0 t}\right)/6} = p^{\frac16+\frac1{12}e^{2\rho_0 t}} \,,
    \end{equation*}
    which proves \eqref{eq:explicit-bound}.
\end{proof}

\subsection{Cavity field representation}{\label{subsec:BP-fixed-point}}

For the ferromagnetic Ising model on a tree $T=T(V,E)$ with inverse temperature $\beta$, the cavity field representation (Bethe recursion) \cite{EKPS00} is defined as follows. Define the cavity message from $u$ to $v$ by $m_{u \to v}=\tanh(h_{u\to v})$, where $h_{u \to v}$ is the effective field at $u$ after deleting the edge $(u,v)$. Then the exact recursion on a tree is
\begin{align*}
    h_{u \to v} = \sum_{ w\sim u, w \neq v } \tanh^{-1}\!\big( \tanh(\beta) m_{w \to u} \big) \,,
\end{align*}
or equivalently, we have the following belief propagation (BP)
\begin{align}{\label{eq:BP-recursion}}
    m_{u \to v} = \tanh\left( \sum_{ w\sim u, w \neq v } \tanh^{-1}\!\big( \tanh(\beta) m_{w \to u} \big) \right) \,.
\end{align}
It is known \cite{BRZ95, YP24} that in the low-temperature regime $\beta>\tanh^{-1}(\frac{1}{d-1})$, the BP recursion has three fixed points $-m_*(\beta),0,m_*(\beta)$. We first need the following result, which shows that when the cavity messages from the leaves are positive, the cavity message converges exponentially fast to the BP fixed point $m_*$.

\begin{lemma}{\label{lem-exponential-contraction-BP-fixed-point}}
    Consider the BP recursion on $\mathbb T_{1,\ell}$ where $\beta=\frac{C}{d}$ satisfying \eqref{eq-assump-C-d-epsilon}. Suppose that for all leaves $v \in \mathsf L(\mathbb T_{1,\ell})$, the message satisfies $m_{v \to p(v)} \geq 0.1$. Then  the message of the root satisfies 
    \begin{align*}
        \left| m_{o}-m_* \right| \leq 2^{-\ell} \,.
    \end{align*}
\end{lemma}
\begin{proof}
    Recall \eqref{eq:BP-recursion}. Set $\theta:=\tanh\beta$ and $q:=d-1$. Then $m_*$ satisfies $m_*=F(m_*)$, where
    \[
        F(m):=\tanh\bigl((d-1)\atanh(m\tanh(\beta))\bigr),\quad 0\le m\le 1 \,.
    \]
    We first show that $F$ maps $I:=[0.1,1]$ into itself and is a uniform contraction there. After increasing $d_0$ if necessary, the upper bound $C<C_0\log d$ guarantees that $\beta=C/d\le 1/2$. For $0\le x\le 1/2$, we have $\frac{x}{2}\le \tanh x\le x$. Consequently,
    \[
        \frac C4 \le (d-1)\tanh(C/d)=(d-1)\tanh\beta \le C \,.
    \]
    Since $\atanh x\ge x$ for $x\ge0$, for $m\in I$ we have
    \[
        s(m):=(d-1)\atanh(m(\tanh\beta)) \ge (d-1)m(\tanh\beta) \ge C/40 \,.
    \]
    Thus, by \eqref{eq-assump-C-d-epsilon} we have $F(m)\ge \tanh(C/40)\ge \frac1{10}$. Also $F(m)<1$, so $F(I)\subseteq I$. Differentiating gives
    \[
        F'(m) =\frac{ (d-1)(\tanh\beta)\sech^2(s(m))}{1-(\tanh\beta)^2m^2} \,.
    \]
    Since $\tanh\beta \le\tanh(1/2)<1/2$, $(d-1)\tanh\beta \le C$, and $\sech^2 x\le 4e^{-2x}$ for $x\ge0$, we obtain, uniformly for $m\in I$,
    \[
        0\le F'(m) \le \frac{C\cdot 4e^{-2(C/40)}}{1-1/4} =\frac{16C}{3}e^{-C/20} \overset{\eqref{eq-assump-C-d-epsilon}}{\le} \frac{1}{2} \,.
    \]
    Therefore $|F(x)-F(y)|\le \frac12|x-y|$ for all $x,y\in I$. Since $F(I)\subseteq I$, the contraction theorem gives a unique fixed point of $F$ in $I$. In addition, it follows from the contraction estimate that, for every $x\in I$ and every $r\ge0$,
    \[
        |F^{(r)}(x)-m_*|\le 2^{-r}|x-m_*|\le 2^{-r}.
    \]
    We now pass to arbitrary, possibly non-homogeneous, leaf messages. Define
    \[
        a_0:=\frac1{10}, \quad u_0:=1, \quad a_r:=F^{(r)}(a_0), \quad u_r:=F^{(r)}(u_0) \,.
    \]
    The multivariate BP update
    \[
        \Phi(x_1,\dots,x_q) :=\tanh\left( \, \sum_{i=1}^{d-1} \atanh((\tanh\beta)x_i)\right)
    \]
    is increasing in each coordinate. By induction from the leaves toward the root, it is clear that a vertex whose descendant subtree has height $r$ sends a message in $[a_r,u_r]$. Thus, we obtain
    \[
        a_\ell\le m_o\le u_\ell, \quad a_\ell\le m_*\le u_\ell \,.
    \]
    Both endpoints $a_\ell,u_\ell$ are within $2^{-\ell}$ of $m_*$, and therefore
    \[
        |m_o-m_*| \le \max\big\{ |a_\ell-m_*|,|u_\ell-m_*| \big\} \le 2^{-\ell} \,. \qedhere
    \]
\end{proof}

%\textcolor{red}{amenable might be confusing word to use since there is a notion of amenable graphs}

We will also give a sufficient condition that ensures that the message from the root is positive.
\begin{defn}{\label{def:amenable-realization-in-tree}}
    Consider the boundary condition $\sigma_S \in \{ -1,+1 \}^S$ where $S=\mathsf L(\mathbb T_{1,\ell})$. We say a leaf $u$ is non-favorable with respect to $\sigma_S$, precisely when its boundary spin is minus. In addition, we say an internal vertex $u$ is non-favorable with respect to $\sigma_S$, if it has at least $0.1d$ non-favorable children. Equivalently, non-favorability produces a complete \(\lceil0.1d\rceil\)-ary witness whose terminal boundary spins are all minus. A boundary condition $\sigma_S \in \{ -1,+1 \}^S$ is called a favorable boundary condition, if the root $o$ is favorable with respect to $\sigma_S$.
\end{defn}
\begin{lemma}{\label{lem-amenable-boundary-condition-positive-field}}
    If the boundary condition $\sigma_S$ is favorable, then for every favorable vertex $v$ the cavity field at $v$ satisfies $h_v \geq 0.1C$.
\end{lemma}
\begin{proof}
    We will prove this lemma by induction. The result clearly holds for the leaves of $\mathbb T_{1,\ell}$, with the convention that a fixed spin $+1$ (respectively $-1$) has cavity field $+\infty$ (respectively $-\infty$). Recall that the upward cavity field 
    \begin{align*}
        h_v = \sum_{ u:u \leftarrow v } \psi_\beta(h_u), \quad \psi_\beta(x):= \tanh^{-1}\left( \tanh(\beta) \tanh(x) \right) \,.
    \end{align*}
    Favorability of $v$ implies that $v$ has at least $d-1-0.1d=0.9d-1$ favorable children. Thus, using the induction hypothesis we have
    \begin{align*}
        h_v \geq (d-0.1d-1) \cdot \psi_\beta(0.1C) -0.1d\beta \,.
    \end{align*}
    Since $\beta=\frac{C}{d}$ where $C$ is large, we have
    \begin{align*}
        \psi_\beta(0.1C) = \tanh^{-1}\left( \tanh\left( \frac{C}{d} \right) \tanh(0.1C) \right) = [1+o_C(1)] \frac{C}{d} \,.
    \end{align*}
    We then get
    \begin{equation*}
        h_v \geq (0.9d-1) \cdot \frac{C}{2d} - 0.1d \cdot \frac{C}{d} \geq 0.1 C \,.  \qedhere
    \end{equation*}
\end{proof}

\begin{remark}{\label{lem-comparison-Ising-external-field}}
    Later we will also repeatedly use the following comparison result for Ising measures with different external fields. For two external-field vectors $h,g \in \mathbb R^{S}$ where $S=\mathsf L(\mathbb T_{k,\ell})$, let $\pi_g$ and $\pi_h$ be the Ising measure with external field $g$ and $h$, respectively. We then have (recall $\beta=\frac{C}{d}$ with \eqref{eq-assump-C-d-epsilon})
    \begin{align*}
        \frac{\dd \pi_g(\sigma)}{\dd \pi_h(\sigma)} =\frac{Z_h}{Z_g} \exp\left( -\sum_{x \in S} \sigma_x (h_x-g_x) \right) \le \exp\left( 2\sum_{x \in S} |h_x-g_x| \right) \,.
    \end{align*}
    In addition, for $0 \leq t \leq 1$, let $h(t)=g+t(h-g)$, and let $\pi_t=\pi_{h(t)}$. For every set $A \subseteq \{ -1,+1 \}^{V}$, we have
    \begin{align*}
        \frac{\mathrm{d}}{\mathrm{d}t} \pi_t(A) = \mathsf{Cov}_{\pi_t}\left( \mathbf 1_A, \sum_{x \in S} (h_x-g_x) \sigma_x \right) \,.
    \end{align*}
    Therefore 
    \begin{align*}
        \left| \frac{\mathrm{d}}{\mathrm{d}t} \pi_t(A) \right| \leq \sum_{x \in S} |h_x-g_x| \cdot \left| \mathsf{Cov}_{\pi_t}\left( \mathbf 1_A, \sigma_x \right) \right| \leq \sum_{x \in S} |h_x-g_x| \,,
    \end{align*}
    where the last inequality follows from $\sigma_x \in \{ -1,+1 \}$. Integrating the above inequality from $0$ to $1$ yields that 
    \begin{align*}
        \left\| \pi_g-\pi_h \right\|_{\mathsf{TV}} \leq \sum_{ x\in S } |g_x-h_x| \,.
    \end{align*}
\end{remark}

\section{Structural results for random regular graphs}{\label{sec:structure-random-regular-graph}}

%\textcolor{red}{It's probably obvious but maybe in the notation setting when we define path we specify that is means a self-avoiding path.  In the bootstrap construction, do we want the $P_i$ to be disjoint or just the endpoints?}

We now formalize the desirable properties we require for typical $G \sim \mathcal G_d(n)$, as described in Section~\ref{subsec:proof-overview}. To this end, we introduce the following variant of bootstrap percolation on a graph, which will be useful throughout this section.
\begin{defn}{\label{def-bootstrap-percolation}}
    Given a graph $G$ and $A \subseteq V(G)$ and an integer $\ell\in \mathbb Z_{\ge0}$, consider the path-bootstrap percolation on $G$ with respect to $A$, defined as follows:
    \begin{itemize}
        \item Define $(A)^{\circ}$ to be the set of vertices $v$ such that either $v \in A$ or there are $3$ non-backtracking walks with different starting edges and length bounded by $4\ell$ that connect $v$ and $A$. Equivalently, define  
        \begin{align*}
            (A)^{\circ} := A \cup \Big\{ & v \in V(G): \exists u_1,u_2,u_3 \in A, P_i=(x_{i,0},\ldots,x_{i,m_i})\ (1\le i\le3), \\
            & P_i\mbox{ is non-backtracking},\ x_{i,0}=v,\ x_{i,m_i}=u_i, \\
            & x_{i,1} \neq x_{j,1} \mbox{ for } i \neq j, 1\le m_i \leq 4\ell \mbox{ for } 1 \leq i \leq 3 \Big\} \,.
        \end{align*}
        \item Define $A_{\mathsf{f}} := (((A)^{\circ})^{\circ} \ldots)^{\circ}$. The iteration stops when no new vertex can be added.
    \end{itemize}
\end{defn}
\begin{defn}{\label{def-plus-chain}}
    For a graph $G$ and $A \subseteq V(G)$, define
    \begin{align*}
        \mathcal P(A):= \Big\{ P=(x_0,\ldots,x_m):\ & P\mbox{ is non-backtracking},\ x_0,x_m\in A, \\
        & x_i\notin A\ (1\le i<m),\ 1\le m \leq 2\ell+2 \Big\} \,,
    \end{align*}
    and let $(A)^+=A \cup \cup_{P \in \mathcal P(A)} V(P)$. Members of $\mathcal P(A)$ are called first-return excursions and are counted up to reversal. Their two endpoints may coincide, but no internal position may lie in $A$. In particular, a simple cycle meeting $A$ at exactly one vertex is allowed, whereas repeated circuits through that vertex are not.
\end{defn}
\begin{defn}{\label{def-typical-graph}}
    Define 
    \begin{equation}{\label{eq-def-ell-p}}
        \ell=\ell(p) = \bigg\lfloor \frac{\log(1/p)}{1000\log d} \bigg\rfloor \,.
    \end{equation}
    We say a connected, $d$-regular graph $G=(V,E)$ is \emph{typical}, denoted by $G \in \mathcal G$, if the following conditions hold.
    \begin{enumerate}
        \item[(1)] The events in Lemmas~\ref{lem-FK-giant-size} and \ref{lem-TV-dist-giant-is-plus} hold for $G$.
        \item[(2)] For any constants $\delta,\eta$ such that $0<\delta<\eta<\frac{1}{2}$ and for any set $A \subset V$ with $|A| \leq \delta n$, we have
        \begin{align*}
            \#\Big\{ v \in V: \mathsf{deg}_A(v) \geq \eta d \Big\} \leq \frac{4}{d(\eta-\delta)^2} \cdot |A| \,.
        \end{align*}
        \item[(3)] For every $A\subseteq V$ with $1\le |A|\le \frac{n}{2}$, we have
        \begin{align*}
            \#\Big\{ (x,y) \in E:x\in A, y\notin A \Big\} \geq \frac{d|A|}{4} \,.
        \end{align*} 
        \item[(4)] For all $A \subseteq V$ such that $|A|=pn$ where $p \leq d^{-4}$, we have
        \begin{align*}
            \tau(G_A) \leq \frac{10\log d}{\log(1/p)} \cdot |A| \,.
        \end{align*}
        \item[(5)] For all $A \subseteq V(G)$ such that $1\le |A|=pn$ where $p\leq d^{-5}$, we have $|A_{\mathsf{f}}| \leq 10|A|$. In addition, if $p\le d^{-4000}$, we have $|\mathcal P(A_{\mathsf{f}})|\leq 100|A|$ and $|(A_{\mathsf f})^+|\le200\ell|A|$.
        \item[(6)] For every nonempty $A,A'\subseteq V(G)$ where $|A|=pn$ and $|A'|=qn$, let $W_{m}=W_{m}(A,A')$ denote the number of oriented walks $(x_0,x_1,\ldots,x_m)$ in $G$ such that
        \[
            x_0\in A,\quad x_i\sim x_{i+1}\ (0\leq i<m), \quad x_m \in A' \,.
        \]
        Then $W_{m}(A,A')\leq 2\sqrt{|A||A'|}d^{0.51m}$ for all $1 \le m\leq \frac{\min\{ \log(1/p),\log(1/q) \} }{ 10\log d }$.
    \end{enumerate}
\end{defn}
The main goal of this section is to show the following structural results on a random regular graph $G \sim \mathcal G_d(n)$.
\begin{proposition}{\label{prop-structure-random-regular-graph}}
    Suppose $G\sim\mathcal G_d(n)$ where $d\in\mathbb N$ satisfies \eqref{eq-assump-C-d-epsilon}. Then $\Pb(G \in \mathcal G)=1-o(1)$. 
\end{proposition}

The rest of this section is devoted to the proof of Proposition~\ref{prop-structure-random-regular-graph}. By Lemmas~\ref{lem-FK-giant-size} and \ref{lem-TV-dist-giant-is-plus}, Item~(1) holds for $G \sim \mathcal G_d(n)$ with probability $1-o(1)$. We first prove that Definition~\ref{def-typical-graph}(2) holds for $G \sim\mathcal G_d(n)$ with probability $1-o(1)$ using the following result proved in \cite{GSS25+}.
\begin{lemma}[Lemma~4.9 in \cite{GSS25+}]{\label{lem-edge-expansion}}
    Suppose $G\sim\mathcal G_d(n)$. Then with probability $1-o(1)$, for any constants $\delta,\eta$ such that $0<\delta<\eta<\frac{1}{2}$ and for any set $S \subseteq V(G)$ with $|S|\leq \delta n$, we have
    \begin{align*}
        \#\left\{ v \in V(G): \mathsf{deg}_S(v) \geq \eta d \right\} \leq \frac{4}{d(\eta-\delta)^2} \cdot |S| \,.
    \end{align*}
\end{lemma}
We then prove that Definition~\ref{def-typical-graph}(3) holds for $G \sim\mathcal G_d(n)$ with probability $1-o(1)$, as incorporated in the following lemma.
\begin{lemma}\label{lem:typical-expander}
    For \(S \subseteq V\), let $\partial S:=\{(x,y) \in E:x\in S,\ y\notin S\}$. For every fixed integer \(d\geq 256\), we have 
    \begin{align}{\label{eq-edge-expansion}}
        \Pb_{G\sim\mathcal G_d(n)}\left( |\partial S|\ge \frac{d|S|}{4} \text{ for every }S\subseteq V(G),\ 1\le |S|\le \frac{n}{2} \right)=1-o(1) \,.
    \end{align}
\end{lemma}
\begin{proof}
    Let \(A_G\) be the adjacency matrix of \(G\), and let \(d=\lambda_1(G)\ge\lambda_2(G)\ge\cdots\ge\lambda_n(G)\) be its eigenvalues. By Friedman's theorem \cite{Friedman08}, for every fixed \(d\ge3\) and every \(\varepsilon>0\),
    \begin{align*}
        \Pb_{G \sim \mathcal G_d(n)}\big( \max\{|\lambda_2(G)|, |\lambda_n(G)|\} \le 2\sqrt{d-1}+\varepsilon \big) =1-o(1) \,.
    \end{align*}
    Taking \(\varepsilon=1\), the assumption \(d\ge256\) gives $2\sqrt{d-1}+1\le \frac{d}{2}$. It therefore suffices to show that every \(d\)-regular graph satisfying \(\max\{|\lambda_2(G)|,|\lambda_n(G)|\}\le d/2\) also satisfies the event in \eqref{eq-edge-expansion}. Fix \(S\subseteq V(G)\), put \(s=|S|\), and write (below \(\mathbf 1\) is the all-ones vector)
    \begin{align*}
        \mathbf 1_S=\frac{s}{n}\mathbf 1+f, \quad f\perp\mathbf 1 \,.
    \end{align*}
    Since \(\|f\|_2^2=s(1-s/n)\),
    \begin{align*}
        |\partial S| =ds-\mathbf 1_S^{\mathsf T}A_G\mathbf 1_S =ds\Big(1-\frac{s}{n}\Big)-f^{\mathsf T}A_G f \ge (d-\lambda_2(G))s \Big(1-\frac{s}{n} \Big) \,.
    \end{align*}
    If \(s\le n/2\) and \(\max\{|\lambda_2(G)|,|\lambda_n(G)|\} \le d/2\), the last expression is at least \(ds/4\), proving \eqref{eq-edge-expansion}.
\end{proof}

We then show that Definition~\ref{def-typical-graph}(4) holds for $G \sim\mathcal G_d(n)$ with probability $1-o(1)$, as incorporated in the following lemma.
\begin{lemma}{\label{lem-small-set-edge-density}}
    Suppose that $G\sim\mathcal G_d(n)$ where $d \in \mathbb N$ satisfies \eqref{eq-assump-C-d-epsilon}. Then with probability $1-o(1)$, for all $A \subset V(G)$ such that $|A|=pn$ where $p \leq d^{-4}$, we have
    \begin{align*}
        \tau(G_A) \leq \frac{10\log d}{\log(1/p)} \cdot |A| \,.
    \end{align*}
\end{lemma}
\begin{proof}
We prove the assertion in the configuration model and then condition on simplicity; for fixed $d$ this conditioning has probability bounded away from zero. For a fixed $k$-element set $A$ and an integer $m\ge k$, a union bound over $m$ disjoint internal pairs gives
\[
 \Pr(|E(G_A)|\ge m)
 \le\binom{dk}{2m}\frac{(2m-1)!!}{\prod_{j=0}^{m-1}(dn-2j-1)}
 \le\left(\frac{e^2dk}{n}\right)^m,
 \qquad k\le n/d^4,
\]
whenever $2m\le dk$; otherwise the event is empty. The last inequality follows from $dn-2m\ge dn-dk\ge dn/2$ and $m\ge k$. In particular, no equality is asserted between the event probability and the pairing union bound.

Put $k_0=\lceil\log n/\log\log n\rceil$. For $1\le k<k_0$, any violation of the claimed excess bound implies $|E(G_A)|\ge k+1$. Therefore
\[
 \sum_{k<k_0}\binom nk\left(\frac{e^2dk}{n}\right)^{k+1}
 \le\sum_{k<k_0}\frac{e^2dk}{n}(e^3d)^k
 =n^{-1+o(1)}.
\]
For $k_0\le k\le n/d^4$, set $p=k/n$ and $\varepsilon=10\log d/\log(1/p)$. A violation requires at least $m=\lfloor(1+\varepsilon)k\rfloor+1$ internal edges. Since $e^2dp<1$, the preceding bound and $\binom nk\le(en/k)^k$ yield
\begin{align*}
 \Pr(\exists A:|A|=k,\ \tau(G_A)>\varepsilon k)
 &\le\exp\{k\log(e/p)+(1+\varepsilon)k\log(e^2dp)\}\\
 &=\exp\{k[-\varepsilon\log(1/p)+(1+\varepsilon)\log d+3+2\varepsilon]\}\\
 &\le e^{-k\log d}.
\end{align*}
Here $\varepsilon\le2.5$, and the last inequality holds for sufficiently large $d$. Summing the geometric series over $k\ge k_0$ gives $o(1)$. Together with the small-$k$ bound this proves the simultaneous assertion for all nonempty $A$; the empty-set case is interpreted separately as trivial.
\end{proof}

Then, we prove that Definition~\ref{def-typical-graph}(5) holds for $G \sim \mathcal G_d(n)$ with probability $1-o(1)$, as incorporated in the following two lemmas.
\begin{lemma}{\label{lem-bootstrap-stop}}
    Suppose that $G \sim \mathcal G_d(n)$, where $d \in \mathbb N$ satisfies \eqref{eq-assump-C-d-epsilon}. Then with probability $1-o(1)$, for all $A \subseteq V(G)$ such that $1\le |A|=pn$ where $p\leq d^{-5}$, we have $|A_{\mathsf{f}}| \leq 10pn$.
\end{lemma}
\begin{proof}
    If $\ell=0$, no vertex can be added and $A_{\mathsf f}=A$. We may therefore assume $\ell\ge1$. By \eqref{eq-def-ell-p}, we have $p\le d^{-1000\ell}$; consequently, $120\ell p\le d^{-4}$ and $\log(1/(120\ell p))\ge500\ell\log d$ for sufficiently large $d$.
    Denote by $N=10pn$ and suppose on the contrary that $|A_{\mathsf{f}}| > N$. Then clearly we can find $v_1,\ldots,v_{0.9N}$ such that if we define $A_0=A$ and $A_i=A_{i-1}\cup \{ v_i \}$, then there exist three non-backtracking walks $P_{i,1},P_{i,2},P_{i,3}$ connecting $v_{i}$ and $A_{i-1}$ and the first edges of these walks are distinct. Each walk has length at most $4\ell$ and can be stopped upon first reaching $A_{i-1}$. Define 
    \begin{align*}
        G_0 = G_A \mbox{ and } G_i = G_{i-1} \cup P_{i,1} \cup P_{i,2} \cup P_{i,3} \,.
    \end{align*}
    We then have the following observation:
    \begin{itemize}
        \item $|V(G_{0.9N})| \leq |A|+(12\ell-2)\cdot 0.9N \leq 12\ell N=120\ell pn \le d^{-4}n$. 
        \item The degree of $v_1,\ldots,v_{0.9N}$ in $G_{0.9N}$ is at least $3$.
        \item The degree of $v \in V(G_{0.9N}) \setminus A$ in $G_{0.9N}$ is at least $2$. Indeed, at its first appearance, a vertex not among the added starting vertices is internal to a non-backtracking witness and therefore has distinct preceding and succeeding neighbors in its trace.
    \end{itemize}
    Thus, we have
    \begin{align*}
        2|E(G_{0.9N})| &= \sum_{ v \in V(G_{0.9N}) } \mathsf{deg}_{G_{0.9N}}(v) \geq  \sum_{ v \in V(G_{0.9N}) \setminus A } \mathsf{deg}_{G_{0.9N}}(v) \\
        &\geq 3\times 0.9N + \sum_{ \substack{ v \in V(G_{0.9N}) \setminus A \\ v \not\in \{ v_1,\ldots,v_{0.9N} \} } } \mathsf{deg}_{G_{0.9N}}(v) \geq 2.7N + 2(|V(G_{0.9N})|-0.9N-|A|) \,.
    \end{align*}
    Thus, we have
    \begin{align*}
        \tau(G_{0.9N}) &= |E(G_{0.9N})|-|V(G_{0.9N})| \geq \frac{0.9N-2|A|}{2} \geq \frac{N}{3} \\
        &\geq \frac{1}{36\ell} \cdot |V(G_{0.9N})| \overset{\eqref{eq-assump-C-d-epsilon},\eqref{eq-def-ell-p}}{>} \frac{ 10\log d }{ \log(1/(120\ell p)) } \cdot |V(G_{0.9N})| \,,
    \end{align*}
    thereby contradicting Lemma~\ref{lem-small-set-edge-density} applied to the induced graph on $V(G_{0.9N})$, whose excess is at least $\tau(G_{0.9N})$.
\end{proof}

\begin{lemma}{\label{lem-bound-num-path-after-bootstrap}}
    Suppose $G \sim\mathcal G_d(n)$ where $d \in\mathbb N$ satisfies \eqref{eq-assump-C-d-epsilon}. Then with probability $1-o(1)$, for all $A \subseteq V(G)$ such that $1\le |A|=pn$ where $p\leq d^{-4000}$, we have $|\mathcal P(A_{\mathsf{f}})| \leq 100pn$ and $|(A_{\mathsf f})^+|\le200\ell pn$. 
\end{lemma}
\begin{proof}
    We work on the events in Lemmas~\ref{lem-small-set-edge-density} and \ref{lem-bootstrap-stop}. Since $p\le d^{-4000}$, we have $\ell\ge4$. We first record a deterministic consequence of $(A_{\mathsf f})^\circ=A_{\mathsf f}$. Consider the trace of one excursion in $\mathcal P(A_{\mathsf f})$, or the union of the traces of two excursions. If a vertex $w\notin A_{\mathsf f}$ has three incident edges in this union, each such edge can be followed along a suffix of one excursion, or a reversed prefix, to reach $A_{\mathsf f}$. The resulting three non-backtracking walks have different first edges and length at most $2\ell+2\le4\ell$, contradicting closure. On the other hand, every internal vertex of an excursion has trace degree at least two, by non-backtracking. Hence every excursion is either a simple path between two distinct vertices of $A_{\mathsf f}$, or a simple cycle meeting $A_{\mathsf f}$ at exactly one vertex. Distinct excursions have disjoint interiors and disjoint edge sets: if two shared an internal vertex, their continuations in both directions along the degree-two trace would be identical, up to reversal.

    Denote by $N=50pn$. From Lemma~\ref{lem-bootstrap-stop} we know that $|A_{\mathsf{f}}|\leq 0.2N$. Suppose on the contrary that $|\mathcal P(A_{\mathsf{f}})| \geq N$. Then there exist distinct excursions $P_1,\ldots,P_N$ such that $V(P_i) \cap A_{\mathsf{f}} = \mathsf{EndP}(P_i)$ and $|E(P_i)|\leq 2\ell+2$. Let $H$ consist of all vertices of $A_{\mathsf f}$ and the trace edges of these $N$ selected excursions; vertices of $A_{\mathsf f}$ not used by the selected excursions are included as isolated vertices. In particular, $H$ is not the full enlargement, and we do not add all the edges of $G_{A_{\mathsf f}}$ separately. Thus, we have
    \begin{align*}
        |V(H)| & =|A_{\mathsf f}|+\sum_{i=1}^N \bigl(|E(P_i)|-1\bigr) \leq (2\ell+1)N+|A_{\mathsf f}| \le(100\ell+60)pn \le200\ell pn\le d^{-4}n\,; \\
        \tau(H) & =\sum_{i=1}^N|E(P_i)|-|V(H)|=N-|A_{\mathsf f}| \geq 0.8N = 40pn \,.
    \end{align*}
    By \eqref{eq-def-ell-p}, $\log(1/(200\ell p))\ge500\ell\log d$ for sufficiently large $d$. Therefore Lemma~\ref{lem-small-set-edge-density}, applied to the induced graph on $V(H)$, gives
    \[
        \tau(H)\le\tau(G_{V(H)})\le\frac{10\log d}{\log(n/|V(H)|)}|V(H)|\le\frac{|V(H)|}{50\ell}\le4pn.
    \]
    This contradicts $\tau(H)\ge40pn$. Thus $|\mathcal P(A_{\mathsf f})|<50pn$, proving the asserted weaker bound. Finally, summing the disjoint excursion interiors yields
    \[
        |(A_{\mathsf f})^+|\le|A_{\mathsf f}|+(2\ell+1)|\mathcal P(A_{\mathsf f})|\le(100\ell+60)pn\le200\ell pn. \qedhere
    \]
\end{proof}
Finally, we prove that Definition~\ref{def-typical-graph}(6) holds for $G \sim \mathcal G_d(n)$ with probability $1-o(1)$, as incorporated in the following lemma.
\begin{lemma}{\label{lem:enu-num-returning-path}}
    Suppose $G \sim\mathcal G_d(n)$ where $d \in\mathbb N$ satisfies \eqref{eq-assump-C-d-epsilon}. Then with probability $1-o(1)$, for all $A,A' \subseteq V(G)$, we have $|W_{m}(A,A')| \leq 2\sqrt{|A||A'|} d^{0.51m}$ for $1 \le m\leq \frac{\min\{ \log(1/p),\log(1/q) \} }{ 10\log d }$.
\end{lemma}
\begin{proof}
    Let \(M=M_G\) be the adjacency matrix of \(G\), and let \(d=\lambda_1(G)\ge\lambda_2(G)\ge\cdots\ge\lambda_n(G)\) be its eigenvalues. By Friedman's theorem \cite{Friedman08} and \eqref{eq-assump-C-d-epsilon}, 
    \begin{align*}
        \Pb_{G \sim \mathcal G_d(n)}\Big( |\lambda_2(G)|,|\lambda_n(G)|\le d^{0.51} \Big) =1-o(1) \,.
    \end{align*}
    It remains to show that $|\lambda_2(G)|,|\lambda_n(G)| \le d^{0.51}$ implies that $|W_m(A,A')| \leq 2\sqrt{|A||A'|}d^{0.51m}$ for all $A,A' \subseteq V(G)$ and $1 \le m\leq \frac{\min\{ \log(1/p),\log(1/q) \} }{ 10\log d }$. Fix nonempty sets $A,A'\subseteq V(G)$, put $p=|A|/n$ and $q=|A'|/n$. The number of walks under consideration is $W_{m}(A,A')=\mathbf 1_A^{\top}M^m\mathbf 1_{A'}$. Decompose the two indicator vectors into their constant and mean-zero parts: $\mathbf 1_A=p\mathbf 1+u, \mathbf 1_{A'}=q\mathbf 1+v$ where $u,v \perp \mathbf 1$. Since $M\mathbf 1=d\mathbf 1$, the cross terms vanish and hence
    \[
        W_{m}(A,A')=npq\cdot d^m+u^{\top}M^m v \le |A||A'|d^{m}/n + u^{\top}M^m v \,.
    \]
    Moreover,
    \[
        \|u\|_2^2=np(1-p)\leq np=|A|, \quad \|v\|_2^2=nq(1-q)\leq nq=|A'| \,.
    \]
    The restriction of $M$ to $\mathbf 1^\perp$ has operator norm $\lambda_\star(G)=\max\{ |\lambda_2(G)|,|\lambda_n(G)| \} \le d^{0.51}$, so the Cauchy--Schwarz inequality yields
    \[
        |u^{\top}M^mv| \leq d^{0.51m} \|u\|_2\|v\|_2 \leq d^{0.51m} \sqrt{|A||A'|}  \,.
    \]
    We have therefore proved the deterministic estimate
    \[
        W_{m}(A,A') \leq \left( d^{m}|A||A'|/n + d^{0.51m} \sqrt{|A||A'|} \right) \le 2d^{0.51m}\sqrt{|A||A'|} \,,
    \]
    where the last inequality follows from $1 \le m\leq \frac{\min\{ \log(1/p),\log(1/q) \} }{ 10\log d }$.
\end{proof}
We finish this section with some facts on $A_{\mathsf{f}}$. Recall the isomorphism classes $\mathbb T_{k,\ell}$ defined in Definition~\ref{def:specific-tree}.
\begin{cor}{\label{cor-structure-A-final}}
    The set $A_{\mathsf{f}}$ satisfies the following conditions whenever $G$ is simple and $d$-regular and $\ell\ge4$:
    \begin{itemize}
        \item[(1)] For any $v_1,v_2 \in A_{\mathsf{f}}$ and for any $u_i \in \mathcal N_G(v_i) \setminus (A_{\mathsf{f}})^{+}$ with $(v_1,u_1) \neq (v_2,u_2)$, we have 
        \begin{align*}
            \mathcal B\Big( u_1,G_{\setminus \{v_1\}},\ell \Big),\ \mathcal B\Big( u_2,G_{\setminus \{v_2\}},\ell \Big) \mbox{ are pairwise vertex-disjoint trees} \,.
        \end{align*}
        \item[(2)] For any $v_1,v_2 \in (A_{\mathsf{f}})^+$ with $\mathsf{dist}(v_i,A_{\mathsf{f}})=m_i$ and for any $u_i \in \mathcal N_{G}(v_i) \setminus (A_{\mathsf{f}})^{+}$ with $(v_1,u_1) \neq (v_2,u_2)$, we have
        \begin{align*}
            \mathcal B\Big( u_1,G_{\setminus \{v_1\}},\ell \Big),\ \mathcal B\Big( u_2,G_{\setminus \{v_2\}},\ell \Big) \mbox{ are pairwise vertex-disjoint trees} \,.
        \end{align*}
        \item[(3)] For any $v \in (A_{\mathsf{f}})^+$ and for any $u \sim v, u \notin (A_{\mathsf f})^+$, define $\mathbf T_u$ to be the graph induced by all the vertices that can be connected to $u$ by a path of length at most $\ell$ that does not intersect $(A_{\mathsf{f}})^+$. Then the trees $\{ \mathbf T_u \}$ are pairwise vertex-disjoint. In addition, we have $\mathbf T_u \sim \mathbb T_{1,\ell}$, and
        \[
            \mathbf T_u=\mathcal B\bigl(u,G_{\setminus(A_{\mathsf f})^+},\ell\bigr)=\mathcal B\bigl(u,G_{\setminus\{v\}},\ell\bigr).
        \]
        The only edge joining a vertex of $\mathbf T_u$ to $(A_{\mathsf f})^+$ is $uv$.
    \end{itemize}
\end{cor}
\begin{proof}
    Write $S=A_{\mathsf f}$ and $B=(A_{\mathsf f})^+$. The deterministic argument at the beginning of the proof of Lemma~\ref{lem-bound-num-path-after-bootstrap} shows that the excursions in $\mathcal P(S)$ are simple paths or simple cycles, with pairwise disjoint interiors; that argument uses only closure and $2\ell+2\le4\ell$, not the size of $A$. Every $v\in B\setminus S$ therefore has two distinct first-edge directions along its excursion leading to $S$. The corresponding lengths sum to at most $2\ell+2$, so the shorter is at most $\ell+1$. In particular, every vertex of $B$ has a path to $S$ inside $G_B$ of length at most $\ell+1$. Moreover, every $v\in B\setminus S$ has exactly two neighbors in $B$: a third edge into $B$, followed by a shortest excursion segment to $S$, would give a third non-backtracking witness of length at most $\ell+2$. There is no immediate reversal at the joining vertex, since an additional edge into an excursion interior cannot be one of that excursion's two edges, by interior-disjointness.

    We first exclude a non-backtracking walk of length between $2$ and $2\ell+2$ with endpoints in $B$ and all internal vertices outside $B$. If both endpoints lie in $S$, the walk is a first-return excursion in $\mathcal P(S)$, contradicting the definition of $B$. Otherwise, reverse the walk if needed so that its starting vertex $v$ belongs to $B\setminus S$, and append a path inside $G_B$ from the other endpoint to $S$. The concatenation is non-backtracking because the arriving edge lies outside $B$ and the departing edge lies inside $B$. Its length is at most $3\ell+3\le4\ell$, and its first edge leaves $B$. Together with the two directions along the excursion containing $v$, it would give three bootstrap witnesses and force $v\in S$, a contradiction. The endpoints of the excluded walk are allowed to coincide.

    We now prove Items~(1) and (2) simultaneously, first using the balls in $G_{\setminus B}$. Suppose on the contrary that
    \begin{align*}
        \mathcal B\Big( u_1,G_{\setminus B},\ell \Big) \cap \mathcal B\Big( u_2,G_{\setminus B},\ell \Big) \neq \emptyset \,.
    \end{align*}
    A shortest exterior path between $u_1$ and $u_2$, together with the attachment edges $v_1u_1$ and $u_2v_2$, gives an excluded non-backtracking walk of length at most $2\ell+2$. This also covers $v_1=v_2$: the attachment edges are distinct, so the walk is still non-backtracking. If $u_1=u_2$, the two attachment edges alone give the excluded length-two walk. Thus the exterior balls are pairwise vertex-disjoint.

    Now suppose on the contrary that $\mathcal B( u_1,G_{\setminus B},\ell )$ has a cycle $C$. Take a non-tree edge in a breadth-first spanning tree rooted at $u_1$. The tree paths from $u_1$ to the endpoints of that edge, followed by the edge and a return along the tree, form a non-backtracking closed walk of length at most $2\ell+1$. Adding the attachment edge $v_1u_1$ at its beginning and end gives a walk from $v_1$ back to $v_1$, through the exterior of $B$, of length at most $2\ell+3$.

    If $v_1\in B\setminus S$, append its shorter excursion segment to $S$. This gives a non-backtracking walk of length at most $3\ell+4\le4\ell$ whose first edge leaves $B$. Together with the two excursion directions at $v_1$, it forces $v_1\in S$, a contradiction. If $v_1\in S$, write the breadth-first witness as a stem from $v_1$ to a cycle vertex $w$, one traversal of the simple cycle, and the reverse stem. Let $a\ge1$ be the stem length and $c$ the cycle length, so that $2a+c\le2\ell+3$. From $w$, the stem toward $v_1$ and the two full cycle traversals followed by the stem are three non-backtracking walks to $S$ with distinct first edges, of lengths at most $a+c\le2\ell+2\le4\ell$. This would put $w\in S$, although the cycle lies outside $B$. Thus the exterior balls are trees.

    There is no additional edge from $\mathcal B(u_1,G_{\setminus B},\ell)$ into $B$: the tree path starting with $v_1u_1$, followed by such an edge, would be an excluded exterior walk of length at most $\ell+2$. Likewise, a path of length at most $\ell$ from $u_1$ in $G_{\setminus\{v_1\}}$ cannot visit $B$. Its prefix up to the first visit, preceded by $v_1u_1$, would be an excluded exterior walk of length at most $\ell+1$. Consequently,
    \[
        \mathcal B\bigl(u_1,G_{\setminus\{v_1\}},\ell\bigr)=\mathcal B\bigl(u_1,G_{\setminus B},\ell\bigr).
    \]
    The same identity holds for $u_2,v_2$, proving Items~(1) and (2) as stated.

    Based on Items~(1) and (2), the graphs $\mathbf T_u$ in Item~(3) are these pairwise disjoint exterior trees, each with its unique attachment edge $uv$. Since $G$ is $d$-regular, the root has $d-1$ children and every non-root vertex at depth less than $\ell$ has one parent and $d-1$ children. All of these neighbors belong to the radius-$\ell$ ball, and none can give an additional edge inside the tree or into $B$. Hence $\mathbf T_u\simeq\mathbb T_{1,\ell}$, proving Item~(3).
\end{proof}
\begin{figure}[!ht]
    \centering
    \vspace{0cm}
    \includegraphics[width=0.7\linewidth]{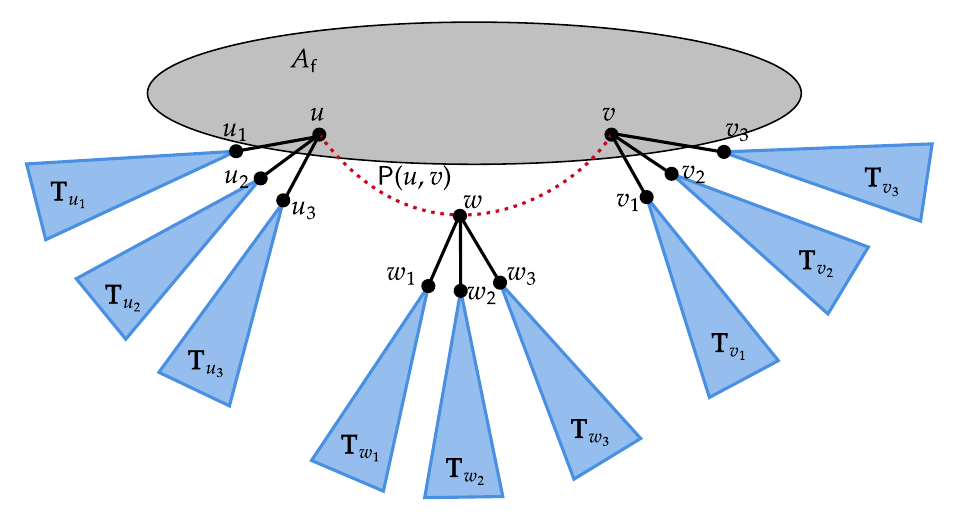}
    \caption{\noindent An illustration of the structure of $A_{\mathsf f}$. Here $u,v \in A_{\mathsf f}$, the red dashed curve is the path $\mathsf P(u,v) \in \mathcal P(A_{\mathsf f})$, $w \in V(\mathsf P(u,v))$, and $\{ u_i \},\{ v_i \},\{ w_i \}$ are neighbors of $u,v,w$, respectively. The blue triangles are the trees $\{ \mathbf T_{u_i} \}, \{ \mathbf T_{v_i} \},\{ \mathbf T_{w_i} \}$. }
    \label{fig:A-decomposition}
\end{figure}

\section{Initial analysis of disagreement clusters}{\label{sec:initial-analysis}}

This section establishes the preliminary sparsity estimate \eqref{eq-initial-sparsity}. We record the relevant events in the following lemma.
\begin{lemma}{\label{lem-bound-initial-steps}}
    Fix $G \in \mathcal G$ and an initial configuration $X_0=x_0$ with $m(x_0)\geq\epsilon$. Define the event $\mathcal E_{\ref{lem-bound-initial-steps}}=\mathcal E_{\star} \cap \mathcal E_{\star\star}$, where
    \begin{align}
        &\mathcal E_{\star} = \Big\{ |\mathcal B_t| \leq e^{-0.004Cd}n \mbox{ for all } 0 \leq t \leq n \Big\} \,; \label{eq-def-E-diamond-1} \\
        &\mathcal E_{\star\star}=\Big\{ |\mathcal D_t| \leq 4e^{-0.002Cd}n \mbox{ for all } t_0 \leq t \leq n \Big\} \,. \label{eq-def-E-diamond-2}
    \end{align}
    Here $t_0=t_0(C,d)$ is a sufficiently large constant. We then have $\Pb(\mathcal E_{\ref{lem-bound-initial-steps}}) \ge 1-n^{-12}$ for all $G \in \mathcal G$.
\end{lemma}
We prove Lemma~\ref{lem-bound-initial-steps} in Section~\ref{subsec:proof-initial-sparsity}, using the bad-cluster estimates developed in Section~\ref{subsec:subcritial-bad-cluster}.

\subsection{Subcriticality of bad clusters}{\label{subsec:subcritial-bad-cluster}}

We first argue that the bad cluster $\mathcal B_t$ behaves subcritically under $\pi_+$, and the core is to show the following lemma.
\begin{lemma}{\label{lem-FK-non-giant-cluster-bound}}
    Fix $G \in \mathcal G$, and consider the random-cluster measure on $G$ with $q=2$ and $p=1-e^{-2\beta}$ where $\beta=\frac{C}{d}$ with $C,d$ satisfying \eqref{eq-assump-C-d-epsilon}. For any $A \subset V(G)$, define 
    \begin{align*}
        A_{\mathsf{RC}}:= \cup_{v \in A} \mathcal C(v) \,,
    \end{align*}
    where $\mathcal C(v)$ is the open cluster of $v$ in the random-cluster representation. Then, for any fixed $A \subset V(G)$ such that $|A|\le 10^{-4}n$, we have 
    \begin{align}
        &\pi_G^{\mathsf{RC}}\Big( A_{\mathsf{RC}} \cap \mathcal C_{1}=\emptyset \Big) \leq 2e^{-0.2C|A|} \,. \label{eq-prob-not-in-giant} 
    \end{align}
\end{lemma}
\begin{proof}
    The assertion is immediate for $A=\emptyset$. Put $k=|A|\ge1$ and $\rho=p/(p+2(1-p))=\tanh(C/d)$. Every single-edge conditional opening probability belongs to $[\rho,p]$. By successive conditioning, for disjoint edge sets $F,H$,
    \[
        \pi_G^{\mathsf{RC}}(F\text{ open},H\text{ closed})\le p^{|F|}(1-\rho)^{|H|}.
    \]
    On $\{A_{\mathsf{RC}}\cap\mathcal C_1=\emptyset,\ |\mathcal C_1|>0.99n\}$, set $S=A_{\mathsf{RC}}$ and $s=|S|$. Then $k\le s<0.01n$, every edge of $\partial S$ is closed, and there is an open spanning forest of $S$ with exactly $k$ components, each containing exactly one marked vertex of $A$. Indeed, split an open spanning tree in each component by successively separating a leaf of the minimal subtree joining its marked vertices. The resulting forest has $s-k$ edges. Order its roots by the fixed labels of $A$. Summing Lemma~\ref{lem-enu-tree-on-regular-graph} over the positive compositions of $s$ into $k$ parts bounds the number of candidate forests by $\binom{s-1}{k-1}(4d)^{s-k}$. By Definition~\ref{def-typical-graph}(3), $|\partial S|\ge ds/4$. Consequently,
    \begin{align*}
        \pi_G^{\mathsf{RC}}(A_{\mathsf{RC}}\cap\mathcal C_1=\emptyset,\ |\mathcal C_1|>0.99n) \le\sum_{s=k}^{\infty}\binom{s-1}{k-1}(4dp)^{s-k}(1-\rho)^{ds/4} =\left(\frac{(1-\rho)^{d/4}}{1-4dp(1-\rho)^{d/4}}\right)^k.
    \end{align*}
    For sufficiently large $C$ the last display is at most $e^{-0.21Ck}$. Definition~\ref{def-typical-graph}(1), through Lemma~\ref{lem-FK-giant-size}, bounds the complementary event $\{|\mathcal C_1|\le0.99n\}$ by $e^{-10^{-4}Cn}$. Since $k\le10^{-4}n$, this is at most $e^{-Ck}$. Adding the two bounds proves \eqref{eq-prob-not-in-giant}.
\end{proof}

\begin{lemma}{\label{lem-est-bad-minus-cluster-fixed-time}}
    Fix $G \in \mathcal G$ and let $\sigma \sim \pi_{G,+}$. Define $\mathcal M$ to be the collection of $v \in V$ such that $\sigma(v)=-1$. Then, for all $|M|\leq C^{-2}n$ we have
    \begin{align}
        \pi_+\big( M \subseteq \mathcal M \big) \leq 3e^{-C|M|/9} \,.  \label{eq-prob-set-is-minus}
    \end{align}
    In addition, define $\mathcal B^\sharp$ to be the collection of $v \in V$ such that $\sum_{u \sim v} \sigma(u) \le 0.6d$. Then for all $|B|\leq e^{-0.02d}n$,
    \begin{align}
        \pi_+\big( B \subseteq \mathcal B^\sharp \big) \leq 3e^{-0.01Cd|B|} \,.  \label{eq-prob-set-is-bad}
    \end{align}
\end{lemma}
\begin{proof}
    Under the Edwards--Sokal law conditioned on the largest cluster being plus, the bond marginal remains $\pi_G^{\mathsf{RC}}$: conditional on every bond configuration the selected largest cluster has a fair independent sign. If $M$ consists of minus spins, it therefore misses the largest cluster. Lemmas~\ref{lem-FK-non-giant-cluster-bound} and \ref{lem-TV-dist-giant-is-plus} give, for $1\le|M|\le C^{-2}n$,
    \[
        \pi_+(M\subseteq\mathcal M) \le2e^{-0.2C|M|}+4e^{-10^{-4}Cn} \le3e^{-C|M|/9} \,.
    \]
    The empty-set case is immediate.

    For the second assertion suppose $B\ne\emptyset$ and let $N_-:=\{u:\sigma(u)=-1,\ u\sim v\text{ for some }v\in B\}$ and $S=B\cup N_-$. On $B\subseteq\mathcal B^\sharp$, there are at least $0.2d|B|$ oriented incidences from $B$ to $N_-$. Each edge is counted at most once except for edges internal to $B$. Hence
    \[
        |E(G_S)|\ge0.2d|B|-|E(G_B)|.
    \]
    Both $|B|\le e^{-0.02d}n$ and $|S|\le(d+1)|B|$ are small enough that Definition~\ref{def-typical-graph}(4) gives $|E(G_B)|\le1.001|B|$ and $|E(G_S)|\le1.001|S|$, after increasing $d_0$. It follows that
    \[
        |N_-|\ge |S|-|B|\ge\frac{0.2d-1.001}{1.001}|B|-|B|\ge0.19d|B| \,.
    \]
    Choose $m=\lceil0.18d|B|\rceil$. For sufficiently large $n$, $m\le0.19d|B|$ and $m\le C^{-2}n$. Thus some $m$-element subset of $\mathcal N_G(B)$ consists entirely of minus spins. There are at most $2^{d|B|}$ such subsets, so the first assertion yields
    \[
        \pi_+(B\subseteq\mathcal B^\sharp) \le 2^{d|B|}\,3e^{-Cm/9} \le3\exp\{-(0.02C-\log2)d|B|\} \le3e^{-0.01Cd|B|} \,.
    \]
    Again the empty-set case is immediate.
\end{proof}
\begin{remark}{\label{remark:same-est-extend-to-tree}}
The same estimates, without the prefactor $3$, hold under $\phi_{k,\ell}^+$ for $k\le10^4$, with $M$ a set of free vertices and $B$ a set of free vertices whose neighbor sums include all fixed boundary spins:
\[
 \phi_{k,\ell}^+(M\subseteq\mathcal M)\le e^{-0.1C|M|},
 \qquad
 \phi_{k,\ell}^+(B\subseteq\mathcal B^\sharp)\le e^{-0.01Cd|B|} \,.
\]
For completeness, the free spins in each of these trees stochastically dominate independent spins with minus probability $e^{-C}$. In a regular descendant branch with plus leaves, induction in the exact cavity recursion gives every upward free-vertex field at least $0.9C$: for large $C,d$, $(d-1)\psi_\beta(0.9C)\ge0.9C$, and the boundary messages are $+\infty$. Thus a branch vertex, even conditional on a minus parent, has total field at least $0.9C-\beta\ge0.8C$. The root of $\mathbb T_{k,\ell}$ has field at least $(d-k)\psi_\beta(0.9C)-k\beta\ge0.8C$. Sample the root first and then the descendants in their tree order; every conditional minus probability in this sampling is at most $e^{-1.6C}\le e^{-C}$. This constructs the asserted product domination. 

The first displayed estimate follows immediately. For the second, a bad free vertex has at least $0.2d$ minus neighbors. Fixed-minus attachments account for at most $10^4|B|$ incidences. Let $N_-$ be the remaining free minus neighbors and put $S=B\cup N_-$. The forest inequalities $|E(S)|\le|S|$ and $|E(B)|\le|B|$ then imply $|N_-|\ge(0.2d-10^4-2)|B|\ge0.18d|B|$. A union bound over subsets of the at most $d|B|$ neighbors and product domination give
\[
 \phi_{k,\ell}^+(B\subseteq\mathcal B^\sharp)
 \le2^{d|B|}e^{-0.18Cd|B|}
 \le e^{-0.01Cd|B|} \,.
\]
The same proof applies when there are no free vertices, by the empty-set convention.
\end{remark}

\subsection{Proof of Lemma~\ref{lem-bound-initial-steps}}{\label{subsec:proof-initial-sparsity}}

The goal of this section is to prove Lemma~\ref{lem-bound-initial-steps}. Denote $q=e^{-0.002Cd}$ and $q'=e^{-0.02C}$. We first prove that $\Pb(\mathcal E_{\star})=1-\exp(-n^{1-o(1)})$ via the following lemma.
\begin{lemma}{\label{lem-uniform-bound-bad-over-time}}
    Define $\mathcal M_t=\{ v \in V: \widehat{X}^{\pi_+}_t(v)=-1 \}$, and define
    \begin{equation*}
        \mathcal E_{\ref{lem-uniform-bound-bad-over-time}} := \Big\{ |\mathcal M_t|\leq q'n,\ |\mathcal B_t|\leq q^2 n \text{ for every }0\leq t\leq n \Big\} \,.
    \end{equation*}
    We then have $\Pb(\mathcal E_{\ref{lem-uniform-bound-bad-over-time}})\geq 1-e^{-\Omega(n)}$. In particular, this implies that $\Pb(\mathcal E_{\star})=1-\exp(-\Omega(n))$. 
\end{lemma}
\begin{proof}
    Put $k_M=\lceil q'n\rceil$ and $k_B=\lceil q^2n\rceil$. For sufficiently large $n$, $k_M\le C^{-2}n$ and $k_B\le e^{-0.02d}n$, so Lemma~\ref{lem-est-bad-minus-cluster-fixed-time} applies at both sizes. Using its stated exponents and $\binom nk\le(en/k)^k$, we obtain
    \begin{align}
        \pi_+(|\mathcal M|\ge q'n) &\le3\binom n{k_M}e^{-Ck_M/9}  \le3e^{-(C/9-1-0.02C)k_M} \le3e^{-0.08Ck_M}=e^{-\Omega_{C,d}(n)} \,, \label{eq-minus-tail}\\
        \pi_+(|\mathcal B^\sharp|\ge q^2n) &\le3\binom n{k_B}e^{-0.01Cd k_B} \le3e^{-(0.01Cd-1-0.004Cd)k_B} \le3e^{-0.005Cd k_B}=e^{-\Omega_{C,d}(n)} \,.  \label{eq-bad-tail}
    \end{align}
    Here $\log(en/k_M)\le1+0.02C$ and $\log(en/k_B)\le1+0.004Cd$; the final inequalities hold for large $C,d$.
    Let $0=\tau_0<\tau_1<\tau_2<\cdots$ be the successive rings of the superposition of the $n$ rate-one clocks, and let $N(t)$ be the number of rings in $[0,t]$. Thus $N(n)\sim\mathsf{Pois}(n^2)$ and
    \begin{equation}{\label{eq-poisson-upper}}
        \Pb\bigl(N(n)>2n^2\bigr)\leq e^{-\Omega(n^2)} \,.
    \end{equation}
    The restricted chain is stationary with law $\pi_+$.  Consequently, each $\widehat{X}_{\tau_k}^{\pi_+}$ has law $\pi_+$.  Combining \eqref{eq-minus-tail}, \eqref{eq-bad-tail}, and \eqref{eq-poisson-upper}, and taking a union bound over $0\leq k\leq2n^2$ yields $\Pb(\mathcal E_{\ref{lem-uniform-bound-bad-over-time}})\geq 1-e^{-\Omega(n)}$. 
\end{proof}

We then prove $\Pb(\mathcal E_{\star\star})=1-\exp(-\Omega_{C,d}(n))$. Based on Lemmas~\ref{lem-equal-to-restricted-dynamics} and \ref{lem-uniform-bound-bad-over-time}, it suffices to show that $\Pb(\mathcal E_{\star\star}^c; \mathcal E_{\ref{lem-equal-to-restricted-dynamics}};\mathcal E_{\ref{lem-uniform-bound-bad-over-time}})\leq \exp(-\Omega_{C,d}(n))$. We first give a simple bound for $|\mathcal D_0|$.
\begin{lemma}{\label{lem-bound-D-0}}
    Suppose that the initial configuration is an arbitrary configuration
    \(X_0=x_0\in\{-1,+1\}^{V}\) satisfying $m(x_0)=\frac{1}{n}\sum_{v\in V}x_0(v)\ge \epsilon$. Uniformly over every such \(x_0\), we have
    \begin{align*}
        \Pb\Big( |\mathcal D_0| > (\tfrac{1}{2}-\tfrac{\epsilon}{3})n; \mathcal E_{\ref{lem-equal-to-restricted-dynamics}}; \mathcal E_{\ref{lem-uniform-bound-bad-over-time}} \Big) = 0 \,.
    \end{align*}
\end{lemma}
\begin{proof}
    Fix an arbitrary configuration \(x_0\in\{-1,+1\}^{V}\) with \(m(x_0)\ge\epsilon\), and define its set of minus spins by $\mathcal S_0:=\{v\in V:x_0(v)=-1\}$. Then $m(x_0) \ge \epsilon$ implies that $|\mathcal S_0| \le \frac{(1-\epsilon)n}{2}$. Recall from Lemma~\ref{lem-uniform-bound-bad-over-time} that $|\mathcal M_0|\le e^{-0.02C}n$. Since the initial disagreement set is exactly the symmetric difference
    \[
        \mathcal D_0 = \big\{ v\in V:x_0(v)\neq X^{\pi^+}_0(v) \big\} = \mathcal S_0\triangle \mathcal M_0 \,.
    \]
    Therefore, on \(\mathcal E_{\ref{lem-uniform-bound-bad-over-time}}\),
    \begin{align*}
        |\mathcal D_0| =|\mathcal S_0 \triangle \mathcal M_0| \le |\mathcal S_0| + |\mathcal M_0| \le \left( \tfrac{1-\epsilon}{2} + e^{-0.02C} \right)n \overset{\eqref{eq-assump-C-d-epsilon}}{\le} (\tfrac{1}{2}-\tfrac{\epsilon}{3})n \,,
    \end{align*}
    finishing the proof of this lemma.
\end{proof}

Let $Z_k=|\mathcal D_{\tau_k}|$, $p_{-}=\frac{1}{2}-\frac{\epsilon}{3}$, and $p_+=\frac{1}{2}-\frac{\epsilon}{4}$. Fix $k$ and work conditionally on $\mathcal F_{\tau_k}$, given that the event $\mathcal E_{\ref{lem-equal-to-restricted-dynamics}} \cap \mathcal E_{\ref{lem-uniform-bound-bad-over-time}}$ holds up to time $\tau_k$. Abbreviate
\begin{align*}
    \mathcal D_k = \mathcal D_{\tau_k}, \quad \mathcal M_k=\mathcal M_{\tau_k}, \quad p_k = \frac{|\mathcal D_k|}{n} \,.
\end{align*}
For a configuration $\sigma \in \{ -1,+1 \}^V$ and $v \in V$, define 
\begin{align*}
    h(\sigma;v):=\sum_{w\sim v}\sigma(w), \quad g(h):=\frac{e^{\beta h}}{e^{\beta h}+e^{-\beta h}} =\frac{1}{1+e^{-2\beta h}} \,.
\end{align*}
If the next ring is at $v$, then under the common update uniform random variable the probability that the two newly sampled spins disagree is
\begin{equation}{\label{eq-av-definition}}
    a(v):= \Big| g\big(h(X_{\tau_k};v)\big) -g\big(h(X_{\tau_k}^{\pi_+};v)\big) \Big| \,.
\end{equation}
Since the update vertex is uniform on $V$,
\begin{align}
    \Pb( Z_{k+1}=Z_k-1 \mid \mathcal F_{\tau_k} ) &=\frac{1}{n} \sum_{v\in \mathcal D_k}(1-a(v)) \,;  \label{eq-down-transition}\\
    \Pb( Z_{k+1}=Z_k+1 \mid \mathcal F_{\tau_k} ) &=\frac{1}{n} \sum_{v\notin \mathcal D_k} a(v) \,. \label{eq-up-transition}
\end{align}
The core is to show the following lemma on the drift of $Z_k$.
\begin{lemma}
    Write $\mathcal D=\mathcal D_k$, $\mathcal M=\mathcal M_k$, and $p=p_k$. Whenever $2q\leq p \leq p_+$, one has
    \begin{equation}{\label{eq-one-step-bias}}
        \Pb(Z_{k+1}=Z_k-1\mid\mathcal F_{\tau_k})\geq0.8p \,; \quad  \Pb(Z_{k+1}=Z_k+1\mid\mathcal F_{\tau_k})\leq0.2p \,.
    \end{equation}
\end{lemma}
\begin{proof}
    The proof is divided into a dense and a sparse regime.

    \noindent{\bf Case~1:} $1000/C \leq p \leq p_+$. Define
    \begin{align*}
        &R_{\mathcal D} := \Big\{ v\in V:\deg_{\mathcal D}(v) \geq (p+0.01\epsilon)d \Big\} \,;  \\
        &R_{\mathcal M} := \Big\{ v\in V:\deg_{\mathcal M}(v)\geq (100C^{-1}+0.01\epsilon)d \Big\} \,.
    \end{align*}
    Definition~\ref{def-typical-graph}(2) gives
    \begin{equation}{\label{eq-dense-exceptional-bounds}}
        |R_{\mathcal D}|\leq \frac{ 4|\mathcal D| }{ d(0.01\epsilon)^2 } \overset{\eqref{eq-assump-C-d-epsilon}}{\leq} 10^{-4}|\mathcal D|, \quad |R_{\mathcal M}|\leq \frac{ 4|\mathcal M| }{ d(0.01\epsilon)^2 } \overset{\eqref{eq-assump-C-d-epsilon}}{\leq} 10^{-4}|\mathcal M| \,.
    \end{equation}
    For $v\not\in R_{\mathcal D}\cup R_{\mathcal M}$,
    \begin{equation*}
        h(X_{\tau_k}^{\pi_+};v) = d-2\deg_{\mathcal M}(v) \geq \big(1-200C^{-1} -0.02\epsilon \big)d \,.
    \end{equation*}
    Moreover, the two configurations have opposite spins precisely on $\mathcal D$, and hence
    \begin{align}
        h(X_{\tau_k};v) &= h(X_{\tau_k}^{\pi_+};v) -2\sum_{ \substack{ w\sim v \\ w\in \mathcal D } } X_{\tau_k}^{\pi_+}(w) \geq h(X_{\tau_k}^{\pi_+};v)-2\deg_{\mathcal D}(v)\notag\\
        &\geq \big(1-2p-200C^{-1}-0.04\epsilon \big)d \ge 0.4\epsilon d \,.  \label{eq:x-field-dense}
    \end{align}
    Here the last inequality follows from $p\leq p_+=\frac{1}{2}-\frac{\epsilon}{4}$, and \eqref{eq-assump-C-d-epsilon}. Both fields in \eqref{eq-av-definition} are therefore at least $0.4\epsilon d$.  Since $1-g(h)\leq e^{-2\beta h}$ for $h\geq0$, it follows that
    \begin{equation}{\label{eq-dense-av}}
        a(v)\leq e^{-0.8C\epsilon} \quad \mbox{for every} \quad v\notin R_{\mathcal D}\cup R_{\mathcal M} \,.
    \end{equation}
    Let $\rho:=\frac{|R_{\mathcal D}\cup R_{\mathcal M}|}{n}$. Because $p\geq 1000C^{-1}$, \eqref{eq-dense-exceptional-bounds} implies
    \begin{equation*}
        \rho\leq \frac{4(p+1000C^{-1})}{d(0.01\epsilon)^2}  \leq 1.1\times10^{-3}p \,.
    \end{equation*}
    Also, $e^{-0.8C\epsilon}\leq10^{-3}p$. Using \eqref{eq-down-transition} and retaining only vertices in $\mathcal D\setminus(R_{\mathcal D}\cup R_{\mathcal M})$,
    \begin{align*}
        \Pb(Z_{k+1}=Z_k-1\mid\mathcal F_{\tau_k}) \geq(p-\rho)\bigl(1-e^{-0.8C\epsilon}\bigr) \geq0.8p \,.
    \end{align*}
    Similarly, using \eqref{eq-up-transition}, bounding $a(v)$ by $1$ on the exceptional set and by \eqref{eq-dense-av} elsewhere,
    \begin{align*}
        \Pb(Z_{k+1}=Z_k+1\mid\mathcal F_{\tau_k}) \leq\rho+e^{-0.8C\epsilon} \leq 0.2p \,.
    \end{align*}
    This proves \eqref{eq-one-step-bias} in this case.

    \noindent{\bf Case~2:} $2q \leq p \leq 1000C^{-1}$. Define $R(\mathcal D):=\{v\in V:\deg_{\mathcal D}(v)\geq0.01d\}$. Since $|\mathcal D|\leq 1000C^{-1}n$, applying Definition~\ref{def-typical-graph}(2) we have 
    \begin{equation}{\label{eq-sparse-exceptional-bound}}
        |R(\mathcal D)|\leq \frac{4|\mathcal D|}{d(0.01-p)^2} \leq 10^{-3}|\mathcal D| \,.
    \end{equation}
    Consider first a vertex $v\in \mathcal D\setminus (\mathcal B_{\tau_k}\cup R(\mathcal D))$. Under the event $\mathcal E_{\ref{lem-equal-to-restricted-dynamics}}$, the ordinary and restricted plus chains agree. Thus $v\notin \mathcal B_{\tau_k}$ gives $h(X_{\tau_k}^{\pi_+};v)>0.5d$. Moreover, $v\notin R(\mathcal D)$ implies $\deg_{\mathcal D}(v)<0.01d$, and therefore
    \begin{equation}{\label{eq-sparse-positive-fields}}
        h(X_{\tau_k};v) \geq h(X_{\tau_k}^{\pi_+};v)-2\deg_{\mathcal D}(v) >0.48d \,.
    \end{equation}
    It follows that
    \begin{equation}{\label{eq-sparse-av-down}}
        a(v)\leq e^{-0.96C} \quad \mbox{for} \quad v\in \mathcal D\setminus (\mathcal B_{\tau_k}\cup R(\mathcal D)) \,.
    \end{equation}
    Since $|\mathcal B_{\tau_k}|\leq e^{-0.004Cd}n$, equations \eqref{eq-down-transition}, \eqref{eq-sparse-exceptional-bound}, and \eqref{eq-sparse-av-down} yield
    \begin{align}
        \Pb\big( Z_{k+1}=Z_k-1 \mid \mathcal F_{\tau_k} \big) \geq \big( p-e^{-0.004Cd}-10^{-3}p \big)\big( 1-e^{-0.96C} \big) \geq 0.8p \,.  \label{eq-sparse-down}
    \end{align}
    Here we used \eqref{eq-assump-C-d-epsilon} and
    \begin{equation*}
        \frac{e^{-0.004Cd}}{p}\leq\frac{e^{-0.004Cd}}{2q} =\frac12e^{-0.002Cd}\leq5\times10^{-4} \,.
    \end{equation*}
    We next bound upward moves. If $v\notin \mathcal D$ and $\deg_{\mathcal D}(v)=0$, then the two local fields are identical and $a(v)=0$. Suppose instead that
    \begin{align*}
        v\notin \mathcal D\cup \mathcal B_{\tau_k}\cup R(\mathcal D) \quad\text{and}\quad \deg_{\mathcal D}(v)>0 \,.
    \end{align*}
    Both endpoint fields in \eqref{eq-av-definition} are at least $0.48d$ by the same calculation as in \eqref{eq-sparse-positive-fields}.  Furthermore,
    \begin{align*}
        \Big| h(X_{\tau_k};v)-h(X_{\tau_k}^{\pi_+};v) \Big| \leq 2\deg_{\mathcal D}(v) \,.
    \end{align*}
    Since
    \begin{align*}
        g'(h) =\frac{2\beta e^{-2\beta h}}{(1+e^{-2\beta h})^2} \leq2\beta e^{-2\beta h} \,,
    \end{align*}
    the mean-value theorem gives
    \begin{equation}{\label{eq-sparse-av-up}}
        a(v) \leq4\beta e^{-0.96C}\deg_{\mathcal D}(v) =\frac{4C e^{-0.96C}}{d}\deg_{\mathcal D}(v) \,.
    \end{equation}
    Finally, we have the simple observation
    \begin{equation*}
        \sum_{v\in V}\deg_{\mathcal D}(v)=d|\mathcal D| \,.
    \end{equation*}
    Thus, by \eqref{eq-up-transition}, bounding $a(v)$ by $1$ on $\mathcal B_{\tau_k}\cup R(\mathcal D)$ and using \eqref{eq-sparse-av-up} elsewhere,
    \begin{align}
        \Pb\big( Z_{k+1}=Z_k+1\mid\mathcal F_{\tau_k} \big) &\leq e^{-0.004Cd} + 10^{-3}p + \frac{4Ce^{-0.96C}}{dn}\sum_{v\in V} \deg_{\mathcal D}(v) \nonumber  \\
        &\leq e^{-0.004Cd} + 10^{-3}p + 4Ce^{-0.96C}p \leq 0.2p \,.  \label{eq-sparse-up}
    \end{align}
    Equations \eqref{eq-sparse-down} and \eqref{eq-sparse-up} prove \eqref{eq-one-step-bias} in this regime. 
\end{proof}

We can now finish the proof of Lemma~\ref{lem-bound-initial-steps}.
\begin{proof}[Proof of Lemma~\ref{lem-bound-initial-steps}]
    Define the stopping time
    \begin{align*}
        \zeta = \min\Bigg\{ \inf_{k \geq 0} \Big\{ \widehat{X}_{\tau_k}^{\pi_+} \neq X_{\tau_k}^{\pi_+} \Big\}, \ \inf_{k \geq 0} \Big\{ |\mathcal B_{\tau_k}| > q^2n \Big\}, \ \inf_{k \geq 0} \Big\{ |\mathcal M_{\tau_k}| > q'n \Big\} \Bigg\} \,.
    \end{align*}
    Whenever $2q\leq Z_k/n \leq p_+$ holds and $k<\zeta$, equation \eqref{eq-one-step-bias} implies
    \begin{equation}{\label{eq-additive-drift}}
        \E\big[ Z_{k+1}-Z_k\mid\mathcal F_{\tau_k} \big] \leq \frac{-0.6Z_k}{n} \,.
    \end{equation}
    It also gives the exponential supermartingale estimate
    \begin{align}
        \E\Big[ 2^{Z_{k+1}-Z_k}\mid\mathcal F_{\tau_k} \Big] &=1 +\Pb(Z_{k+1}=Z_k+1\mid\mathcal F_{\tau_k}) -\frac12\Pb(Z_{k+1}=Z_k-1\mid\mathcal F_{\tau_k})\nonumber \\
        &\leq1+0.2p-0.4p\leq1 \,. \label{eq-exponential-supermartingale}
    \end{align}
    Set $L_n:=\lfloor 2qn\rfloor$ and $H_n:=\lceil p_+n\rceil$. Define the stopping time 
    \begin{align*}
        \zeta_L = \inf_{k \geq 0}\Big\{ Z_k \in (-\infty,L_n] \Big\}; \quad \zeta_H = \inf_{k \geq 0}\Big\{ Z_k \in [H_n,\infty) \Big\} \,.
    \end{align*}
    By \eqref{eq-exponential-supermartingale}, $\{ 2^{Z_k}:0 \leq k \leq \zeta \wedge \zeta_L \wedge \zeta_H \}$ is a nonnegative supermartingale. Hence, 
    \begin{align}
        \Pb\big( \zeta_H \leq \zeta_L \wedge \zeta,\ Z_0\le p_-n \big) \leq 2^{p_-n-H_n} \leq 2^{-(p_+-p_-)n+1} =2^{-\epsilon n/12+1} \leq \exp(-n^{1-o(1)}) \,. \label{eq-upper-barrier-probability}
    \end{align}
    This optional-stopping argument may be justified first at a bounded stopping time and then by letting the deterministic cutoff tend to infinity. We now show that the lower level is reached after $O(n)$ embedded updates. Define levels
    \begin{align*}
        a_0=p_-, \quad a_{j+1}=\max\Big\{ 2q,\frac{a_j}{2} \Big\} \,.
    \end{align*}
    In addition, let $J$ be the first index for which $a_J=2q$. Then
    \begin{equation}{\label{eq:number-of-levels}}
        J\leq 1+\log_2\Big(\frac{p_-}{2q}\Big) = O_{C,d}(1), \quad a_j \leq 2a_{j+1} \,.
    \end{equation}
    Define $\zeta_j = \inf_{k \geq 0} \{ Z_k \leq a_j n \}$. Suppose a stage starts at an embedded time $s$ with $Z_s\leq a_j n$. By \eqref{eq-additive-drift}, the supermartingale $\{ Z_k: s \leq k \leq \zeta\wedge\zeta_H \wedge \zeta_{j+1} \}$ gives conditional drift at most $-0.6a_{j+1}$ per update. Let $\Delta_i:=\mathbf1_{\{i<\zeta\wedge\zeta_H\wedge\zeta_{j+1}\}} (Z_{i+1}-Z_i)$ and $\xi_i:=\Delta_i-\E[\Delta_i\mid\mathcal F_{\tau_i}]$. Then the partial sums of the $\xi_i$ form a martingale and $|\xi_i|\leq2$. If the next level has not been reached within $4n$ updates, and neither the upper barrier nor $\zeta$ has been reached, then $Z_{s+4n}-Z_s>-a_{j+1}n$, whereas the accumulated predictable drift is at most $-0.6a_{j+1}(4n)=-2.4a_{j+1}n$. The martingale part must therefore be at least $1.4a_{j+1}n$. Azuma-Hoeffding's inequality yields
    \begin{equation}{\label{eq-stage-time}}
        \Pb\Big( \zeta\wedge\zeta_H \wedge \zeta_{j+1} > s+4n \mid \mathcal F_{\tau_s} \Big) \leq e^{-\Omega_{C,d}(n)}  \,.
    \end{equation}
    The same estimate holds when $s$ is a stopping time, by conditioning on $\mathcal F_{\tau_s}$. Since $a_{j+1}\geq2q$ and there are only $J=O_{C,d}(1)$ stages, a union bound gives
    \begin{equation}{\label{eq-embedded-hitting-time}}
        \Pb\Big( \zeta_L\wedge\zeta_H\wedge\zeta>4Jn,\ Z_0\le p_-n \Big) \leq \exp(-n^{1-o(1)}) \,.
    \end{equation}
    Choose $t_0=8J$. Since $N(t_0)\sim\mathsf{Pois}(nt_0)=\mathsf{Pois}(8Jn)$, a Poisson lower-tail bound gives
    \begin{equation}{\label{eq-poisson-lower}}
        \Pb\big( N(t_0)<4Jn \big) \leq e^{-\Omega_{C,d}(n)} \,.
    \end{equation}
    On the event in Lemma~\ref{lem-uniform-bound-bad-over-time} and the phase-restriction event, $Z_0\le p_-n$ and $\zeta>N(t_0)$. If also $N(t_0)\ge4Jn$ but $\zeta_L>N(t_0)$, then either $\zeta_H\le\zeta_L\wedge\zeta$ or $\zeta_L\wedge\zeta_H\wedge\zeta>4Jn$. Consequently, Lemmas~\ref{lem-bound-D-0}, \ref{lem-equal-to-restricted-dynamics}, and \ref{lem-uniform-bound-bad-over-time}, together with \eqref{eq-upper-barrier-probability}, \eqref{eq-embedded-hitting-time}, and \eqref{eq-poisson-lower}, imply
    \begin{equation}{\label{eq:reach-small-set}}
        \Pb\left( \inf_{0\leq t\leq t_0}|\mathcal D_t|>2qn \right) \leq e^{-\Omega_{C,d}(n)} \,.
    \end{equation}
    Equivalently, except on an exponentially unlikely event, the disagreement set reaches size at most $2qn$ no later than time $t_0$.

    Finally, we argue that $Z_k$ stays below $4qn$ after $t_0$. Let $U_n:=\lfloor4qn\rfloor+1$. For all sufficiently large $n$, $L_n+1<U_n<H_n$.  After the first time at which $Z_k\leq L_n$, every path to $U_n$ must contain an excursion that starts at $L_n+1$ and reaches $U_n$ before returning to $(-\infty,L_n]$. Consider any such excursion and condition on its starting history. Apply the stopped supermartingale from \eqref{eq-exponential-supermartingale}, now stopping on first hitting $(-\infty,L_n]$, first hitting $[U_n,\infty)$, or reaching $\zeta$. Optional stopping gives
    \begin{align}
        \Pb_{L_n+1}\big( \zeta_{U}<\zeta_L\wedge \zeta \big) \leq 2^{L_n+1-U_n} \leq e^{-\Omega(1)\cdot qn} \,. \label{eq:excursion-probability}
    \end{align}
    On the event $N(n)\leq2n^2$, there are at most $2n^2$ excursion starts before time $n$. The strong Markov property and a union bound using \eqref{eq:excursion-probability} therefore give
    \begin{align}
        &\Pb\left( \sup_{t_0\leq t\leq n}|\mathcal D_t|>4qn, \zeta>N(n), \inf_{0\leq t\leq t_0}|\mathcal D_t|\leq2qn \right) \nonumber \\
        &\leq2n^2e^{-\Omega(1) \cdot qn}+\Pb(N(n)>2n^2) \leq \exp(-n^{1-o(1)}) \,. \label{eq:persistence}
    \end{align}
    The last step uses that $q=e^{-0.002Cd}$ is a positive constant once $C$ and $d$ are fixed. Combining \eqref{eq:reach-small-set}, \eqref{eq:persistence}, and Lemma~\ref{lem-uniform-bound-bad-over-time} yields the desired result.
\end{proof}

\section{Bootstrap analysis for the disagreement clusters}{\label{sec:bootstrap-disagree-cluster}}

The goal of this section is to prove Lemma~\ref{lem-configuration-are-admissible}.

\subsection{Evolution of the disagreement in spacetime}{\label{subsec:space-time-cluster}}

Fix $G \in \mathcal G$ and $\sigma \in \{ -1,+1 \}^V$. For each $A \subset V$ such that $1\le |A|=pn$ for some $p\le e^{-0.02d}$, define
\begin{equation}{\label{eq-def-ell-t-given-A}}
    \ell = \bigg\lfloor \frac{ \log(1/p) }{ 10^3\log(d) } \bigg\rfloor, \quad t = \frac{ \log(1/p) }{ 10^3C\log(d)^{10} } \,.
\end{equation}
By \eqref{eq-assump-C-d-epsilon}, $p\le e^{-0.02d}\le d^{-4000}$ and $\ell\ge4$ for sufficiently large $d$. In addition, define
\begin{align}{\label{eq:def-paired-vertices}}
    E_A := \left\{ (u,v): v\in (A_{\mathsf{f}})^+, u \notin (A_{\mathsf f})^+, u \sim v \mbox{ in } G \right\} \,.
\end{align}
By Item~(5) of Definition~\ref{def-typical-graph}, we have $|A_{\mathsf{f}}|\leq 10pn, |\mathcal P(A_{\mathsf{f}})| \leq 100pn$ and $|(A_{\mathsf f})^+|\le200\ell pn$. In addition, by Corollary~\ref{cor-structure-A-final}, we see that
\begin{align}{\label{eq:ball-decomposition}}
    V\bigl(\mathcal B((A_{\mathsf{f}})^+,G,\ell+1)\bigr) = (A_{\mathsf f})^+ \sqcup \big( \sqcup_{(u,v) \in E_A} V(\mathbf T_u) \big) \,,
\end{align}
where $\mathbf T_{u} \in \mathbb T_{1,\ell}$. This is a partition of vertex sets, not a disjoint union of graphs: the attachment edges remain, and edges between outer leaves of different trees are not excluded. In particular, $|E_A|\le d|(A_{\mathsf f})^+|\le200d\ell pn$.
Set
\begin{equation}{\label{eq-def-cal-D-t-A}}
    \mathcal D_t^{A,\sigma} := \Big\{ v \in V: X_t^{\sigma}(v) \neq X_t^{\sigma_A}(v) \Big\} \,.
\end{equation}
Then $\mathcal D^{A,\sigma}_0 = A$. Throughout this section, $\mathcal D_s$ denotes $\mathcal D_s^{A,\sigma}$ for the fixed pair $(A,\sigma)$. We now record a causal-arrow diagram that expresses the passage of disagreements over time.  It is similar to the information-percolation approach to coupling high-temperature Ising chains pioneered in \cite{LS16} but rather than giving a grand coupling, it is defined with respect to a particular pair of configurations evolving under the monotone coupling.  Each time a new disagreement is added, it is attributed to one of the neighboring disagreement giving a causal diagram. Recall that for each $x\in V$, let $T_1^x<T_2^x<\cdots$ be the points of its rate-one Poisson clock and let $U_1^x,U_2^x,\ldots$ be the corresponding independent $\operatorname{Unif}[0,1]$ variables. Let $(\mathcal F_s)_{s\geq 0}$ be the completed natural filtration of this marked graphical construction, with the variable $U_k^x$ revealed only at time $T_k^x$. 
\begin{defn}[Oriented causal arrows]\label{def:oriented-causal-arrows}
    Write $\overrightarrow E(G):=\{(y,x):\{x,y\}\in E(G)\}$ for the set of oriented edges of $G$. Fix, once and for all, a deterministic ordering $\mathcal N_G(x)=\{y_1^x,\ldots,y_d^x\}$ for every $x \in V$. For $a,b\in[0,1]$, define the half-open interval 
    \[
        J(a,b):=
        \begin{cases}
            (a,b], &a<b \,;  \\
            (b,a], &b<a \,;  \\
            \emptyset, &a=b \,.
        \end{cases}
    \]
    For $s>0$, $x\in V$, and $0\leq i\leq d$, define the hybrid boundary configuration $\eta_i^{x,s}\in\{-1,+1\}^{N_G(x)}$ by
    \[
        \eta_i^{x,s}(y_j^x):=
        \begin{cases}
            X_{s-}^{\sigma}(y_j^x),&j\leq i \,;  \\
            X_{s-}^{\sigma_A}(y_j^x),&j>i \,.
        \end{cases}
    \]
    %In particular, $\eta_0^{x,s}$ is the $X_{s-}^{\sigma_A}$-boundary at $x$ and $\eta_d^{x,s}$ is the $X_{s-}^{\sigma}$-boundary at $x$.  
    Put
    \[
        g(h):=\frac{1}{1+e^{-2\beta h}}, \quad q_i^x(s):=g\left(\sum_{z\sim x}\eta_i^{x,s}(z)\right), \quad 0\leq i\leq d.
    \]
    For a hypothetical update variable $u\in[0,1]$ at $x$, define
    \[
        K(x,s;u):= \left\{ i\in\{1,\ldots,d\}: y_i^x\in \mathcal D_{s-}^{A,\sigma} \mbox{ and } u\in J\bigl(q_{i-1}^x(s),q_i^x(s)\bigr) \right\} \,.
    \]
    If $K(x,s;u)\neq\emptyset$, let $i_*(x,s;u):=\min K(x,s;u)$. The time-labeled oriented causal-arrow set is the random subset $\mathcal E_{\mathrm{caus}}^{A,\sigma} \subseteq \overrightarrow E(G)\times(0,\infty)$, defined by
    \[
        \mathcal E_{\mathrm{caus}}^{A,\sigma} :=\Big\{ \big(y_{i_*(x,s;U_k^x)}^x,x,s \big): s=T_k^x,\ x\in \mathcal D_s^{A,\sigma}, K(x,s;U_k^x)\neq\emptyset \Big\} \,.
    \]
\end{defn}
In this definition, the triple $(y,x,s)\in\mathcal E_{\mathrm{caus}}^{A,\sigma}$ is interpreted as the collection of oriented spacetime edges $(y,s-)\longrightarrow(x,s)$. By construction, for each new disagreement interval  exactly one incoming arrow is selected representing that the disagreement at the tail of the arrow caused the disagreement at its head.  Moreover, as in standard path coupling arguments, the intensity of arrows on edges is bounded by $C/d$.
\begin{figure}[!ht]
    \centering
    \includegraphics[width=\textwidth]{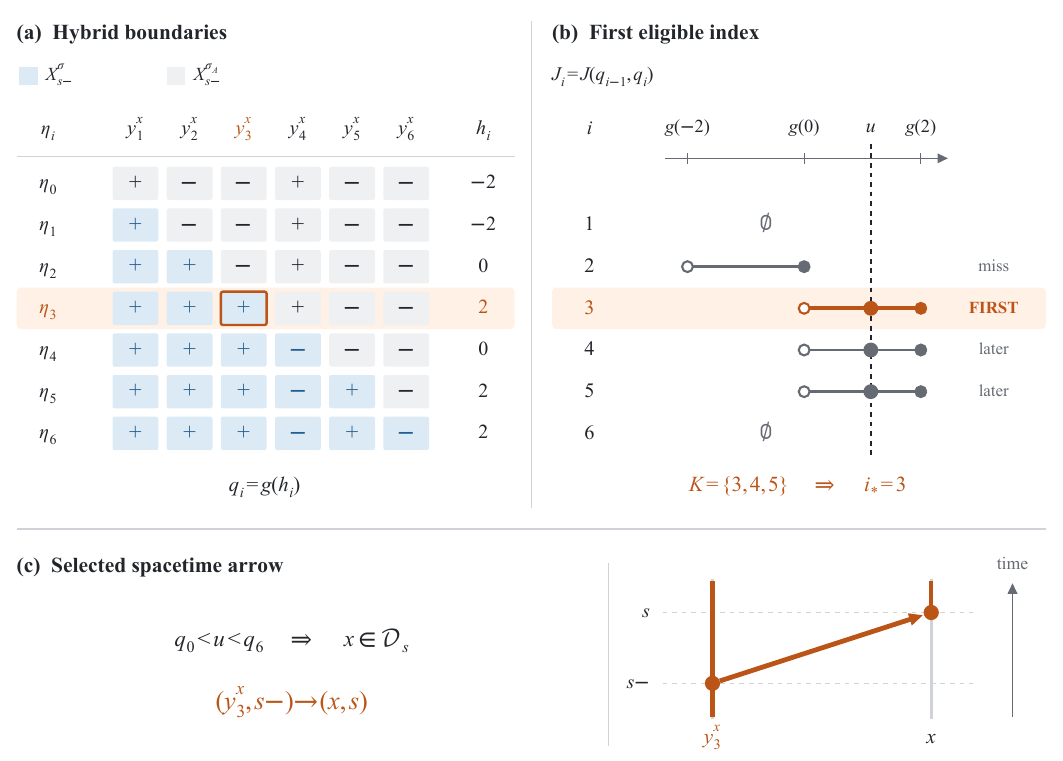}
    \caption{Graphical explanation of the oriented causal-arrow construction at an update time $s=T_k^x$. Write $\eta_i=\eta_i^{x,s}$, $q_i=q_i^x(s)$, $u=U_k^x$, $K=K(x,s;u)$, and $\mathcal{D}_t=\mathcal{D}_t^{A,\sigma}$. (a) In row $\eta_i$, the first $i$ spins (blue) come from $X_{s-}^{\sigma}$ and the remaining spins (gray) from $X_{s-}^{\sigma_A}$. Their sum $h_i=\sum_{j=1}^{6}\eta_i(y_j^x)$ gives the plus-update probability $q_i=g(h_i)$. (b) Row $i$ shows $J_i=J(q_{i-1},q_i)$. Open and filled endpoints indicate an excluded lower endpoint and an included upper endpoint, respectively. For $g(0)<u<g(2)$, the eligible indices are $K=\{3,4,5\}$, so the minimum-index rule selects $i_*=3$ (orange). (c) Since $q_0<u<q_6$, the coupled update gives $X_s^{\sigma_A}(x)=-1$ and $X_s^{\sigma}(x)=+1$. Thus $x\in\mathcal{D}_s$, and the construction records $(y_3^x,s-)\longrightarrow(x,s)$.}
    \label{fig:causal-oriented-arrow}
\end{figure}

\begin{lemma}\label{cl:arrows}
    The set $\mathcal E_{\mathrm{caus}}^{A,\sigma}$ is locally finite.  Moreover:
    \begin{enumerate}
    \item[(1)] If $s=T_k^x$ and $x\in \mathcal D_s^{A,\sigma}$, then $K(x,s;U_k^x)\neq\emptyset$. Consequently, there is a unique selected arrow $(y,x,s)\in\mathcal E_{\mathrm{caus}}^{A,\sigma}$, and it satisfies $y\in \mathcal D_{s-}^{A,\sigma} \cap \mathcal N_G(x)$.
    \item[(2)] For every $(y,x)\in\overrightarrow E(G)$, define
    \[
        N_{y\to x}(r):= \#\bigl\{s\in(0,r]:(y,x,s)\in \mathcal E_{\mathrm{caus}}^{A,\sigma}\bigr\} \,.
    \]
    Then $N_{y\to x}$ admits an $(\mathcal F_r)$-predictable intensity $\lambda_{y\to x}(r)$ satisfying $0\leq \lambda_{y\to x}(r)\leq \frac{C}{d}$ for all $r \ge 0$.
    \end{enumerate}
\end{lemma}
\begin{proof}
    Local finiteness is immediate because, on every bounded time interval, the finite collection of rate-one Poisson clocks has only finitely many rings. Fix an update time $s=T_k^x$. The heat-bath probabilities in the two coupled chains immediately before the update are
    \[
        q_0^x(s)=g\left(\sum_{z\sim x}X_{s-}^{\sigma_A}(z)\right), \qquad q_d^x(s)=g\left(\sum_{z\sim x}X_{s-}^{\sigma}(z)\right) \,.
    \]
    If the update creates a new disagreement interval at $x$, then the two updated spins (which use the same variable $U_k^x$) are different. Hence $U_k^x\in J(q_0^x(s),q_d^x(s))$. For every finite sequence $a_0,\ldots,a_d \in[0,1]$ one has $J(a_0,a_d)\subseteq \cup_{i=1}^d J(a_{i-1},a_i)$. Thus $U_k^x\in J(q_{i-1}^x(s),q_i^x(s))$ for at least one $i$, showing that $K(x,s;U_k^x)\neq\emptyset$. If $y_i^x\notin \mathcal D_{s-}^{A,\sigma}$, then $X_{s-}^{\sigma}(y_i^x)=X_{s-}^{\sigma_A}(y_i^x)$ so $\eta_{i-1}^{x,s}=\eta_i^{x,s}$ and therefore $q_{i-1}^x(s)=q_i^x(s)$. Thus every index obtained above satisfies $y_i^x\in \mathcal D_{s-}^{A,\sigma}$, proving the first assertion.

    We next prove the intensity bound. Fix an oriented edge $(y_i^x,x)\in\overrightarrow E(G)$. Conditionally on $\mathcal F_{s-}$ and on the occurrence of a ring at $x$ at time $s$, let $A_i^x(s)\subseteq[0,1]$ be the set of variable values $u\in[0,1]$ for which the deterministic rule in Definition~\ref{def:oriented-causal-arrows} selects the arrow $y_i^x\to x$. The set $A_i^x(s)$ is determined by the pre-update configurations and hence is $\mathcal F_{s-}$-measurable. By construction,
    \[
        A_i^x(s)\subseteq J\bigl(q_{i-1}^x(s),q_i^x(s)\bigr).
    \]
    Therefore, conditionally on $\mathcal F_{s-}$ and on a ring at $x$, the probability that $y_i^x\to x$ is selected is at most $|q_i^x(s)-q_{i-1}^x(s)|$. The two hybrid fields $\eta^{x,s}_i(\cdot), \eta^{x,s}_{i-1}(\cdot)$ differ only through the spin at $y_i^x$, and hence
    \[
        \left| \sum_{z\sim x}\eta_i^{x,s}(z) - \sum_{z\sim x}\eta_{i-1}^{x,s}(z) \right| \leq 2 \,.
    \]
    Furthermore,
    \[
        |g'(h)| = \frac{2\beta e^{-2\beta h}}{(1+e^{-2\beta h})^2} \leq \frac{\beta}{2} \quad\text{for every}\quad h\in\mathbb R \,.
    \]
    The mean-value theorem therefore gives
    \[
        \bigl|q_i^x(s)-q_{i-1}^x(s)\bigr| \leq 2\sup_{h\in\mathbb R}|g'(h)| \leq \beta =\frac{C}{d} \,.
    \]
    Since the clock at $x$ rings at rate one, the standard compensator theorem for a predictably marked Poisson process shows that $N_{y_i^x\to x}$ has predictable intensity
    \[
        \lambda_{y_i^x\to x}(s)=|A_i^x(s)| \leq \bigl|q_i^x(s)-q_{i-1}^x(s)\bigr| \leq \frac{C}{d} \,.
    \]
    This proves the second assertion.
\end{proof}

Based on Definition~\ref{def:oriented-causal-arrows}, we rewrite the oriented causal arrows as a directed acyclic graph~$\mathscr G_t^{A,\sigma}$ whose vertex set $\mathscr V_t$ is the vertex-time pairs of new disagreements.
\begin{defn}[Causal disagreement clusters]{\label{def:causal-disagreement-cluster}} 
    For $t>0$ and $x\in V$, define the set of disagreement birth times
    \[
        \mathscr H_x(t) := \left\{ a\in[0,t]:
        \begin{array}{l}
            a=0\text{ and }x\in \mathcal D_0^{A,\sigma},\quad\text{or}\\
            a\in(0,t],\ a=T_k^x\text{ for some }k,\ x\in \mathcal D_a^{A,\sigma}
        \end{array}
        \right\} \,.
    \]  
    For $a\in\mathscr H_x(t)$, let
    \[
        \delta_x(a) := \inf\Big\{ r>a: r=T_k^x\text{ for some }k \Big\} \,,
    \]
    with the convention $\inf\emptyset=\infty$, and define the corresponding disagreement interval by
    \[
        I_x(a):=
        \begin{cases}
            [a,\delta_x(a)),&\delta_x(a)\leq t \,, \\
            [a,t],&\delta_x(a)>t \,.
        \end{cases}
    \]
    Let $\mathscr V_t := \{(x,a):x\in V, a\in\mathscr H_x(t)\}$. We now put directed edges on $\mathscr V_t$ as follows. If $(y,x,s)\in\mathcal E_{\mathrm{caus}}^{A,\sigma}$ with $s\leq t$, then $(x,s)\in\mathscr V_t$. Moreover, there is a unique $a\in\mathscr H_y(t)$ such that $s\in I_y(a)$, because $y\in \mathcal D_{s-}^{A,\sigma}$ and the clock at $x$, rather than the clock at $y$, rings at time $s$. Insert the directed edge $(y,a)\longrightarrow(x,s)$. Denote the resulting finite directed graph by $\mathscr G_t^{A,\sigma}$. In particular, $\mathscr G_t^{A,\sigma}$ is acyclic.

    For any $B\subseteq V$, let $\mathscr S_t(B):=\{ (x,a) \in\mathscr V_t:x\in B \}$. Let $\mathsf{Desc}_{t}(B)$ be the set of all vertices of $\mathscr G_t^{A,\sigma}$ that are reachable by a directed path from at least one vertex in $\mathscr S_t(B)$. Define the causal disagreement cluster generated by $\{v\}\times[0,t]$ by
    \[
        \mathcal L_{v,\leq t} := \bigcup_{(x,a)\in\mathsf{Desc}_{t}(\{v\})} \bigl(\{x\}\times I_x(a)\bigr) \subseteq V\times[0,t] \,.
    \]
    In addition, define its time-$t$ section by
    \[
        \mathcal L_{v,t} := \bigl\{ x\in V:(x,t)\in\mathcal L_{v,\leq t} \bigr\} \,.
    \]
\end{defn}
\begin{figure}[htbp]
    \centering
    \includegraphics[width=\textwidth]{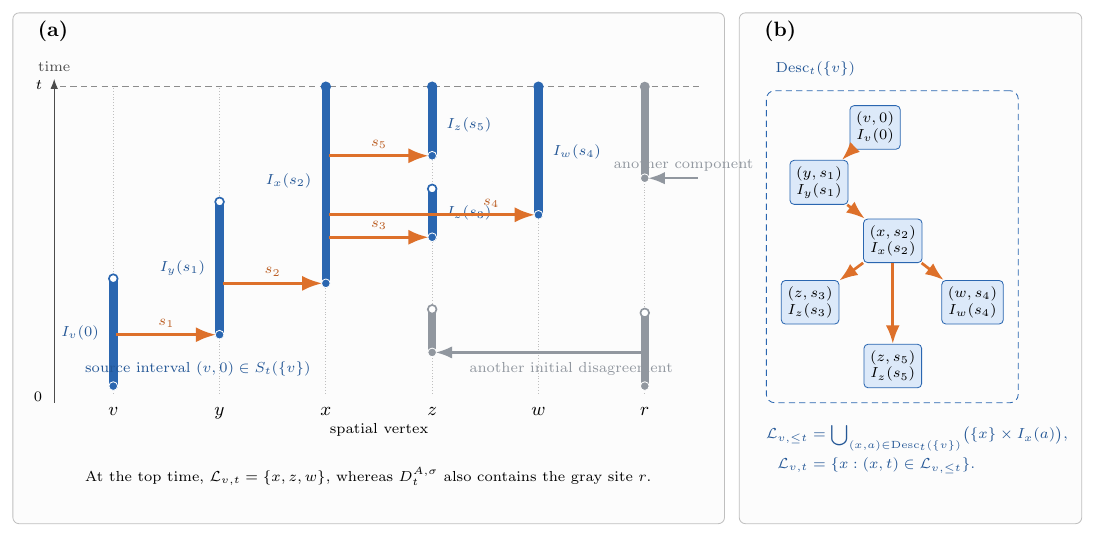}
    \caption{Graphical explanation of the causal disagreement cluster.  Each thick vertical segment is a disagreement interval $I_x(a)$, and an orange arrow at time $s$ joins the unique active parent interval to the disagreement interval born at its head. Collapsing every interval to one vertex gives the acyclic graph $\mathscr G_t^{A,\sigma}$. The blue vertices reachable from $\mathscr S_t(\{v\})$ form $\mathsf{Desc}_t(\{v\})$; their spacetime union is $\mathcal L_{v,\le t}$ and its section at the terminal time is $\mathcal L_{v,t}$. Gray intervals belong to the full disagreement set but not to this causal cluster.}
    \label{fig:causal-disagreement-cluster}
\end{figure}

\begin{defn}[Local tree]{\label{def:local-tree}}
    Define
    \begin{align}{\label{eq-def-R-R'}}
        R_0= \bigg\lfloor \frac{\ell}{C(\log d)} \bigg\rfloor, \quad R' = \bigg\lfloor \frac{\ell}{C(\log d)^3} \bigg\rfloor, \quad R= \bigg\lfloor \frac{\ell}{C(\log d)^5} \bigg\rfloor \,.
    \end{align}
    For a copy $\mathbf T \simeq \mathbb T_{k,\ell}$, define $\mathbf T'\subseteq \T$ induced by all $z \in V(\mathbf T)$ such that $\dist_{\mathbf T}(z,\rho(\mathbf T)) \le R'+1$. In addition, define $\mathbf T_*\subseteq \mathbf T$ induced by all $z \in V(\mathbf T)$ such that $\dist_{\mathbf T}(z,\rho(\mathbf T)) \le R+1$.
\end{defn}

\begin{defn}[Local chain]{\label{def:local-chain}}
    For each $(u,v) \in E_A$, we have $\mathbf T_u \sim \mathbb T_{1,\ell}$. Define $(\mathbf T_u')^{\circ}$ to be the tree obtained by attaching one minus leaf $v$ to the root of $\mathbf T_u'$. Let $(\underline{Y}_s^u)_{0 \le s \le t}$ and $(\overline{Y}_s^u)_{0 \le s \le t}$ be the Ising Glauber dynamics on $(\mathbf T_u')^{\circ}$ as follows:
    \begin{itemize}
        \item The initial conditions of $(\underline{Y}_s^u)_{0 \le s \le t}$ and $(\overline{Y}_s^u)_{0 \le s \le t}$ are given by
        \begin{align*}
            \underline{Y}_0^u(w) =
            \begin{cases}
                \sigma(w), & w \in V(\mathbf T_u') \setminus \mathsf L(\mathbf T_u') \,; \\
                -1, & w \in \mathsf L(\mathbf T_u') \,; \\
                -1, & w=v \,,
            \end{cases}
            \quad
            \overline{Y}_0^u(w) = 
            \begin{cases}
                \sigma(w), & w \in V(\mathbf T_u') \setminus \mathsf L(\mathbf T_u') \,; \\
                +1, & w \in \mathsf L(\mathbf T_u') \,; \\
                -1, & w=v \,.
            \end{cases}
        \end{align*}
        \item The boundary conditions of $(\underline{Y}_s^u)_{0 \le s \le t}$ and $(\overline{Y}_s^u)_{0 \le s \le t}$ are given by
        \begin{align*}
            \underline{Y}_s^u(w) =
            \begin{cases}
                -1, & w \in \mathsf L(\mathbf T_u') \,; \\
                -1, & w=v \,,
            \end{cases}
            \quad
            \overline{Y}_s^u(w) = 
            \begin{cases}
                +1, & w \in \mathsf L(\mathbf T_u') \,; \\
                -1, & w=v \,.
            \end{cases}
        \end{align*}
        \item $(\underline{Y}_s^u)_{0 \le s \le t}$ and $(\overline{Y}_s^u)_{0 \le s \le t}$ use the same Poisson clocks and uniform random variables of $(X_s^{\sigma})_{0 \le s \le t}$.
    \end{itemize}
\end{defn}

For each $(u,v) \in E_A$, define $\mathbf T_u(R)$ to be the set of vertices in $\mathbf T_u$ that have graph distance at most $R$ to $u$, and define $\mathbf T_u(R,R')=\mathbf T_u(R') \setminus \mathbf T_u(R)$. For a generic rooted tree $\mathbf T$ the notation $\mathbf T(R)$ has the same meaning, with distance measured in $\mathbf T$. Define the ``bad'' pairs 
\begin{align}
    &\mathcal V:= \Big\{ (u,v) \in E_A: \mathcal D_{s} \cap \mathbf T_u(R,R') \neq \emptyset \mbox{ for some } 0 \leq s \leq t \Big\} \nonumber \\
    &\quad\quad\quad \cup \Big\{ (u,v) \in E_A: \overline{Y}_s^u(w)> \underline{Y}_s^{u}(w) \mbox{ for some } w \in (\mathbf T_u)_*, 0 \leq s \leq t \Big\}  \,.  \label{eq-def-mathcal-V-1} 
\end{align}
We first make the following observation:
\begin{lemma}{\label{lem-disgreement-cluster-bound-1}}
    Let $r=t/(2000\log C)$. We have $\mathcal D^{A,\sigma}_t \subseteq \mathcal D_t^{(1)} \cup \mathcal D_t^{(2)} \cup \mathcal D_t^{(3)}$, where 
    \begin{align*}
        &\mathcal D^{(1)}_t = \mathcal D_{t}^{A,\sigma} \cap \mathcal B((A_{\mathsf f})^+,G,r) \,; \\
        &\mathcal D^{(2)}_t =  \cup_{(u,v) \in \mathcal V} \mathcal L_{u,t} \,; \\
        &\mathcal D^{(3)}_t =  \cup_{ (u,v) \in E_A \setminus \mathcal V } \big( \mathbf T_u(R) \cap \mathcal L_{u,t} \setminus \mathcal B(u,\mathbf T_u,r-1) \big)  \,.
    \end{align*}
    Thus, we have
    \begin{align}
        \big| \mathcal D^{A,\sigma}_t \big| \leq\ & \big| \mathcal D^{(1)}_t \big| + \big| \mathcal D^{(2)}_t \big| + \big| \mathcal D^{(3)}_t \big| \,.  \label{eq-disagreement-cluster-bound-1}
    \end{align}
\end{lemma}
\begin{proof}
    Fix a realization of the marked graphical construction up to time $t$, and fix $z\in \mathcal D_t^{A,\sigma}$. Starting from $(z,t)$, move backward along the corresponding disagreement interval. If the lower endpoint of the current interval is zero, then the current vertex belongs to $\mathcal D_0^{A,\sigma}=A$, and the procedure stops. Otherwise, the lower endpoint is an update time $s=T_k^x$ at the current vertex $x$, with $x\in \mathcal D_s^{A,\sigma}$. By Lemma~\ref{cl:arrows}(1), there is a unique selected causal arrow $(y,x,s) \in \mathcal{E}_{\mathrm{caus}}^{A, \sigma}$, where $y\in \mathcal D_{s-}^{A,\sigma}\cap N_G(x)$. Follow this arrow backward to the unique disagreement interval at $y$ that contains $s-$, and continue in the same manner. Since only finitely many updates occur in $[0,t]$, time strictly decreases at every use of a causal arrow and the procedure must terminate at time zero. It therefore produces an oriented walk $x_0,x_1,\ldots,x_m$ and times $0<s_1<\cdots<s_m<t$ such that $x_0\in A$, $x_m=z$, and $(x_{i-1},x_i,s_i) \in \mathcal{E}_{\mathrm{caus}}^{A,\sigma}$ for every $1\le i\le m$.

    If $z\in\mathcal{B}((A_{\mathsf f})^+,G,r)$, then $z \in \mathcal D_t^{(1)}$. We may therefore assume that $z\notin\mathcal{B}((A_{\mathsf f})^+,G,r)$. In particular, $z\notin(A_{\mathsf f})^+$.  Since $x_0\in A\subseteq A_{\mathsf f}\subseteq(A_{\mathsf f})^+$, the oriented walk has a last step that leaves $(A_{\mathsf f})^+$. Write this step as $x_{j-1}=v$ and $x_j=u$. Thus $v\in(A_{\mathsf f})^+$, $u\notin(A_{\mathsf f})^+$, $u\sim v$, and $x_i\notin(A_{\mathsf f})^+$ for every $j\le i\le m$. Consequently, $(u,v)\in\mathcal{E}_A$. The terminal directed segment beginning at $(u,s_j)$ and ending at the disagreement interval containing $(z,t)$ shows, by Definition~\ref{def:causal-disagreement-cluster}, that $z\in\mathcal{L}_{u,t}$.
    
    If $(u,v)\in\mathcal{V}$, then $z \in \mathcal D^{(2)}_t$. It remains to consider $(u,v)\notin\mathcal{V}$. The first event in the definition of $\mathcal{V}$ then fails, so $\mathcal D_s^{A,\sigma} \cap\mathbf{T}_u(R,R')=\emptyset$ for every $0\le s\le t$. By the choice of the last exit step, the terminal spatial walk $x_j,x_{j+1},\ldots,x_m$ never returns to $(A_{\mathsf f})^+$.  Corollary~\ref{cor-structure-A-final} and \eqref{eq:ball-decomposition} imply that, while this walk remains within distance $R'$ of $u$, it is contained in the exterior tree $\mathbf{T}_u$.  Moreover, by \eqref{eq-def-R-R'} and \eqref{eq-assump-C-d-epsilon}, any path in $G\setminus(A_{\mathsf f})^+$ that starts at $u$ and leaves $\mathbf{T}_u(R)$ must visit $\mathbf{T}_u(R,R')$. Every vertex of the terminal causal segment is a disagreement at the corresponding time. Hence such a visit would imply $\mathcal D_s^{A,\sigma} \cap \mathbf{T}_u(R,R')\ne\emptyset$ for some $0\le s\le t$, a contradiction. Therefore the whole terminal segment is contained in $\mathbf{T}_u(R)$, and in particular $z\in\mathbf{T}_u(R)$.

    The edge $uv$, followed by the unique path in $\mathbf{T}_u$ from $u$ to $z$, is a path in $G$ from $(A_{\mathsf f})^+$ to $z$ of length $1+\mathsf{dist}_{\mathbf{T}_u}(u,z)$. Since $z\notin\mathcal{B}((A_{\mathsf f})^+,G,r)$, this length is larger than $r$. Consequently, $\mathsf{dist}_{\mathbf{T}_u}(u,z)>r-1$, and hence $z\notin\mathcal{B}(u,\mathbf{T}_u,r-1)$.  We have thus shown that $z\in\mathbf{T}_u(R)\cap\mathcal{L}_{u,t}\setminus \mathcal{B}(u,\mathbf{T}_u,r-1)$ for some $(u,v)\in \mathcal{E}_A \setminus \mathcal{V}$, leading to $z \in \mathcal D_t^{(3)}$. Since $z\in D_t^{A,\sigma}$ was arbitrary, the asserted set inclusion follows. Taking cardinalities gives \eqref{eq-disagreement-cluster-bound-1}. 
\end{proof} 

It remains to bound $|\mathcal D^{(1)}_t|, |\mathcal D^{(2)}_t|$ and $|\mathcal D^{(3)}_t|$ respectively. We first bound $|\mathcal D^{(2)}_t|$ using the following estimate.
\begin{lemma}{\label{lem-bound-far-away-disagreement-cluster}}
    For all $\sigma\in \{ \pm 1 \}^V$ and for all $A \subset V$ with $1\le |A|=pn$ for some $p\le e^{-0.02d}$, we have
    \begin{align*}
        \mathbb E_{\sigma}\Big[ \big| \cup_{ (u,v) \in \mathcal V } \mathcal L_{u,t} \big| \Big] \le e^{-t}pn \,.
    \end{align*}
\end{lemma}
We then bound $|\mathcal D^{(1)}_t|$ using the following lemma. %\textcolor{red}{Should this be $D^1$?}
\begin{lemma}{\label{lem-bound-close-enough-disagreement-cluster}}
    For all $\sigma\in \{ \pm 1 \}^V$ and for all $A \subset V$ with $1\le |A|=pn$ for some $p\le e^{-0.02d}$, we have
    \begin{align*}
        \mathbb E_{\sigma}\Big[ \big| \mathcal D_{t}^{A,\sigma} \cap \mathcal B((A_{\mathsf f})^+,G,r) \big| \Big] \le e^{-t/2} pn \,.
    \end{align*}
\end{lemma}
Finally, we bound $|\mathcal D^{(3)}_t|$ using the following lemma. 
\begin{lemma}{\label{lem-growth-rooted-tree-subtree}}
    For all $pn\in \mathbb Z_+$ and $p\le e^{-0.02d}$, we have $\pi_+(\overline{\mathcal A}) \ge 1-n^{-20}$, where $\overline{\mathcal A}$ is the collection of $\sigma\in\{ \pm 1 \}^V$ such that for all $1\le |A|=pn$,
    \begin{align*}
        \mathbb E_{\sigma}\Bigg[ \sum_{ (u,v) \in E_A } \big| \mathbf T_u(R) \cap \mathcal L_{u,t} \setminus \mathcal B(u,\mathbf T_u,r-1) \big| \mathbf 1_{ \{ (u,v) \notin \mathcal V \} } \Bigg] \le e^{-t/C^3} \cdot 100\ell pn \,.
    \end{align*}
\end{lemma}
Based on Lemmas~\ref{lem-bound-far-away-disagreement-cluster}--\ref{lem-growth-rooted-tree-subtree}, we can now finish the proof of Lemma~\ref{lem-configuration-are-admissible}.
\begin{proof}[Proof of Lemma~\ref{lem-configuration-are-admissible}]
    For each integer $1\le k\le e^{-0.02d}n$, denote the set in Lemma~\ref{lem-growth-rooted-tree-subtree}, with $p=k/n$, by $\overline{\mathcal A}(k/n)$, and set
    \[
        \overline{\mathcal A}:=\cap_{1\le k\le e^{-0.02d}n} \overline{\mathcal A}(k/n).
    \]
    There are at most $n$ possible sizes, so $\pi_+(\overline{\mathcal A})\ge1-n^{-19}$. Fix $\sigma\in\overline{\mathcal A}$ and a nonempty $A$ in the admissible range. Lemma~\ref{lem-growth-rooted-tree-subtree} gives
    \[
        \Eb_\sigma[|\mathcal D_t^{(3)}|]\le100\ell e^{-t/C^3}pn\le e^{-t/(2C^3)}pn.
    \]
    Combining this with Lemmas~\ref{lem-bound-far-away-disagreement-cluster} and \ref{lem-bound-close-enough-disagreement-cluster} and \eqref{eq-disagreement-cluster-bound-1} proves \eqref{eq-def-shrink-time}. Thus $\overline{\mathcal A}\subseteq\mathscr A$.

    Let $0=\tau_0<\tau_1<\cdots$ be all rings of the $n$ rate-one clocks, including updates that leave the configuration unchanged. The embedded transition is the average of the restricted one-site heat-bath kernels, so $\widehat X_{\tau_j}^{\pi_+}$ has law $\pi_+$ for every deterministic index $j$. Also $N(n)\sim\operatorname{Pois}(n^2)$. Since the configuration is constant between rings,
    \[
        \Pb(\exists s\le n:\widehat X_s^{\pi_+}\notin\mathscr A) \le (2n^2+1)n^{-19}+\Pb(N(n)>2n^2) \le n^{-12}
    \]
    for all sufficiently large $n$. This proves Lemma~\ref{lem-configuration-are-admissible}. 
\end{proof}

\subsection{Proof of Lemma~\ref{lem-bound-far-away-disagreement-cluster}}{\label{subsec:proof-lem-5.3}}

To control the growth of causal disagreement clusters $\mathcal L_{u,t}$, we will further introduce the following infectious process on $G$ which is a rate $C/d$ contact process without healing.
\begin{defn}{\label{def-infectious-process}}
    Consider the infectious process on $G$ defined as follows. Put an independent exponential clock with rate $\frac{C}{d}$ on each oriented edge. Denote by $\mathcal I_t$ the set of infected vertices at time $t$. Once $u \in \mathcal I_t$, the clock $u \to v$ tries to add $v$ into $\mathcal I_t$. If at time $t$ a vertex $v$ has $k$ neighbors in $\mathcal I_t$, then there are $k$ active clocks trying to add $v$, each with rate $\frac{C}{d}$. Since the minimum of these $k$ exponential clocks has rate $\frac{kC}{d}$, it is equivalent that $v$ will join $\mathcal I$ with rate $\frac{Ck}{d}$. For $(u,v) \in E_A$, define $\mathcal I_{u,t}$ to be the infectious process started from $\mathcal I_{u,0}=\{ u \}$. These auxiliary processes may be sampled independently for different roots. The domination below is a separate marginal comparison for each root; no simultaneous independent domination of all causal clusters is asserted or used.
\end{defn}
From Lemma~\ref{cl:arrows} and the fact that $\mathcal L_{u,0} \subset \mathcal I_{u,0}$, we immediately have the following observation.
\begin{claim}{\label{lem-stochastic-doination-by-infection}}
    $\mathcal L_{u,t}$ is stochastically dominated by $\mathcal I_{u,t}$.
\end{claim}
To justify this comparison, keep the root permanently infected and ignore all recoveries. Before an as-yet unreached vertex is first reached from the root spacetime line, it has an incoming causal arrow from an already reached vertex. The total rate of such first entries is at most $(C/d)$ times its number of reached neighbors, by Lemma~\ref{cl:arrows}. The usual monotone coupling of these first-entry counting processes with the infection process proves the assertion. 

We now return to the proof of Lemma~\ref{lem-bound-far-away-disagreement-cluster}. We first show the following simple bound for the growth rate of $\mathcal L_{u,t}$. 
\begin{lemma}{\label{lem-growth-infectious-process}}
    We have $\Eb[|\mathcal L_{u,t}|^2] \le 2e^{2Ct}$.
\end{lemma}
\begin{proof}
    By Claim~\ref{lem-stochastic-doination-by-infection}, it suffices to show that $\Eb[|\mathcal I_{u,t}|^2] \le 2e^{2Ct}$. The total birth rate of $\mathcal I_{u,s}$ is at most $C$ times its current size, since each infected vertex has at most $d$ neighbors and each edge transmits at rate $C/d$. Therefore, $|\mathcal I_{u,s}|$ is stochastically dominated by a Yule process $(Y_s)_{s\ge0}$ with per-particle birth rate $C$ and $Y_0=1$. For this Yule process,
    \begin{equation*}
        \Eb[Y_t^2] = e^{2Ct}+ e^{Ct}(e^{Ct}-1) \le 2e^{2Ct} \,. \qedhere
    \end{equation*}
\end{proof}
We now show that $\mathcal V$ is rare in the following sense. 
\begin{lemma}{\label{lem-V-1-W-1-are-rare}}
    We have $\Pb((u,v) \in \mathcal V) \le 3e^{-\ell/(C(\log d)^5)}$.
\end{lemma}
\begin{proof}
    Put $a_0=\ell/(C(\log d)^5)$, so $R=\lfloor a_0\rfloor$. Initially there are no disagreements inside $\mathbf T_u$, since $A\subseteq(A_{\mathsf f})^+$ and the exterior trees are disjoint from this set. A backward causal path reaching $\mathbf T_u(R,R')$ must therefore enter this tree either at its root or at an ordinary outer leaf. Keeping all of these possible entry vertices active from time zero only enlarges the event. The first alternative requires $R$ ordered transmissions along a downward path from $u$; the second requires $\ell-R'$ ordered transmissions along an upward path from an outer leaf. There are at most $d^R$ and $d^\ell$ such paths, respectively.

    For the two local chains, the only initial and boundary discrepancies are the ordinary leaves of $\mathbf T_u'$ at depth $R'+1$. To reach their comparison region $\mathbf T_*$, which has depth $R+1$, a causal discrepancy path requires $R'-R$ ordered transmissions. There are at most $d^{R'+1}$ possible upward paths. The same single-edge intensity bound $C/d$ applies to this pair of chains.

    For a prescribed path of length $j$, iterated conditioning on its successive transmission counting processes bounds the probability of ordered transmissions before $t$ by $(Ct/d)^j/j!\le(eCt/(dj))^j$. No independence of the causal arrows is needed. Consequently the three contributions are at most
    \[
        \left(\frac{eCt}{R}\right)^R,\qquad d^\ell\left(\frac{eCt}{d(\ell-R')}\right)^{\ell-R'},\qquad d^{R'+1}\left(\frac{eCt}{d(R'-R)}\right)^{R'-R} \,.
    \]
    Each is at most $e^{-a_0}$ in our working regime. Indeed $Ct\le2\ell/(\log d)^9$, so the logarithm of the first is at most $-2R\le-a_0$. The logarithm of the second is
    \[
        R'\log d+(\ell-R')\log\!\left(\frac{eCt}{\ell-R'}\right)\le-\ell\log\log d,
    \]
    and that of the third is
    \[
        (R+1)\log d+(R'-R)\log\!\left(\frac{eCt}{R'-R}\right)\le-R'\log\log d.
    \]
    These inequalities follow uniformly from $C\le C_0\log d$ and the rounded-radius comparisons after \eqref{eq-def-ell-t-given-A}. Adding the three contributions proves the lemma.
\end{proof}
We can now finish the proof of Lemma~\ref{lem-bound-far-away-disagreement-cluster}.
\begin{proof}[Proof of Lemma~\ref{lem-bound-far-away-disagreement-cluster}]
    By Definition~\ref{def-typical-graph}(5) and Corollary~\ref{cor-structure-A-final}, we have $|E_A| \le 200d\ell np$. By Lemma \ref{lem-growth-infectious-process}, Lemma~\ref{lem-V-1-W-1-are-rare} and Cauchy-Schwarz inequality, we have
    \begin{align*}
        \mathbb E_{\sigma}\Big[ \big| \cup_{ (u,v) \in \mathcal V } \mathcal L_{u,t} \big| \Big] &\le \left( \sum_{(u,v) \in E_A} \Eb_\sigma\Big[ |\mathcal L_{u,t}|^2 \Big] \right)^{1/2} \left( \sum_{(u,v)\in E_A} \Pb_\sigma\Big( (u,v) \in \mathcal V \Big) \right)^{1/2} \\
        &\le \left( 200d\ell np \cdot 6e^{2Ct} \right)^{1/2} \left( 200d\ell np \cdot 3e^{-\ell/(C(\log d)^5)} \right)^{1/2} \overset{\eqref{eq-def-ell-t-given-A}, \eqref{eq-assump-C-d-epsilon}}{\le} e^{-t}pn \,. \qedhere
    \end{align*}
\end{proof}

\subsection{Proof of Lemma~\ref{lem-bound-close-enough-disagreement-cluster}}{\label{subsec:proof-lem-5.8}}

Write $\mathcal D_t=\mathcal D_{t}^{A,\sigma}$. Fix $z\in \mathcal D_t$. Starting from $(z,t)$, move backward as follows. Suppose the current spacetime point is $(x,u)$ with $x\in \mathcal D_u$. If the clock at $x$ has no ring in $(0,u)$, then $x\in \mathcal D_0=A$, and the procedure stops. Otherwise, let $s<u$ be the last ring of the clock at $x$ before $u$. Since there is no later ring at $x$ and $x\in \mathcal D_u$, one has $x\in \mathcal D_s$. Hence by Definition~\ref{def:causal-disagreement-cluster}, an oriented causal arrow $y\to x$ was inserted at time $s$, with $y\in \mathcal D_{s-}$. Jump to $(y,s-)$ and continue. Only finitely many updates occur before time $t$, and time decreases strictly at every jump. Thus the procedure terminates and produces an oriented walk $x_0,x_1,\ldots,x_m$ and times $0<s_1<\cdots<s_m<t$ such that the following conditions hold:
\begin{itemize}
    \item $x_0\in A,x_m=z$;
    \item An oriented causal arrow $(x_{i-1},x_i,s_i) \in \mathcal E_{\rm caus}^{A,\sigma}$;
    \item With $s_0:=0$ and $s_{m+1}:=t$, the clock at $x_i$ has no ring in $(s_i,s_{i+1})$.
\end{itemize}
\begin{lemma}{\label{lem:prob-causal-time-sequence}}
    For every fixed oriented walk $w=(x_0,\ldots,x_m)$, the probability that there exist times $(s_0,\ldots,s_m)$ for which the above three conditions hold is at most $e^{-t}\frac{(Ct/d)^m}{m!}$.
\end{lemma}
\begin{proof}
Count the ordered update tuples that realize the walk, and bound the probability in question by their expected number. Fix $0<s_1<\cdots<s_m<t$ and insert rings at $x_i$ at times $s_i$. The factorial-moment formula for independent rate-one Poisson clocks gives Lebesgue measure $\dd s_1\cdots\dd s_m$ for these insertions, also when the walk revisits vertices. The forbidden intervals $(s_i,s_{i+1})$, assigned to the clocks at $x_i$, are disjoint in time and have total length $t$. Their joint no-ring probability is therefore $e^{-t}$; none contains an inserted ring.

Condition on this complete clock schedule, and integrate the independent update marks at $s_m,s_{m-1},\ldots,s_1$, in that order. For any values of the earlier marks, the set of marks producing the prescribed arrow has length at most $C/d$, by Lemma~\ref{cl:arrows}. Integrating from latest to earliest avoids conditioning an earlier arrow on a future arrow event. Thus the conditional probability of all prescribed arrows is at most $(C/d)^m$. The expected number of admissible tuples is at most
\[
 e^{-t}(C/d)^m\int_{0<s_1<\cdots<s_m<t}\dd s_1\cdots\dd s_m
 =e^{-t}\frac{(Ct/d)^m}{m!} \,.
\]
For $m=0$, this is just the probability of no ring at $x_0$ before $t$.
\end{proof}
For $m\geq0$, let $\widetilde W_{m,r}$ be the number of oriented $m$-step walks $(x_0,\ldots,x_m)$ such that $x_0\in A$ and $x_m\in \mathcal B((A_{\mathsf f})^+,G,r)$.
Based on Lemma~\ref{lem:prob-causal-time-sequence}, we have
\begin{equation}{\label{eq:bound-close-disagreement-by-walks}}
    \mathbb E_{\sigma}\Big[ \big| \mathcal D_{t}^{A,\sigma} \cap \mathcal B((A_{\mathsf f})^+,G,r) \big| \Big] \le e^{-t} \sum_{m\geq0}\frac{(Ct/d)^m}{m!} \cdot \widetilde W_{m,r} \,.
\end{equation}
We shall use two bounds on $\widetilde W_{m,r}$. First, for every $m\geq0$, we clearly have
\begin{equation}{\label{eq:enu-walks-naive}}
   \widetilde W_{m,r}\leq |A|d^m=pn d^m \,.
\end{equation}
Second, since $A\subseteq A_{\mathsf f}$, Definition~\ref{def-typical-graph}(5) implies 
\[
    |A| \le |A_f|\leq10pn \quad \mbox{and} \quad \big|\mathcal B((A_{\mathsf f})^+,G,r)\big| \le 2d^r|(A_{\mathsf f})^+| \le 400d^r \ell pn \,.
\]
Definition~\ref{def-typical-graph}(6) then gives whenever $1\leq m\leq \ell$ and $r\leq0.1m$,
\begin{equation}{\label{eq:enu-walks-careful}}
    \widetilde W_{m,r} \leq W_{m}\big( A,\mathcal B((A_{\mathsf f})^+,G,r) \big) \leq 2\sqrt{|A||\mathcal B((A_{\mathsf f})^+,G,r)|} d^{0.51m} \leq 200\ell pnd^{0.51m+0.5r} \,.
\end{equation}
Now we provide the proof of Lemma~\ref{lem-bound-close-enough-disagreement-cluster}.
\begin{proof}[Proof of Lemma~\ref{lem-bound-close-enough-disagreement-cluster}]
    We divide \eqref{eq:bound-close-disagreement-by-walks} into the following parts.

    \noindent{\bf Part 1: contribution from $0\leq m<10r$.} By \eqref{eq:enu-walks-naive}, we have 
    \begin{align*}
        e^{-t} \sum_{0\le m<10r}\frac{(Ct/d)^m}{m!} \cdot \widetilde W_{m,r} \le pne^{-t}\sum_{0\leq m\leq 10r}\frac{(Ct)^m}{m!} \le pn e^{-0.9t} \,.
    \end{align*}

    \noindent{\bf Part 2: contribution from $10r \le m \le \ell$.} In this case the bound \eqref{eq:enu-walks-careful} holds, so 
    \begin{align*}
        e^{-t} \sum_{10r \le m\le\ell}\frac{(Ct/d)^m}{m!} \cdot \widetilde W_{m,r} \le 200\ell pne^{-t} \sum_{m\geq0} \frac{\bigl(Ctd^{-0.44}\bigr)^m}{m!} = 200\ell pne^{ -(1-Cd^{-0.44})t } \overset{\eqref{eq-assump-C-d-epsilon}}{\le} pne^{-0.9t} \,.
    \end{align*}

    \noindent{\bf Part 3: contribution from $m \ge \ell+1$.} By \eqref{eq:enu-walks-naive}, we have
    \begin{align*}
        e^{-t} \sum_{m\ge\ell+1}\frac{(Ct/d)^m}{m!} \cdot \widetilde W_{m,r} \le pne^{-t}\sum_{m\ge\ell+1}\frac{(Ct)^m}{m!} \leq pne^{-t} \frac{(Ct)^{\ell+1}}{(\ell+1)!} \left(1-\frac{Ct}{\ell+1} \right)^{-1} \overset{\eqref{eq-def-ell-t-given-A}}{\le} pne^{-t} \,.
    \end{align*}
    Finally, combining the contribution from the three parts yields Lemma~\ref{lem-bound-close-enough-disagreement-cluster}.
\end{proof}

\subsection{Proof of Lemma~\ref{lem-growth-rooted-tree-subtree}}{\label{subsec:proof-lem-5.5}}

By Corollary~\ref{cor-structure-A-final}, the trees $\{ \mathbf T_u:(u,v) \in E_A \mbox{ for some } v \in (A_{\mathsf f})^+ \}$ are pairwise disjoint as well. Every child subtree $\mathbf T_u$ occurring above is a rooted tree of type $\mathbb T_{1,\ell}$. Let $\mathscr T_{\ell}$ be the family of all distinct rooted balls $\mathcal B(u,G_{\setminus\{v\}},\ell)\simeq\mathbb T_{1,\ell}$ with $u\sim v$ in $G$. It is clear that $|\mathscr T_{\ell}|\le dn$, since every member is determined by an oriented edge $(u,v)$, and Corollary~\ref{cor-structure-A-final}(3) shows that every exterior tree used here belongs to this family. Recall \eqref{eq-def-R-R'} and Definition~\ref{def:local-tree}. Also recall that for each $\mathbf T \in \mathscr T_{\ell}$ such that $\rho(\mathbf T)=u$, we denote by $\partial_R \mathbf T$ the set of vertices in $\mathbf T$ that have distance $R$ to $u$, and recall that for each $w \in \partial_R \mathbf T$, we denote by $\mathbf T_w$ the child subtree of $\mathbf T$ rooted at $w$. Clearly $\mathbf T_w \in \mathbb T_{1,\ell-R}$. 
\begin{defn}{\label{def-mathcal-A-1-lem-5.4}}
    We say $\tau \in \{ \pm 1 \}^{V(\mathbf T)}$ is good, if for every $u \in \partial_{R_0} \mathbf T$, $\tau$ is favorable with respect to $u$. Let $s:=\lceil0.1d\rceil^{\ell-R_0}$. Define $\overline{\mathcal A}_{1}$ to be the set of $\sigma \in \{ \pm 1 \}^V$ such that there are at most $e^{-0.001Cs}n$ disjoint $\mathbf T\in \mathscr T_{\ell}$ such that $\sigma|_{\mathbf T}$ is not good.
\end{defn}
Given a tree $\mathbf T \in \mathscr T_{\ell}$ and $D,D' \subset V(\mathbf T)$, we say $D \Longrightarrow D'$ if every connected component of $D'$ intersects $D$. Recall $\mathbf T_*$ defined in Definition~\ref{def:local-tree}. Put $\Delta=0.01$ and $T=\lceil t/\Delta\rceil$. Define $I_j=[(j-1)\Delta,j\Delta]$ for $1 \le j \le T$. For $1 \le i \le T$, define
\begin{align}
    \mathcal O_i(\mathbf T):= \Big\{ & (B_i,D_i,\ldots,B_{T},D_T): D_s \neq \emptyset, D_s \subset V(\mathbf T_*), B_s \subset D_{s} \mbox{ for } i \le s \le T \nonumber  \\
    & \rho(\mathbf T) \Longrightarrow D_i \Longrightarrow \ldots \Longrightarrow D_{T} \Big\} \,. \label{eq-def-trajectory}
\end{align}
In addition, define
\begin{equation}{\label{eq-def-good-trajectory}}
    \mathcal O_{i}^{\diamond}(\mathbf T):= \Big\{ (B_i,D_i,\ldots,B_{T},D_T) \in \mathcal O_i(\T): \sum_{i\le s\le T} |B_s| \ge 10^{-4} C^{-2} \sum_{i \le s \le T} |D_s| \Big\} \,.
\end{equation}
\begin{defn}{\label{def-mathcal-A-2-lem-5.4}}
    Recall the local chain $\underline{Y}_t=\underline{Y}_t^u, \overline{Y}_t=\overline{Y}_t^u$ in Definition~\ref{def:local-chain}. Given $\tau=\sigma|_{\mathbf T} \in \{ \pm 1 \}^{V(\mathbf T)}$, define 
    \begin{align*}
        \mathfrak B_t = \Big\{ v \in V(\mathbf T_*): \sum_{u \in\mathcal N_{\T}(v)} \overline{Y}_t(u) - \mathbf 1_{v=\rho(\mathbf T)} \le 0.5d \Big\}, \quad \mathfrak B_{[s,t]} = \cup_{ s \le r \le t } \mathfrak B_{r} \,.
    \end{align*}
    We say $\tau\in \{ \pm 1 \}^{V(\mathbf T)}$ is typical, if for all $1 \le i \le T$ and $(B_i,D_i,\ldots,B_T,D_T) \in \mathcal O^{\diamond}_i(\mathbf T)$, we have
    \begin{align*}
        \Pb_{\tau}\Big( B_i \subseteq \mathfrak B_{I_i}, \ldots, B_T \subseteq \mathfrak B_{I_T} \Big) \le \exp\Big( -10^{-6} C^{-2} d \sum_{j=i}^{T} |B_j| \Big) \,.
    \end{align*}
    Define $\overline{\mathcal A}_{2}$ to be the set of $\sigma \in \{ \pm 1 \}^V$ such that there are at most $e^{-dt/(10^{20}C^5)}n$ disjoint $\mathbf T\in \mathscr T_{\ell}$ such that $\sigma|_{\mathbf T}$ is not typical.
\end{defn}
\begin{lemma}{\label{lem-mathcal-A-1-2-is-typical-lem-5.4}}
    We have the following estimates:
    \begin{enumerate}
        \item[(1)] $\pi_+(\overline{\mathcal A}_{1}^c) \le n^{-24}$;
        \item[(2)] $\pi_+(\overline{\mathcal A}_{2}^c \cap \overline{\mathcal A}_{1}) \le n^{-24}$.
    \end{enumerate}
\end{lemma}

We first show how to derive Lemma~\ref{lem-growth-rooted-tree-subtree} based on Lemma~\ref{lem-mathcal-A-1-2-is-typical-lem-5.4}. Fix $\sigma \in\{\pm1\}^{V}$ that is typical with respect to $\mathbf T$ where $\rho(\mathbf T)=u$ and $(u,v)\in E_A$, and write $\tau=\sigma|_{\mathbf T}$. Let $(\underline{Y}_s)_{0 \le s \le t}, (\overline{Y}_s)_{0 \le s \le t}$ be local chain on $(\mathbf T')^{\circ}$ as in Definition~\ref{def:local-chain}, and define the bad vertex set $\mathfrak B_t$ with respect to $\overline{Y}_t$. By \eqref{eq-def-mathcal-V-1} and monotonicity, we have when $(u,v) \notin \mathcal V$ 
\begin{align}
    (X_s^{\sigma}|_{\mathbf T_*})_{0 \le s \le t} \succeq (\underline{Y}_s|_{\mathbf T_*})_{0 \le s \le t} = (\overline{Y}_s|_{\mathbf T_*})_{0 \le s \le t} \,.  \label{eq:local-monotone-comparison}
\end{align}
We now introduce a space-time cluster $\overline{\cL}_t$ on $\T_*$ to stochastically dominate the disagreements.  We set $\overline{\cL}_0=\{ u \}$, use the same Poisson clock as $(X_t)$, and use the following update rule. Let 
\begin{equation}
    m_s(x):=\big| \mathcal N_{\mathbf T}(x)\cap \overline{\cL}_{s-} \big| + \mathbf 1_{x=u} \,.
\end{equation} 
When the clock at $x \in V(\mathbf T_*)$ rings at time $s$, all vertices other than $x$ retain their previous occupation status, and $x$ is declared to belong to $\overline{\cL}_s$ with conditional probability
\begin{equation}\label{eq:extended-update-rule}
    p_s(x):=
    \begin{cases}
        0, &x\notin\overline{\cL}_{s-}\text{ and }m_s(x)=0 \,; \\
        \min\{ 1, \frac{C}{d} m_s(x) \}, &x\in\mathfrak B_{s-} \text{ or } m_s(x)>d/100\,; \\ 
        \frac{e^{-C/2}}{d}m_s(x), &\text{otherwise} \,.
    \end{cases}
\end{equation}

\begin{lemma}{\label{lem:bound-by-extended-spacetime-lem-5.4}}
    For any $(u,v)\in E_A$, we have
    \begin{equation*}
        \E_{\sigma}\Big[ \big| (\mathbf T(R)\cap \mathcal L_{u,t}) \setminus \mathcal B(u,\mathbf T,r-1) \big| \cdot \mathbf 1_{\{(u,v) \notin \mathcal V\}} \Big] \le \E_{\sigma}\Big[ \big| \overline\cL_t \setminus \mathcal B(u,\mathbf T,r-1) \big| \Big] \,.
    \end{equation*}
\end{lemma}
\begin{proof}
Throughout this proof $\mathcal D_s=\mathcal D_s^{A,\sigma}$ and $\mathbf T=\mathbf T_u$. Define
\[
 \zeta:=\inf\Big\{s\ge0:\mathcal D_s\cap\mathbf T(R,R')\ne\emptyset
 \ \text{or}\ \underline Y_s|_{\mathbf T_*}\ne\overline Y_s|_{\mathbf T_*}\Big\}.
\]
This is a stopping time for the joint graphical construction. Its annulus is exactly the one in \eqref{eq-def-mathcal-V-1}, namely $R<\dist_{\mathbf T}(u,x)\le R'$. Hence $\zeta>t$ on $\{(u,v)\notin\mathcal V\}$. Define
\begin{equation}\label{eq:stopped-cluster}
 \cL_s^\circ:=\begin{cases}
  \mathcal D_s\cap\mathbf T(R),&s<\zeta,\\
  \emptyset,&s\ge\zeta.
 \end{cases}
\end{equation}
We use all inner-ball disagreements, not just descendants of the root. Since $\mathcal L_{u,t}\subseteq\mathcal D_t$, this gives
\begin{equation}\label{eq:typical-rooted-tree-estimate-relax-0}
 \big|\mathbf T(R)\cap\mathcal L_{u,t}\setminus\mathcal B(u,\mathbf T,r-1)\big|
 \mathbf1_{\{(u,v)\notin\mathcal V\}}
 \le\big|\cL_t^\circ\setminus\mathcal B(u,\mathbf T,r-1)\big|.
\end{equation}
This inner restriction is important: the globally defined $\mathcal L_{u,t}$ can leave through $uv$ and later re-enter the exterior tree at an outer leaf, without crossing the inner annulus. No domination of its outer part is asserted or needed.

We construct a coupling with $\cL_s^\circ\subseteq\overline{\cL}_s$ by induction over updates, revealing auxiliary uniforms only at their update times. Initially $\cL_0^\circ=\emptyset\subseteq\{u\}$, since $A$ is disjoint from $\mathbf T$. Before $\zeta$, for every $x\in\mathbf T(R)$ all its tree neighbors belong to $\mathbf T_*$. Those at depth $R+1$ are in $\mathbf T(R,R')$ and are not disagreements. The only graph neighbor outside this tree neighborhood is the attachment $v$ at $x=u$. Consequently, whenever the inclusion holds before an update at $s<\zeta$,
\begin{equation}\label{eq:locality-bound}
 r_s(x):=|\mathcal N_G(x)\cap\mathcal D_{s-}|
 \le |\mathcal N_{\mathbf T}(x)\cap\cL_{s-}^\circ|+\mathbf1_{\{x=u\}}
 \le m_s(x),\qquad x\in\mathbf T(R).
\end{equation}
This is a spatial observation about all disagreements and does not require a claim that every disagreeing neighbor has root ancestry.

Monotonicity gives $X_s^\sigma|_{\mathbf T'}\ge\underline Y_s|_{\mathbf T'}$ for all $s$. Before $\zeta$, $\underline Y_s=\overline Y_s$ on $\mathbf T_*$. Writing
\[
 H_s^\sigma(x)=\sum_{y\sim x}X_{s-}^\sigma(y),\qquad
 H_s^{\sigma_A}(x)=\sum_{y\sim x}X_{s-}^{\sigma_A}(y),\qquad
 H_s^{\overline Y}(x)=\sum_{y\in\mathcal N_{\mathbf T}(x)}\overline Y_{s-}(y)-\mathbf1_{\{x=u\}},
\]
we therefore have $H_s^\sigma(x)\ge H_s^{\overline Y}(x)$ for every $x\in\mathbf T(R)$. In particular $x\notin\mathfrak B_{s-}$ implies $H_s^\sigma(x)>d/2$. By \eqref{eq:locality-bound},
\begin{equation}\label{eq:field-difference}
 |H_s^\sigma(x)-H_s^{\sigma_A}(x)|\le2r_s(x)\le2m_s(x).
\end{equation}
Thus, when also $m_s(x)\le d/100$, both fields are at least $0.48d$. For $g(h)=(1+e^{-2\beta h})^{-1}$, $g'(h)\le2\beta e^{-2\beta h}$ at nonnegative $h$, so
\begin{equation}\label{eq:small-disagreement-probability}
 |g(H_s^\sigma(x))-g(H_s^{\sigma_A}(x))|
 \le4\beta m_s(x)e^{-0.96C}
 \le e^{-C/2}\frac{m_s(x)}d.
\end{equation}
In all cases $\sup_h g'(h)\le\beta/2$ and \eqref{eq:field-difference} give the bound $\min\{1,Cm_s(x)/d\}$. If $m_s(x)=0$, the two updated spins agree. Hence the conditional probability $q_s(x)$ of an inner disagreement after an update is at most $p_s(x)$ in \eqref{eq:extended-update-rule}, conditional on the full joint history immediately before the update.

Here is an explicit coupling that preserves the original graphical law. At an update of $x\in\mathbf T(R)$ for which killing has not occurred before the update, first perform the original spin updates. If the two updated spins disagree, occupy $x$ in the extended process; otherwise occupy it with probability $(p_s(x)-q_s(x))/(1-q_s(x))$, using an additional fresh uniform. When $q_s(x)=1$, necessarily $p_s(x)=1$. This gives occupation probability exactly $p_s(x)$ and preserves inclusion. Use this same rule if the current update triggers $\zeta$, and then erase the stopped set. At other vertices, or if killing has already occurred before the current update, use a fresh uniform with occupation probability $p_s(x)$. At time $\zeta$ the entire stopped set is erased, which also preserves inclusion. Thus $\cL_s^\circ\subseteq\overline{\cL}_s$ for all $s$. Taking expectations and using \eqref{eq:typical-rooted-tree-estimate-relax-0} proves the lemma.

The marginal law of the reference chain remains the original law started from $\tau=\sigma|_{\mathbf T}$. The joint law of the reference chain and the extended process need not be determined by $\tau$ alone, and they need not be independent. The estimates below hold uniformly over the couplings just constructed: they use only the predictable occupation probabilities \eqref{eq:extended-update-rule} and the unchanged marginal law of the reference chain.
\end{proof}

It remains to bound the growth of $\overline{\mathcal L}_t \setminus \mathcal B(u,\mathbf T,r-1)$. Recall \eqref{eq-def-bad-cluster}. Let
\begin{equation*}
    b_s:=\#\Big\{ x\in \mathfrak B_{s-}:x\in \overline{\cL}_{s-}\text{ or }m_s(x)>0 \Big\} \,.
\end{equation*}
This includes the root whenever the root is bad, even if the root is not currently occupied, because $m_s(u)\ge 1$. Fix $\alpha:=10^{-3}$. Let $\zeta$ be the first time at which the occupied process becomes empty, and stop the following supermartingale at $\zeta$. The original extended process is not made absorbing: it may start a new episode through a later root update. We first make the following observation.
\begin{lemma}{\label{lem:super-martingale}}
    The process
    \begin{equation}{\label{eq:A.11}}
        M_r:=\exp\left\{ \alpha|\overline{\cL}_{r\wedge\zeta}| +\frac{\alpha}{8}\int_0^{r\wedge\zeta}|\overline{\cL}_q|\dd q -3\alpha\int_0^{r\wedge\zeta}b_q\dd q \right\}
    \end{equation}
    is a nonnegative supermartingale. 
\end{lemma}
\begin{proof}
    Let $\mathfrak G_s$ be the predictable generator of the localized occupied process, before it hits the empty set. Then it suffices to show that for every nonempty $D\subseteq V(\T_*)$,
    \begin{equation}{\label{eq:A.6}}
        e^{-\alpha|D|} \cdot \mathfrak G_s\big( e^{\alpha|D|} \big) \le \Big( -\frac{\alpha}{8}|D|+3\alpha b_s \Big) \,.
    \end{equation}
    For a deterministic set $D \subset V(\T_{*})$, set
    \[
        m_D(x):=|\mathcal N_{\T}(x)\cap D|+\mathbf 1_{\{x=u\}} \,.
    \]
    In the following calculation, $p_s(x)$ and $b_s$ are evaluated with current occupied set $D$. Root $\T$ at $u$. At a non-root vertex, at most one contribution to $m_D(x)$ comes from its parent; at $u$, the extra contribution $\mathbf 1_{\{x=u\}}$ plays the same role. Thus $m_D(x)>d/100$ implies that $x$ has more than $d/100-1\ge d/200$ children in $D$, for $d\ge200$. Since every vertex of $D$ is a child of at most one vertex,
    \begin{equation}{\label{eq:A.4-1}}
        \#\Big\{ x\in V(\T_{*}):m_D(x)>d/100 \Big\} \le \frac{200}{d}|D| \,.
    \end{equation}
    Also, $\sum_{x\in V(\T_*)}m_D(x)\le d|D|+1$. By \eqref{eq:extended-update-rule}, at a vertex that is neither bad nor has $m_D(x)>d/100$, the occupation probability is $e^{-C/2}m_D(x)/d$. All remaining probabilities are at most one, and every bad vertex with a nonzero occupation probability is counted by $b_s$. Hence \eqref{eq:A.4-1} gives
    \begin{align*}
        \sum_{x\in V(\T_*)}p_s(x)
        \le \frac{e^{-C/2}}{d}(d|D|+1)+\frac{200}{d}|D|+b_s \le \left(2e^{-C/2}+\frac{200}{d}\right)|D|+b_s \,.
    \end{align*}
    The last inequality uses $|D|\ge1$, which is why we stop at the first visit to the empty set. Every birth increases the cardinality by one, and every death decreases it by one. Therefore
    \begin{align*}
        e^{-\alpha|D|}\mathfrak G_s\big(e^{\alpha|D|}\big)
        &= (e^\alpha-1)\sum_{x\in V(\T_*)\setminus D}p_s(x) +(e^{-\alpha}-1)\sum_{x\in D}(1-p_s(x)) \\
        &\le -(1-e^{-\alpha})|D| +(e^\alpha-1)\sum_{x\in V(\T_*)}p_s(x) \\
        &\le \alpha\left(-\frac12+4e^{-C/2}+\frac{400}{d}\right)|D|+2\alpha b_s \\
        &\le -\frac{\alpha}{4}|D|+2\alpha b_s \le -\frac{\alpha}{8}|D|+3\alpha b_s \,.
    \end{align*}
    Here we used $1-e^{-\alpha}\le e^\alpha-1$, $1-e^{-\alpha}\ge\alpha/2$, $e^\alpha-1\le2\alpha$, and \eqref{eq-assump-C-d-epsilon}. This proves \eqref{eq:A.6}. Multiplying by the finite-variation exponential factor in \eqref{eq:A.11} gives a nonpositive predictable drift up to $\zeta$. The stopped process is a nonnegative local supermartingale, and hence a supermartingale. The same predictable calculation applies after any stopping time at which the occupied set is nonempty, conditionally on the entire joint history and without restarting the local reference chain.
\end{proof}

We then record the deterministic consequence of Definition~\ref{def-mathcal-A-2-lem-5.4} that is used below.
\begin{lemma}\label{lem:typical-rooted-tree-growth}
    The following holds uniformly in $\mathbf T \in\mathscr T_{\ell}$. If $\rho(\mathbf T)=u$ where $(u,v)\in E_A$ and $\sigma \in\{\pm1\}^{V}$ is typical with respect to $\mathbf T$, then
    \begin{equation}\label{eq:typical-rooted-tree-estimate}
        \E_{\sigma}\Big[ |\mathbf T(R)\cap \mathcal L_{u,t} \setminus \mathcal B(u,\T,r-1) | \cdot \mathbf 1_{\{(u,v) \notin \mathcal V\}} \Big] \le e^{-2t/C^3} \le 10^{-10} \,.
    \end{equation}
\end{lemma}
\begin{proof}
    By Lemma~\ref{lem:bound-by-extended-spacetime-lem-5.4}, it suffices to prove
\begin{equation}\label{eq:typical-rooted-tree-estimate-goal}
 \E_\tau\big[|\overline{\cL}_t\setminus\mathcal B(u,\mathbf T,r-1)|\big]\le e^{-2t/C^3},
 \qquad \tau=\sigma|_{\mathbf T}.
\end{equation}
All probabilities below are under the original local reference chain started from $\tau$. In particular, this reference chain is not restarted when the occupied process enters a new episode. An episode starts at time $0$, with state $\{u\}$, or at a later transition from $\emptyset$ to $\{u\}$, and ends at the next visit to $\emptyset$. The entrance times are stopping times. Their counting process, excluding time zero, has predictable intensity at most $C/d$, since only a root update can create an occupation from the empty set. 

Fix an episode entrance $s<t$, and kill this episode at its ending time $\zeta_s$. Put $Z_{s,t}=|\overline{\cL}_t|\mathbf1_{\{\zeta_s>t\}}$. Uniformly in the entire joint history at $s$, its birth rate while nonempty is at most
\[
 (C/d)\sum_xm_q(x)\le C|\overline{\cL}_q|+C/d\le2C|\overline{\cL}_q|.
\]
Comparison with a rate-$2C$ Yule process started from one particle gives
\begin{equation}\label{eq:A.13}
 \E_\tau[Z_{s,t}^2\mid\mathcal F_s]\le2e^{4C(t-s)}.
\end{equation}
Its total jump rate, including deaths, is at most $3C|\overline{\cL}_q|$ before $\zeta_s$.
On survival to $t$, let $i$ be the first block meeting $[s,t]$, and form
\begin{align*}
 U_j&=\bigcup_{q\in I_j\cap[s,t]}\overline{\cL}_q,\\
 B_j&=\bigcup_{q\in I_j\cap[s,t]}\{x\in\mathfrak B_q:x\in\overline{\cL}_q\text{ or }m_q(x)>0\},\\
 D_j&=U_j\cup B_j\cup\{u\},\qquad i\le j\le T.
\end{align*}
The same definitions use pre-update occupation where needed at a jump time; adding the finitely many left limits does not change any of the estimates. Every newly occupied vertex attaches to an occupied neighbor or the root source. Every added vertex in $B_j$ attaches to $U_j$ or is the root. Thus the first union is connected to $u$, and every component of the next union intersects the preceding union or the permanently included copy of $u$. Hence $(B_i,D_i,\ldots,B_T,D_T)\in\mathcal O_i(\mathbf T)$.

Write $K=\int_s^t|\overline{\cL}_q|\,dq$, let $J$ be the number of occupation changes during this episode up to $t$, and put $H=\sum_j|B_j|$, $M=\sum_j|D_j|$. These integrals and counts are stopped at $\zeta_s$ off the survival event. With $\eta=10^{-4}C^{-2}$, partition the survival event into
\[
 \mathsf F_0(s)=\{H<\eta M,\ J\le20CK\},\quad
 \mathsf F_1(s)=\{H\ge\eta M\},\quad
 \mathsf F_2(s)=\{H<\eta M,\ J>20CK\}.
\]
We first bound $\mathsf F_0(s)$ and $\mathsf F_2(s)$ conditionally on $\mathcal F_s$; these bounds do not use typicality. An occupied time interval of length $a$ meets at most $a/\Delta+2$ blocks. There are at most $J+1$ such intervals, initially counting the root. Consequently $\sum_j|U_j|\le K/\Delta+2J+2$. There are at most $K/\Delta+2$ blocks on survival, since $K\ge t-s$. On $\mathsf F_0(s)$ this implies
\[
 (1-\eta)M\le2K/\Delta+2J+4,\qquad
 M\le100C(K+1),\qquad
 \int_s^t b_q\,dq\le\Delta H\le K/100+1/100.
\]
Apply Lemma~\ref{lem:super-martingale} from the stopping time $s$, stopping at $\zeta_s$. Its proof is a predictable generator estimate, so it applies conditionally on every joint history at $s$. Since $|\{u\}|=1$, $K\ge t-s$ on survival, and $z\le\alpha^{-1}e^{\alpha z}$,
\begin{equation}\label{eq:A.19}
 \E_\tau[Z_{s,t}\mathbf1_{\mathsf F_0(s)}\mid\mathcal F_s]
 \le\alpha^{-1}e^{\alpha+3\alpha/100-\alpha(t-s)/12}
 \le\frac2\alpha e^{-c(t-s)},\qquad c:=\alpha/12.
\end{equation}
For $\mathsf F_2(s)$ use the exponential counting-process supermartingale
\[
 \exp\left\{J_{s,q}-(e-1)\int_s^{q\wedge\zeta_s}3C|\overline{\cL}_v|\,dv\right\},
\]
where $J_{s,q}$ counts changes up to $q\wedge\zeta_s$. On $\mathsf F_2(s)$ at time $t$ its exponent is at least $14CK\ge14C(t-s)$. Therefore its conditional probability is at most $e^{-14C(t-s)}$. By \eqref{eq:A.13} and conditional Cauchy--Schwarz,
\begin{equation}\label{eq:A.24}
 \E_\tau[Z_{s,t}\mathbf1_{\mathsf F_2(s)}\mid\mathcal F_s]\le2e^{-5C(t-s)}.
\end{equation}

Set $h=\lfloor r\rfloor$ and $a=h/(100C)$. Then $h\ge r/2\ge1$ and $a<t$. At time $t$ there is at most one surviving episode. If it has age at least $a$, its entrance belongs to $[0,t-a]$. Sum \eqref{eq:A.19} and \eqref{eq:A.24} over actual entrances in that interval, including time zero. First condition at each entrance, and only then use its predictable intensity bound. Since $5C\ge c$ and $2/\alpha+2\le3/\alpha$, the contribution from old episodes in $\mathsf F_0\cup\mathsf F_2$ is at most
\begin{equation}\label{eq-bound-good-entrance}
 \frac3\alpha e^{-ct}+\frac{3C}{\alpha d}\int_0^{t-a}e^{-c(t-s)}\,ds
 \le\frac3\alpha e^{-ct}+\frac{3C}{\alpha dc}e^{-ca}.
\end{equation}
This compensation step uses only the history-uniform conditional estimates just proved.

For an old episode in $\mathsf F_1$, we instead take an unconditional union bound under the original law started from $\tau$. Its realized skeleton belongs to $\mathcal O_i^\diamond(\mathbf T)$ and has $M\ge\lfloor a/(2\Delta)\rfloor$. For any fixed such skeleton, typicality gives
\[
 \Pr_\tau(B_j\subseteq\mathfrak B_{I_j}\text{ for all }j)
 \le e^{-10^{-6}C^{-2}dH}\le e^{-10^{-10}dM/C^4}.
\]
Lemma~\ref{lem:-enu-sequence-longrightarrow-forest}, the at most $2^M$ choices of the $B_j$'s, and the at most $2^{M-1}$ positive compositions give at most $3(512d)^M$ skeletons with total size $M$ for a fixed initial block. There are at most $T$ initial blocks. Thus the probability of an old surviving episode in $\mathsf F_1$ is at most
\[
 3T\sum_{M\ge\lfloor a/(2\Delta)\rfloor}(512d)^M e^{-10^{-10}dM/C^4}
 \le e^{-da/(10^{12}C^4)}.
\]
The last geometric-series estimate follows from the working bounds after \eqref{eq-def-ell-t-given-A}; in particular $a\ge t/(4\cdot10^5 C\log C)$ and $T\le101t$.
For the full occupied process, including all entrances, $|\overline{\cL}_q|+1$ is dominated by a rate-$C$ Yule process started from two particles, since its total birth rate is at most $C(|\overline{\cL}_q|+1)$. Hence $\E_\tau[|\overline{\cL}_t|^2]\le6e^{2Ct}\le6e^{4Ct}$. Cauchy--Schwarz bounds the contribution of these old $\mathsf F_1$ episodes by
\begin{equation}\label{eq:A.22}
 \sqrt6\,e^{2Ct-da/(2\cdot10^{12}C^4)}\le e^{-t}.
\end{equation}
No conditional typicality at the random entrance is asserted in this argument.

Finally, consider an episode younger than $a$. At each birth of the extended process, choose its parent uniformly from the $m_q(x)$ occupied neighbors and, at the root, the additional source. The resulting birth-arrow intensity on each oriented edge is at most $p_q(x)/m_q(x)\le C/d$. Every vertex of the young episode has ancestry from the root entirely after time $t-a$. On a tree, reaching a vertex at distance $j$ requires $j$ ordered arrows on its unique root path, even if the full ancestral walk has backtracking. Iterated conditioning on the bounded arrow intensities therefore gives
\begin{equation}\label{eq:finite-speed-short-age}
\begin{aligned}
 &\E_\tau\big[|\overline{\cL}_t\setminus\mathcal B(u,\mathbf T,r-1)|;
                 \text{the terminal episode has age }<a\big]\\
 &\qquad\le\sum_{j\ge h}d^j\frac{(Ca/d)^j}{j!}
 \le\sum_{j\ge h}\left(\frac{eCa}{j}\right)^j
 \le2\cdot4^{-h}.
\end{aligned}
\end{equation}
The definition $h=\lfloor r\rfloor$ matches the integer distances outside the radius-$r-1$ ball.

Adding \eqref{eq-bound-good-entrance}, \eqref{eq:A.22}, and \eqref{eq:finite-speed-short-age} yields
\[
 \E_\tau[|\overline{\cL}_t\setminus\mathcal B(u,\mathbf T,r-1)|]
 \le\frac3\alpha e^{-ct}+\frac{3C}{\alpha dc}e^{-ca}+e^{-t}+2\cdot4^{-h}
 \le e^{-2t/C^3}.
\]
This proves both assertions of the lemma.
\end{proof}

Now we provide the proof of Lemma~\ref{lem-growth-rooted-tree-subtree}.
\begin{proof}[Proof of Lemma~\ref{lem-growth-rooted-tree-subtree}]
    Based on Lemma~\ref{lem-mathcal-A-1-2-is-typical-lem-5.4}, it suffices to show that for all $\sigma \in \overline{\mathcal A}_{1} \cap \overline{\mathcal A}_{2}$, we have the conclusion of Lemma~\ref{lem-growth-rooted-tree-subtree} holds. Fix $\sigma\in\overline{\mathcal A}_{1} \cap \overline{\mathcal A}_{2}$ and retain the arbitrary set $A$ with $|A|=pn$. Split
    \begin{align*}
        H_A :=\big\{ (u,v)\in E_A:\sigma|_{\mathbf T_u}\mbox{ is typical} \big\}, \quad H_A^c:= E_A\setminus H_A \,.
    \end{align*}
    Because the child trees in $E_A$ are pairwise disjoint, from $\sigma\in\overline{\mathcal A}_{2}$ we have $|H_A^c|\le e^{-dt/(10^{20}C^5)}n$. For $(u,v) \in H_A$, Lemma~\ref{lem:typical-rooted-tree-growth} yields 
    \[
        \E_\sigma\Big[ \big| \mathbf T_u(R) \cap \mathcal L_{u,t} \setminus \mathcal B(u,\mathbf T_u,r-1) \big| \cdot \mathbf 1_{\{(u,v)\notin \mathcal V\}} \Big] \le e^{-2t/C^3}\,.
    \]
    Thus, the total typical-tree contribution is at most $200e^{-2t/C^3}d\ell pn$. For atypical trees, by Claim~\ref{lem-stochastic-doination-by-infection}, Lemma~\ref{lem-growth-infectious-process}, and Cauchy-Schwarz inequality, we have
    \begin{align*}
        &\E_\sigma\Bigg[ \sum_{(u,v)\in E_A} \big| \mathbf T_u(R) \cap \mathcal L_{u,t} \setminus \mathcal B(u,\mathbf T_u,r-1) \big| \mathbf 1_{\{(u,v)\notin \mathcal V\}} \mathbf 1_{ \{ (u,v) \in H_A^c \} } \Bigg] \\
        \le\ & \E_{\sigma}\Big[ |H_A^c| \Big]^{1/2} \E_{\sigma}\Bigg[ \sum_{(u,v)\in E_A} |\mathbf T_u(R) \cap \mathcal L_{u,t}|^2 \mathbf 1_{\{(u,v)\notin \mathcal V\}} \bigg]^{1/2} \\
        \le\ & \big( e^{-dt/(10^{20}C^5)}n \big)^{1/2} \cdot \big( 6e^{2Ct} \cdot 200d\ell pn \big)^{1/2} \overset{\eqref{eq-assump-C-d-epsilon}}{\le} e^{-t} pn \,.
    \end{align*}
    This proves Lemma~\ref{lem-growth-rooted-tree-subtree}.
\end{proof}

\subsubsection{Proof of Lemma~\ref{lem-mathcal-A-1-2-is-typical-lem-5.4}(1)}

Let
\[
    b:=\lceil0.1d\rceil, \qquad h:=\ell-R_0, \qquad s:=b^{h}, \qquad \gamma:=\frac{C}{1000} \,.
\]
We first count the possible certificates that a fixed tree is not good. If $\mathbf T\in\mathscr T_{\ell}$ is not good, then there is some $u\in\partial_{R_0} \mathbf T$ and a complete $b$-ary witness subtree $\mathbf Q$ below $u$ of height $\ell-R_0$ whose terminal leaves are all minus (with no spin restriction on its internal vertices). Let $\mathsf L(\mathbf Q)$ be the set of leaves of $\mathbf Q$. Then $|\mathsf L(\mathbf Q)|=b^{\ell-R_0}=s$ and $\mathsf L(\mathbf Q) \subseteq \mathcal M(\sigma)$, where $\mathcal M(\sigma):=\{v:\sigma(v)=-1\}$. There are at most $d^{R_0}$ choices for $u$. A complete $b$-ary tree with $s=b^{\ell-R_0}$ leaves has
\begin{align*}
    1+b+\cdots+b^{\ell-R_0-1}=\frac{s-1}{b-1}
\end{align*}
internal vertices. At each internal vertex, its $b$ witness-children can be chosen in at most $\binom{d}{b}$ ways. Consequently, the number of possible witnesses for a fixed $\mathbf T$ is at most
\[
    d^{R_0}\binom{d}{b}^{(s-1)/(b-1)} \le e^{5s} \,.
\]
If $\overline{\mathcal A}_{1}^c$ holds, then there exist pairwise vertex-disjoint trees $\mathbf T_{1},\ldots,\mathbf T_{m}\in\mathscr T_{\ell}$ with $m=\lfloor e^{-\gamma s} n \rfloor +1$ and each restriction $\sigma|_{\mathbf T_i}$ is not good. Choose one witness $\mathbf Q_i$ in each $\mathbf T_i$. Since the trees are vertex-disjoint, the terminal sets $\mathsf L(\mathbf Q_i)$ are disjoint. Thus $M:=\cup_{i=1}^{m} \mathsf L(\mathbf Q_i)$ satisfies $|M|=ms$. Whenever $ms\leq C^{-2}n$, Lemma~\ref{lem-est-bad-minus-cluster-fixed-time} gives
\begin{equation*}
    \pi_{+}\bigl(M\subseteq\mathcal M(\sigma)\bigr) \leq 3e^{-Cms/9} \le e^{-0.1Cms} \,.
\end{equation*}
Also recall that $|\mathscr T_{\ell}|\leq dn$. A union bound over the choices of the trees and their witnesses therefore yields
\begin{align}{\label{eq-bound-prob-mathcal-A-1-lem-5.4}}
    \pi_{+}(\overline{\mathcal A}_{1}^c) \leq \binom{dn}{m} e^{5sm} e^{-0.1Cms} \overset{\eqref{eq-assump-C-d-epsilon}}{\leq} \exp\left\{  m\left(1+d\log2+\log\frac{n}{m}-\frac{C}{20}s\right) \right\} \,.
\end{align}
We consider two cases.
    
\medskip
\noindent {\bf Case 1:} $e^{-0.001Cs}n\geq1$. Then $m\leq2e^{-0.001C s}n$, and
\begin{align*}
    ms\leq 2s e^{-0.001Cs}n \overset{\eqref{eq-assump-C-d-epsilon}}{\leq} C^{-2}n \,,
\end{align*}
so Lemma~\ref{lem-est-bad-minus-cluster-fixed-time} is applicable. Since $m>e^{-0.001Cs}n$, Equation~\eqref{eq-bound-prob-mathcal-A-1-lem-5.4} gives
\begin{align*}
    \pi_{+}(\overline{\mathcal A}_{1}^c) \le \exp( -48 \times 0.001Csm ) \le n^{-48} \,,
\end{align*}
where the last inequality follows from $1 \le 0.001Cs\leq\log n$ in this case and the fact that $x\mapsto xe^{-x}$ is decreasing on $[1,\infty)$. 

\medskip
\noindent {\bf Case 2:} $e^{-0.001Cs}n<1$. Here $m=1$, and thus for all sufficiently large $n$, $ms=s\le d^\ell \leq C^{-2}n$, so Lemma~\ref{lem-est-bad-minus-cluster-fixed-time} is again applicable. Since $0.001Cs>\log n$, Equation~\eqref{eq-bound-prob-mathcal-A-1-lem-5.4} gives
\begin{align*}
    \pi_{+}(\overline{\mathcal A}_{1}^c) &\leq \exp( 1+d\log2+\log n-0.05Cs ) \\
    &< \exp( 1+d\log2-49\log n) \leq n^{-24} \,.
\end{align*}
Combining the two cases proves Lemma~\ref{lem-mathcal-A-1-2-is-typical-lem-5.4}(1).

\subsubsection{Proof of Lemma~\ref{lem-mathcal-A-1-2-is-typical-lem-5.4}(2)}

The key to proving Lemma~\ref{lem-mathcal-A-1-2-is-typical-lem-5.4}(2) is to set up the following proposition. Let \(\mathbf T \sim \mathbb T_{1,\ell}\). The ordinary depth-\(\ell\) leaves carry a boundary condition \(\eta\in\{\pm1\}^{\partial_{\ell}\mathbf T}\), and the attached root leaf carries sign \(-1\). Also, recall Definition~\ref{def-mathcal-A-1-lem-5.4}. Note that whether $\sigma\in\{ \pm 1 \}^V$ is good for $\mathbf T$ only depends on the spins on $\partial_{\ell}\mathbf T$. Let \(\mu_{\mathbf T^\circ}^{\eta}\) be the Ising Gibbs measure on the unfrozen vertices of \(\mathbf T^\circ\), with ordinary boundary \(\eta\) and attached root sign \(-1\). Let $\mu_{\mathbf T'}^{\eta} :=(\mu_{\mathbf T^\circ}^{\eta})_{V(\mathbf T')}$ be its marginal on \(\mathbf T'\). The descendants below level \(R'+1\) are therefore represented by effective cavity fields on \(\partial_{R'+1}\mathbf T\). 
\begin{lemma}\label{lem:one-tree}
    Uniformly over every good \(\eta\) and every \(\mathbf T \in \mathscr T_{\ell}\),
    \[
        \mu_{\mathbf T'}^{\eta}\bigl( \tau \mbox{ is atypical for } \mathbf T \big) \le e^{-dt/(10^{15}C^5)} \,.
    \]
\end{lemma}
We first present the proof of Lemma~\ref{lem-mathcal-A-1-2-is-typical-lem-5.4}(2) given Lemma~\ref{lem:one-tree}.
\begin{proof}[Proof of Lemma~\ref{lem-mathcal-A-1-2-is-typical-lem-5.4}(2) assuming Lemma~\ref{lem:one-tree}]
Since $\pi_+=\pi(\cdot\mid\sum_v\sigma(v)\ge0)$ and the conditioning event has probability at least $1/2$, it suffices to bound the unconditioned probability and multiply by two. Put $a=dt/(10^{20}C^5)$, $r=\lfloor e^{-a}n\rfloor+1$, and $m=\lceil r/2\rceil$. If $\overline{\mathcal A}_2^c\cap\overline{\mathcal A}_1$ occurs, there are $r$ pairwise vertex-disjoint atypical trees, of which at most $e^{-0.001Cs}n\le\tfrac12e^{-a}n$ are non-good. The last comparison follows from $s=\lceil0.1d\rceil^{\ell-R_0}$ and the parameter bounds after \eqref{eq-def-ell-t-given-A}. Thus at least $m$ of these trees are both good and atypical.

Fix $m$ disjoint trees $\mathbf T_1,\ldots,\mathbf T_m$. Condition on their ordinary outer-leaf spins and on all vertices outside their free interiors. Goodness is determined by these conditioned outer-leaf spins. The remaining interiors need not be independent: the attached root neighbor of one tree can be the root of another. No such root is additionally conditioned on. Instead compare the conditional Gibbs law on the union of free interiors with the product law that fixes the attached neighbor of every root to minus. At a non-root free vertex all incident edges stay within its tree or meet an already fixed ordinary leaf. Any edge between two distinct free interiors must therefore join their roots. In the actual conditional law each external root spin is at least minus, whereas the comparison law uses minus there. The monotone heat-bath coupling, with the same fixed ordinary-leaf data, shows that the actual conditional law stochastically dominates this product law.

Atypicality is decreasing in the free initial spins: the plus-boundary local chain is monotone in its initialization, and every bad-block event is decreasing. Conditional on good leaf data, Lemma~\ref{lem:one-tree} and the product comparison therefore give
\[
    \pi(\mathbf T_1,\ldots,\mathbf T_m\text{ all good and atypical}) \le e^{-dtm/(10^{15}C^5)}.
\]
Integrating over the fixed boundary data preserves this bound. There are at most $dn$ rooted trees in $\mathscr T_\ell$, so
\begin{equation}\label{eq-bound-prob-mathcal-A-2-lem-5.4}
    \pi(\overline{\mathcal A}_2^c\cap\overline{\mathcal A}_1) \le\binom{dn}{m}e^{-10^5am} \le e^{-5\cdot10^4am} \,.
\end{equation}
Indeed $m>e^{-a}n/2$, so $\log(edn/m)\le a+\log(2ed)\le2a$ in our regime. If $a\le\log n$, then $a\ge1$ and monotonicity of $xe^{-x}$ on $[1,\infty)$ gives $am\ge\tfrac12na e^{-a}\ge\tfrac12\log n$. If $a>\log n$, then $m=1$ and $am>\log n$. In either case twice the right-hand side of \eqref{eq-bound-prob-mathcal-A-2-lem-5.4} is at most $n^{-24}$ for sufficiently large $n$, proving the claim.
\end{proof}

We now focus on the proof of Lemma~\ref{lem:one-tree}. Let $\mu_{\T'}^{+}$ be the measure obtained by taking $\eta$ to be the plus boundary condition in $\mu_{\mathbf T'}^{\eta}$. We first bound the Radon-Nikodym derivative between $\mu_{\T'}^{+}$ and $\mu_{\mathbf T'}^{\eta}$, as incorporated in the following lemma.
\begin{lemma}\label{lem:screening}
    For every good boundary condition \(\eta\), and every \(\T\simeq \mathbb T_{1,\ell}\), we have the pointwise estimate
    \[
        \exp(-\varepsilon_R) \le \frac{\dd\mu_{\T'}^{\eta}}{\dd\mu_{\T'}^{+}} \le \exp(\varepsilon_R), \quad \varepsilon_R \le e^{-0.1R'} \,.
    \]
\end{lemma}
\begin{proof}
    For an oriented edge \(v\to p(v)\) where $p(v)$ is the parent of $v$, write \(h_{v\to p(v)}\) for the upward cavity field and \(m_{v\to p(v)}=\tanh h_{v\to p(v)}\). The cavity field recursion is given by
    \begin{align*}
        h_{v\to p(v)} = \sum_{w:w\to v} \psi_\beta(h_{w\to v}), \quad \psi_\beta(h) = \tanh^{-1}\big( \tanh\beta \tanh h \big) \,.
    \end{align*}
    Goodness and Lemma~\ref{lem-amenable-boundary-condition-positive-field} imply that every upward message sent by a vertex at level \(R_0\) has cavity field at least \(0.1C\), and hence magnetization message at least \(\tanh(0.1C)>0.1\). Fix \(y\in\partial_{R'+1}\T\). Between \(y\) and level \(R_0\) there are $L:=R_0-R'-1$ generations. Applying the BP contraction estimate of Lemma~\ref{lem-exponential-contraction-BP-fixed-point} once to the good boundary and once to the all-plus boundary gives
    \begin{align*}
        \big| m_{y\to p(y)}^{\eta}-m_* \big| + \big| m_{y\to p(y)}^{+}-m_* \big| \le 2^{1-L}  \,.
    \end{align*}
    After one BP iteration both magnetization messages are at most $\tanh C$, so the derivative of $\operatorname{arctanh}$ is at most $\cosh^2C\le e^{2C}$ on the interval between them. Therefore $|h_y^\eta-h_y^+| \le e^{2C} 2^{1-L}$, where \(h_y^\eta\) and \(h_y^+\) are the effective descendant fields at \(y\). After integrating out all descendants below level \(R'+1\), the two measures differ only through the potential $V(\tau) = \sum_{y\in\partial_{R'+1}\T} (h_y^\eta-h_y^+)\tau_y$. Therefore
    \begin{align*}
        \osc(V) \le 2|\partial_{R'+1}\T| \cdot e^{2C}2^{1-L} \le 4d^{R'+1} e^{2C} 2^{-(R_0-R'-1)} \overset{\eqref{eq-def-R-R'}, \eqref{eq-assump-C-d-epsilon}}{\le} e^{-0.1R'} \,.
    \end{align*}
    Normalization of the density gives the pointwise ratio bound.
\end{proof}

For the next estimates, let $\mu_+:=\phi_{1,\ell}^+$ on the full tree $\mathbf T^\circ$, and let $(Y_s)_{s\ge0}$ be its stationary Glauber dynamics. Ordinary depth-$\ell$ leaves are fixed plus and the attached root leaf is fixed minus. For a configuration $\tau$ on this full tree, define
\[
    \cB_s:=\left\{v\in V(\mathbf T_*): \sum_{w\in\mathcal N_{\mathbf T}(v)}Y_s(w)-\mathbf1_{\{v=u\}}\le0.5d\right\}, \qquad \cB_{[a,b]}:=\bigcup_{a\le s\le b}\cB_s \,.
\]
The process $Y$ here is the full stationary chain, not a chain having the marginal law $\mu_{\mathbf T'}^+$ as its invariant law. Define the stronger bad set
\begin{align*}
    \cB^\sharp(\tau) := \Big\{ v\in V(\T_*): \sum_{ \substack{ w\sim v , w\in V(\T) }} \tau(w) - \mathbf 1_{\{v=u\}} \le 0.6d \Big\} \,.
\end{align*}
In addition, for a time interval \(I\) of length \(\Delta\), let $\mathcal U_I$ be the collection of $v\in V(\T_*)$ such that at least $0.05d$ distinct dynamic neighbors of \(v\) receive an update during \(I\). We first bound the tail probability of $|\mathcal U_I|$.
\begin{lemma}\label{lem:update-set}
    Independently of the spin configuration, for every \(W\subseteq V(\T_*)\) and every interval \(I\) of length \(\Delta\),
    \[
        \mathbb P(W\subseteq\mathcal U_I) \le \exp\{-0.01d|W|\} \,.
    \]
\end{lemma}
\begin{proof}
    Let \(Z\) be the set of dynamic vertices that receive at least one update during \(I\). The indicators \(\{\mathbf 1_{\{u\in Z\}}\}\) are independent Bernoulli variables with parameter $1-\e^{-\Delta}<0.011$. On \(W\subseteq\mathcal U_I\), there are at least \(0.05d|W|\) oriented incidences from \(W\) to \(Z\). Incidences along edges internal to \(W\) contribute at most \(2|W|\). With \(Z_0:=(Z\cap\mathcal N_\T(W))\setminus W\), the forest bound gives at most \(|W|+|Z_0|-1\) edges from \(W\) to \(Z_0\). Therefore $|Z_0|\ge(0.05d-3)|W|\ge0.04d|W|$ for sufficiently large \(d\). Since \(Z_0\subseteq \mathcal N_\T(W)\setminus W\), a set of size at most \(d|W|\), it follows from a union bound that
    \begin{equation*}
        \mathbb P\big( W\subseteq\mathcal U_I \big) \le \mathbb P\Big(    \Bin(d|W|,0.011)\ge0.04d|W| \Big) \le \exp(-0.01d|W|) \,. \qedhere
    \end{equation*}
\end{proof}

We now give a tail bound for the bad block $\cB_{[u,u+\Delta]}$. 
\begin{lemma}\label{prop:one-block}
    For every \(B\subseteq V(\T_*)\), we have
    \begin{align*}
        \Pb_{\mu_{+}}\big( B\subseteq\cB_{[u,u+\Delta]} \big) \le \exp(-0.009d|B|) \,.
    \end{align*}
\end{lemma}
\begin{proof}
    Suppose \(v\in\cB_s\) for some \(s\in[u,u+\Delta]\). If \(v\notin\cB^\sharp_u\), then its initial neighbor sum is greater than \(0.6d\), while at time \(s\) it is at most \(0.5d\). Since changing one neighbor spin changes the sum by at most \(2\), at least \(0.05d\) \emph{distinct} dynamic neighbors of \(v\) must have received an update during \([u,s]\). Hence $\cB_{[u,u+\Delta]} \subset \cB^\sharp_u \cup \mathcal U_{[u,u+\Delta]}$. In addition, the initial spins and the Poisson clocks are independent. Therefore, by summing over the set \(S\subseteq B\) and using Remark~\ref{remark:same-est-extend-to-tree} and Lemma~\ref{lem:update-set}, we get
    \begin{align*}
        \Pb_{\mu_{+}}\big( B\subseteq\cB_{[u,u+\Delta]} \big) \le \sum_{S\subseteq B} e^{-0.01Cd|S|} e^{-0.01d|B\setminus S|} \le e^{-0.009d|B|} 
    \end{align*}
    for sufficiently large \(C,d\).
\end{proof}

Now we provide the proof of Lemma~\ref{lem:one-tree}, thereby finishing the proof of Lemma~\ref{lem-mathcal-A-1-2-is-typical-lem-5.4}(2).
\begin{proof}[Proof of Lemma~\ref{lem:one-tree}]
By Lemma~\ref{lem:screening}, it suffices to prove
\begin{equation}\label{eq:one-tree-goal}
    \mu_{\mathbf T'}^+(\tau\text{ is atypical})\le\tfrac12 e^{-dt/(10^{15}C^5)} \,,
\end{equation}
since $e^{\varepsilon_R}\le2$. Sample a full initial configuration $\zeta\sim\mu_+$, and couple its full stationary chain $Y$ to the local chain $\overline Y$ initialized from $\tau=\zeta|_{\mathbf T'}$. The local chain has fixed plus spins at its artificial outer boundary, and both chains have the same fixed minus root neighbor. Monotonicity gives $\overline Y_s\ge Y_s$ throughout $\mathbf T'$ for all $s\ge0$. Every neighbor entering a bad-set test in $\mathbf T_*$ belongs to $\mathbf T'$. Hence, for every deterministic full configuration $\zeta$ extending $\tau$ and every skeleton $\Gamma=(B_i,D_i,\ldots,B_T,D_T)$,
\[
    \Pr_\tau(B_j\subseteq\mathfrak B_{I_j}\text{ for all }j) \le\Pr_\zeta(B_j\subseteq\cB_{I_j}\text{ for all }j) \,.
\]
This inequality, valid for every extension, is what allows us to use the full-tree log-Sobolev inequality; no stationarity of the truncated plus-boundary chain under $\mu_{\mathbf T'}^+$ is assumed.

Fix $1\le i\le T$ and $\Gamma\in\mathcal O_i^\diamond(\mathbf T)$, and put $K=\sum_{j=i}^T|B_j|$, $M=\sum_{j=i}^T|D_j|$. Then $K\ge10^{-4}C^{-2}M$ and $M\ge T-i+1\ge1$. Write $E_\Gamma=\{B_j\subseteq\cB_{I_j}\text{ for every }i\le j\le T\}$ and $p_\Gamma=\Pr_{\mu_+}(E_\Gamma)$. Lemma~\ref{lem-log-Sobolev-Tkl} gives $c_{\mathsf{LS}}(\mu_+)\ge\rho:=1/(10C)$. Lemmas~\ref{lem-almost-indepdent-in-time-ls} and \ref{prop:one-block} therefore give
\[
    p_\Gamma\le\left(\prod_{j=i}^T e^{-0.009d|B_j|}\right)^{10^{-3}\rho} \le e^{-9\cdot10^{-7}dK/C} \,.
\]
If $p_\Gamma=0$, the event has zero conditional probability for every full initial state, since the finite-volume Gibbs measure has full support. Such a skeleton cannot witness atypicality. Otherwise $0<p_\Gamma<1$, and, for sufficiently large $C$,
\[
    e^{-10^{-6}C^{-2}dK}\ge p_\Gamma^{1/2} \,.
\]
Failure of typicality for this skeleton means that the local conditional probability is strictly greater than the threshold on the left. For every full extension $\zeta$ it then follows that $\Pr_\zeta(E_\Gamma)\ge p_\Gamma^{1/2}$. The event $E_\Gamma$ begins at time $(i-1)\Delta$, not $i\Delta$. Applying Lemma~\ref{lem-BAD-lift-ls} at this time and using $K\ge10^{-4}C^{-2}M$ bounds the $\mu_{\mathbf T'}^+$-measure of these local initial states by
\[
    p_\Gamma^{\varphi_i}\le\exp\left\{-\frac{9\cdot10^{-11}d}{C^3}M\varphi_i\right\}, \qquad \varphi_i:=\frac16+\frac1{12}e^{2\rho(i-1)\Delta} \,.
\]
For fixed $i$ and $M$, there are at most $3(512d)^M$ skeletons: apply Lemma~\ref{lem:-enu-sequence-longrightarrow-forest}, then choose all $B_j\subseteq D_j$ in at most $2^M$ ways and the positive sizes of the $D_j$'s in at most $2^{M-1}$ ways. Set $\lambda=9\cdot10^{-11}d/C^3$. Our working bounds imply $\log(512d)\le\lambda\varphi_i/2$ and $\lambda\varphi_i/2\ge\log2$. A union bound and a geometric sum thus give
\[
    \mu_{\mathbf T'}^+(\tau\text{ is atypical}) \le\sum_{i=1}^T\sum_{M\ge T-i+1}3(512d)^M e^{-\lambda M\varphi_i} \le6\sum_{i=1}^T e^{-\lambda(T-i+1)\varphi_i/2} \,.
\]
There is no zero-mass terminal skeleton: when $i=T$, the lower limit is $M\ge1$. To estimate the last sum, put $k=T-i+1$ and use $e^x\ge1+x$ and $2\rho\Delta=1/(500C)\le1$. Then
\[
    k\varphi_i\ge\frac{k}{12}\left(1+\frac{i-1}{500C}\right) \ge\frac{k+i-1}{6000C}=\frac{T}{6000C}.
\]
Consequently,
\[
    \mu_{\mathbf T'}^+(\tau\text{ is atypical}) \le6T e^{-dT/(2\cdot10^{14}C^4)} \le\tfrac12 e^{-dT/(4\cdot10^{14}C^4)} \le\tfrac12 e^{-dt/(10^{15}C^5)}.
\]
For the middle inequality use $\log(12T)\le dT/(4\cdot10^{14}C^4)$, which follows from the working bound on $d$ for every $T\ge1$; for the last use $T\ge100t$. This proves \eqref{eq:one-tree-goal} and hence the lemma.
\end{proof}

\subsection*{Acknowledgments} ZL was partially supported by the NSFC Major Program (Project No. 12595294) the NSFC Key Program (Project No. 12231002). AS was partially supported by a Simons Investigator grant.

\subsection*{Statement on AI usage}

AI was used in the writing of this paper only for proofreading and copy-editing.

\bibliographystyle{alpha}

\end{document}